\documentclass[a4paper,11pt]{article}
\usepackage{amsmath, amsfonts, amsthm, graphicx, amssymb, url, enumitem,subcaption, cite, mathrsfs, emptypage} 
\usepackage[margin=2.7cm]{geometry}
\usepackage[section]{placeins}
\usepackage{hyperref}
\hypersetup{
    colorlinks=true,
    linkcolor=black,
    urlcolor=black,
    citecolor=black,
    linktoc=all
}
\newtheorem{thm}{Theorem}[section]
\newtheorem{cor}[thm]{Corollary} 
\newtheorem{lem}[thm]{Lemma}
\newtheorem{prop}[thm]{Proposition}
\newtheorem{defn}[thm]{Definition}

\newcommand{\R}{\mathbb R}
\newcommand{\C}{\mathbb C}
\newcommand{\Z}{\mathbb Z}
\newcommand{\N}{\mathbb N}

\newcommand{\p}{\varphi}
\renewcommand{\l}{\lambda}

\newcommand{\til}[1]{\widetilde{#1}}

\title{Circular Nets with Spherical Parameter Lines and Terminating Laplace Sequences}

\author{Alexander I. Bobenko, Alexander Y. Fairley}

\date{Institut f\"ur Mathematik, Technische Universit\"at Berlin,\\
Str. des 17. Juni 136, 10623 Berlin, Germany\\
\today \\
bobenko@math.tu-berlin.de,  fairley@math.tu-berlin.de}

\begin{document}

\maketitle

\begin{abstract}
The focus is on circular nets with one or two families of spherical parameter lines, which are treated in M\"obius geometry. These circular nets provide a discretisation of surfaces with one or two families of spherical curvature lines. The special cases of planar, circular and linear parameter lines are also investigated. A Lie-geometric discretisation in terms of principal contact element nets is also presented.  Its properties are analogous to the classical properties of surfaces with one or two families of spherical curvature lines. Circular nets with two families of spherical parameter lines have geometric properties that are related to Darboux cyclides. Circular nets with one or two families of spherical parameter lines are examples of Q-nets with terminating Laplace sequences. More generally, this article considers Q-nets that are inscribed in quadrics and that have terminating Laplace sequences.

\end{abstract}

\section{Introduction}
\label{section: intro}

In the context of discrete differential geometry, discrete conjugate-line parametrised surfaces are defined as Q-nets, which are maps $P: \Z^2 \to \R \mathrm{P}^n$ such that, $\forall (i,j) \in \Z^2,$  the vertices $P_{i,j},P_{i+1,j}, P_{i+1,j+1}, P_{i,j+1}$ are coplanar \cite{sauer1970differenzen}.  Throughout this article, $P_{i,j}$ denotes the image of $(i,j)$ under the function $P$. 

There are various structure-preserving discretisations of curvature-line parametrisations: circular nets were introduced in \cite{Nutbourne1988}, conical nets were introduced in \cite{liu2006geometry} and principal contact element nets were introduced in \cite{bobenko2007organizing}. Principal contact element nets provide a unification of circular nets and conical nets. In this article, the focus is on circular nets. Circular nets are Q-nets with circular faces. They provide a M\"obius-geometric discretisation of curvature-line parametrisations. Circular nets can be characterised as Q-nets that are inscribed in the M\"obius quadric \cite{DDG}.

A classical topic in differential geometry is the study of surfaces with one or two families of curvature lines that are spherical, planar, circular or linear. The foregoing surfaces were studied mostly in the  nineteenth century in the setting of local differential geometry \cite{blaschke1929,  bonnet1853memoire, darboux1896lecons, eisenhart1909treatise,  Serret1853}. 

In modern times, surfaces with planar and spherical curvature lines have garnered renewed interest in the setting of global differential geometry. In particular, the simplest tori with constant mean curvature have one family of planar curvature lines. Based on this property, they were explicitly described in \cite{Abresch1987, Walter1987}. Later on, these results were extended to constant mean curvature tori with one family of spherical curvature lines \cite{Wente1992}. Surfaces with planar curvature lines were studied in \cite{musso1999laguerre} via Laguerre geometry. Non-canal isothermic tori with spherical curvature lines were constructed  in \cite{Bernstein2001}. All isothermic tori with one family of planar curvature lines were recently classified in \cite{bobenko2021compact} and applied to solve the longstanding Bonnet problem. Further examples of surfaces with spherical curvature lines were investigated in \cite{cho2021constrained}.

Curvature-line parametrisations are a subclass of conjugate-line parametrisations. The latter admit Laplace transformations, which can be iterated. The iterations determine a so-called Laplace sequence. Generically, the Laplace sequence is infinite. However, there are interesting cases where the Laplace sequence terminates after finitely many steps, i.e. where an iteration degenerates to a curve. The foregoing cases were notably studied in \cite{darboux1889lecons, Titeica1924geometrie}.  For surfaces with one family of spherical curvature lines, it is a classical result that the Laplace sequence terminates after finitely many steps \cite{goursat1896equations, darboux1896lecons}.

In this paper, we introduce and study discrete curvature-line parametrisations with spherical parameter lines. Although our main focus is on circular nets in M\"obius geometry, we also consider general Q-nets in projective geometry as well as principal contact element nets in Lie geometry. Up to now, this problem has not been investigated in detail although some examples appeared in the literature. In the context of architectural design, discrete surfaces with planar curvature lines were investigated in \cite{Tellier2019, Pottmannplanarplanar2022}. Discrete channel surfaces, which have circular curvature lines, were investigated in \cite{hertrich2020}.  Discrete surfaces with spherical curvature lines were investigated in \cite{rorig2021ribaucour} in the context of Ribaucour transformations and discrete R-congruences.

Terminating Laplace sequences lie at the core of our investigation of discrete conjugate nets with constrained parameter lines. 
The intersection points of opposite edge-lines of the elementary quadrilaterals of Q-nets also build Q-nets. These define two Laplace transforms $\mathcal{L}_A P : \Z^2 \to \R \mathrm{P}^n$ and  $\mathcal{L}_B P : \Z^2 \to \R \mathrm{P}^n$. (See Figure~\ref{figure: Laplacepoints}.) The Laplace transformations can be iterated \cite{doliwa2000lattice, doliwa1997geometricToda} and they determine a Laplace sequence. By definition, a Laplace sequence terminates if an iteration degenerates to a discrete curve, i.e. a sequence of points. In Section~\ref{section: Laplace}, examples of Q-nets $P: \Z^2 \to \R \mathrm{P}^n$ with terminating Laplace sequences are provided by Q-nets with constrained parameter lines such that, for each $j\in \Z$, the parameter line $\{P_{i,j} \mid i \in \Z\}$ is contained in a $d$-dimensional projective subspace of $\R \mathrm{P}^n$, where $d < n$.

Q-nets that are inscribed in quadrics are well established in discrete differential geometry \cite{DDG, doliwa1999}. In particular, they provide a natural way to treat circular nets, conical nets and principal contact element nets in their corresponding geometries \cite{bobenko2007organizing}. In Section~\ref{section: gridinquadric}, we show that for a Q-net that is inscribed in a quadric, its Laplace sequence terminates in one direction if and only if it terminates in both directions, see Theorems~\ref{thm: GoursatimpliesLaplaceQnetquadric} and \ref{thm: LaplaceimpliesGoursatQnetquadric}.

We denote by $[m]$ the set of integers $\{k\in \N \mid 1 \leq k \leq m\}$.

In this article, we consider Q-nets in $\R \mathrm{P}^n$ where $n\ge 2$, without repeating this every time.

\begin{thm}
\label{thm: mxmiteratedlaplaceconjugate}
Let $P: [m]\times [m] \to \mathbb{R}\mathrm{P}^n$ be a Q-net, where $m \geq 2$. Suppose that $\mathcal{L}_A^d P : [m-d]\times [m-d] \to \mathbb{R}\mathrm{P}^n$ and $\mathcal{L}_B^d P: [m-d]\times [m-d] \to \mathbb{R}\mathrm{P}^n$ are well defined Q-nets for all $d \in [m-1]$. Then, $\mathcal{L}_A^{m-1} P$ and $\mathcal{L}_B^{m-1} P$ are two points. Suppose that the points $\{P_{i,j} \mid (i,j) \in [m]\times [m] \text{ s.t. } (i,j)\neq(m,m)\}$ are contained in a quadric 
$\mathcal{Q} \subset \mathbb{R}\mathrm{P}^n$. Then, the point $P_{m,m}$ is contained in $\mathcal{Q}$ if and only if the points $\mathcal{L}_A^{m-1} P$ and $\mathcal{L}_B^{m-1} P$ are conjugate with respect to $\mathcal{Q}$.
\end{thm}

Theorem~\ref{thm: mxmiteratedlaplaceconjugate} is the main tool that enables us to construct Q-nets $\Z^2 \to \R \mathrm{P}^n$ that are inscribed in quadrics and that have constrained parameter lines.

In Section~\ref{section: Mobiusapproach}, we apply the general results of Section~\ref{section: gridinquadric} to investigate circular nets with one or two families of spherical parameter lines. We also investigate in detail the cases of planar, circular and linear parameter lines. Theorem~\ref{thm: mxmiteratedlaplaceconjugate} is used to prove several remarkable incidence theorems that provide effective methods to construct circular nets with parameter lines that are spherical, planar, circular or linear. For example, Proposition~\ref{thm: circularcircular4x4} is an incidence theorem concerning the configuration in Figure~\ref{figure: 4x4circulargridincidence} which is a finite patch of a circular net with two families of circular parameter lines. We show in Section~\ref{section: Mobiusapproach} that circular nets with two families of spherical parameter lines have algebraic and geometric properties involving quadrics and Darboux cyclides. The properties are compared and contrasted with those of surfaces with spherical curvatures lines. Appendix~\ref{appendix: mobiusgeo} overviews the pertinent properties of Darboux cyclides. 

\begin{figure}[htbp]
\[\includegraphics[width=0.68\textwidth]{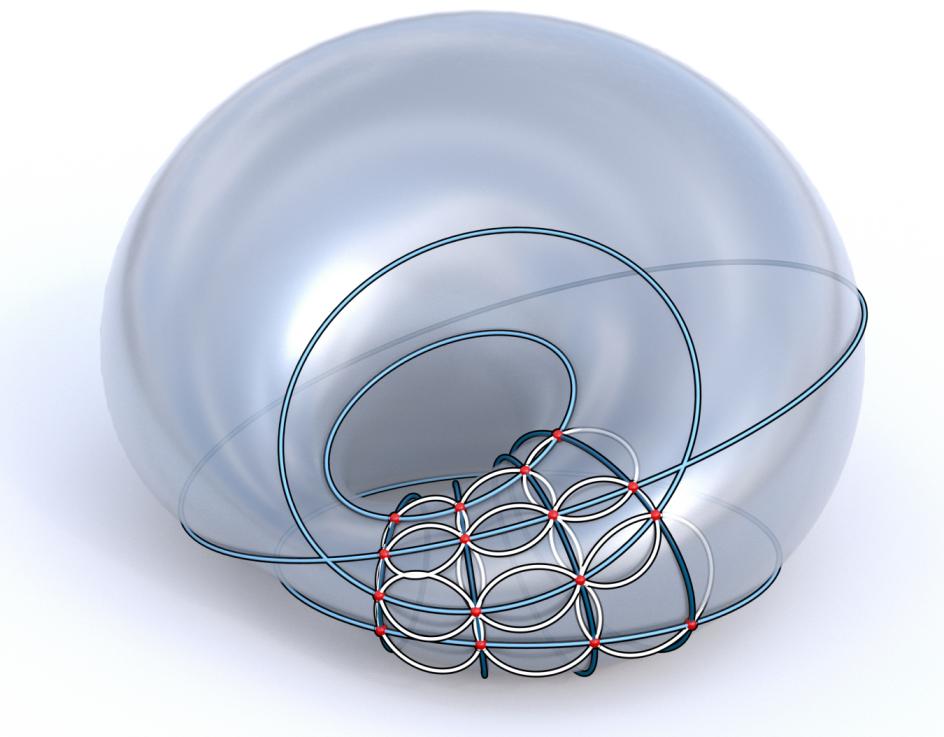}\]
\caption{A circular net with circular parameter lines. Proposition~\ref{thm: circularcircular4x4} is an incidence theorem concerning the configuration. By Proposition~\ref{prop: circularcircularR2R3}, the circular parameter lines are contained in a Darboux cyclide.}
\label{figure: 4x4circulargridincidence}
\end{figure}

Our main motivation is to find a structure-preserving discretisation of surfaces with spherical curvature lines. One can pass to a continuum limit by refining the mesh size. This general picture of discrete differential geometry looks very natural. There is a common belief that the smooth theory can be obtained from the corresponding discrete theory. The corresponding $C^\infty$ convergence of general circular nets to curvature line parametrised surfaces was proven in \cite{bobenko2003discrete}. A method to generate the curvature-line parametrisation of a given surface from circular nets on the surface was developed in \cite{bobenko2018curvature}. However, one should not underestimate possible difficulties and ambiguities concerning connections of the discrete and smooth theories. An example of such a problem is an alternation phenomenon which is discussed in Section~\ref{section: Lieapproach}. We demonstrate that although the natural discretisation via circular nets with spherical parameter lines comprises a nice consistent theory, it is slightly too general to be seen as the perfect discretisation model.

As indicated in \cite{bobenko2007organizing, DDG, pottmann2008focal}, there are three natural discretisations of curvature-line parametrisations: circular nets, conical nets and principal contact element nets. They correspond to M\"obius-, Laguerre- and Lie-geometric descriptions of curvature-line parametrisations. The most general one is the Lie-geometric description, which is also applicable to the case of spherical curvature lines. 

Curvature-line parametrised surfaces with spherical curvature lines belong to Lie geometry. If one applies a Lie transformation to such a surface, one obtains again a curvature-line parametrised surface with spherical curvature lines. The corresponding discretisation is given in Section~\ref{section: Lieapproach} in terms of principal contact element nets.  A Lie-geometric approach is used to exhibit the aforementioned alternation phenomenon. Once the alternation phenomenon is resolved, the discrete theory fits perfectly to the smooth theory and clarifies the classical properties of surfaces with spherical curvature lines. 

\textbf{Genericity assumption:} In the spirit of local differential geometry, this article assumes implicitly in all statements and proofs that {\em all the data are generic}. This means that the geometric structure of the nets under consideration is invariant under all small perturbations that preserve the constraints of the data.  For example, Q-nets are maps $P: \Z^2 \to \R \mathrm{P}^n$, with the constraint that the points $P_{i,j}, P_{i+1,j}, P_{i+1, j+1}, P_{i, j+1}$ are coplanar. The genericity assumption implies that the points $P_{i,j}, P_{i+1,j}, P_{i+1, j+1}, P_{i, j+1}$ are distinct and that no three of them are collinear. The genericity assumption also implies that the lines $P_{i,j} \vee P_{i,j+1}$, $P_{i+1,j} \vee P_{i+1,j+1}$, and $P_{i+2,j} \vee P_{i+2,j+1}$ are not concurrent. More generally, the genericity assumption ensures that generic Q-nets have well-defined Laplace sequences. 

\section{Terminating Laplace Sequences}\label{section: Laplace}

Let $P:\Z^2 \to \R \mathrm{P}^n$ be a Q-net. This means that $\forall (i,j) \in \Z^2$, the points $P_{i,j}, P_{i+1,j},P_{i+1,j+i}, P_{i,j+1}$ are coplanar. The intersection points
\begin{align*}
A_{i,j} := (P_{i,j}\vee P_{i+1,j})\cap (P_{i,j+1}\vee P_{i+1,j+1})\\
B_{i,j}:= (P_{i,j} \vee P_{i,j+1})\cap (P_{i+1,j} \vee P_{i+1,j+1})
\end{align*}
are known as the {\em Laplace points} of the planar quad $\Box (P_{i,j}, P_{i+1,j}, P_{i+1,j+1}, P_{i,j+1})$. 

The {\em Laplace transforms} $\mathcal{L}_A P : \Z^2 \to \R \mathrm{P}^n$ and $\mathcal{L}_B P : \Z^2 \to \R \mathrm{P}^n$ of  $P$ are defined such that $\mathcal{L}_A P(i,j) := A_{i,j}$ and $\mathcal{L}_BP (i,j) := B_{i,j}$. The {\em Laplace transformations} are $\mathcal{L}_A$ and $\mathcal{L}_B$ such that $P \mapsto \mathcal{L}_A P$ and $P\mapsto \mathcal{L}_B P$, respectively. It is straightforward to check that the Laplace transforms $\mathcal{L}_A P$ and $\mathcal{L}_B P$ are Q-nets. Thus, Laplace transformations can be iterated. Let $\mathcal{L}_A^d$ and $\mathcal{L}_B^d$  denote the $d$-times iterated Laplace transformations. An example of Laplace transforms is shown in Figure~\ref{figure: Laplacepoints}.

\begin{figure}[htbp] 
\[\includegraphics[width=0.8\textwidth]{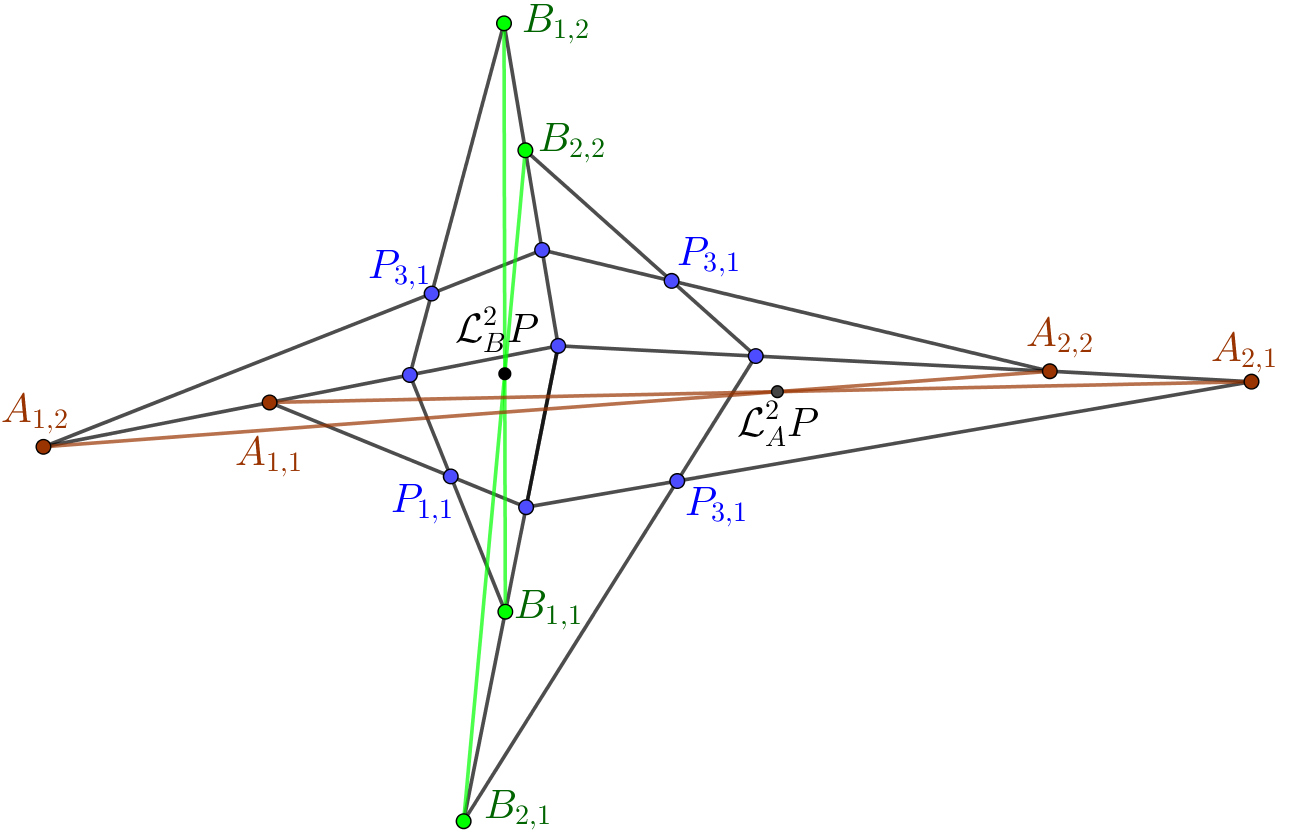}\]
\caption{A patch of Q-net $P$ and its Laplace transforms $\mathcal{L}_A P$ and $\mathcal{L}_B P$. The Laplace transforms $\mathcal{L}^2_A P$ and $\mathcal{L}^2_B P$ are also shown.} 
\label{figure: Laplacepoints}
\end{figure}

It is straightforward to check that the Laplace transformations $\mathcal{L}_A$ and $\mathcal{L}_B$ are mutually inverse up to an index shift:
$$
\mathcal{L}_A \circ \mathcal{L}_B P (i,j) = \mathcal{L}_B \circ \mathcal{L}_A P (i,j) = P(i+1,j+1).
$$ 
Thus, the iterated Laplace transforms can be arranged into the following {\em Laplace sequence}.
\begin{align*}
\ldots  \leftarrow \mathcal{L}_A^3P \leftarrow  \mathcal{L}_A^2P  \leftarrow \mathcal{L}_AP  \leftarrow P \rightarrow \mathcal{L}_BP \rightarrow \mathcal{L}_B^2P \rightarrow \mathcal{L}_B^3P \rightarrow \ldots
\end{align*}
Generically, the Laplace sequence is infinite in both directions. However, there are interesting cases where the sequence terminates in one or both directions.

\begin{defn}\label{defn: GoursatandLaplace}
Let $P: \Z^2 \to \mathbb{R}\mathrm{P}^n$ be a Q-net and $m\in \N$. We call
\begin{itemize}
\item $\mathcal{L}_A^m P$  \emph{Goursat degenerate} if $\mathcal{L}_A^m P (i,j)$ is independent of $i$. 
\item $\mathcal{L}_A^m P$  \emph{Laplace degenerate} if $\mathcal{L}_A^m P (i,j)$  is independent of $j$.
\item $\mathcal{L}_B^m P$  \emph{Goursat degenerate} if $\mathcal{L}_B^m P (i,j)$  is independent of $j$.
\item $\mathcal{L}_B^m P$  \emph{Laplace degenerate} if $\mathcal{L}_B^m P (i,j)$  is independent of $i$.
\end{itemize}
\end{defn}

The case of Laplace degenerate $\mathcal{L}_AP$ or $\mathcal{L}_BP$ was studied in \cite{Kilian2023}.

Definition~\ref{defn: GoursatandLaplace} has a counterpart in the smooth theory. The attributions to Goursat and Laplace are compatible with the attributions that are made in \cite{bol1950projektive, lane1932projective}.

If $\mathcal{L}_A^mP$ is either Goursat degenerate or Laplace degenerate, then $\mathcal{L}_A^m P$ degenerates to a discrete curve $\Z \to \R\mathrm{P}^n$, and the Laplace sequence terminates after $m$ steps. If $\mathcal{L}_A^mP$ is both Goursat and Laplace degenerate, then $\mathcal{L}_A^m P$ degenerates to a single point. 

\begin{prop}\label{prop: Goursatterminatemsteps}
Let $P: \Z^2 \to \R \mathrm{P}^n$ be a Q-net,  and $m<n$. Generically,  $\mathcal{L}_A^mP$ is Goursat degenerate if and only if the parameter lines $\{P_{i,j} \}_{i\in \Z}$ are contained in $m$-dimensional projective subspaces of $\R \mathrm{P}^n$.
\end{prop}

\begin{proof}
Due to the genericity assumption, the claim follows by counting dimensions of projective subspaces. Let $H_j$ be the $m$-dimensional space containing the points $\{P_{i,j}\}_{i \in \Z}$. The points $\{\mathcal{L}_AP(i,j)\}_{i \in \Z}$ lie in $H_j\cap H_{j+1}$. Generically, $H_j \cap H_{j+1}$ is $(m-1)$-dimensional because $P$ is a Q-net. By the same reasoning, for each $j\in \Z$ and $d \le m$, the points $\{\mathcal{L}_A^d (i,j)\}_{i \in \Z}$ are contained in a $(m-d)$-dimensional space. In particular, the points $\{\mathcal{L}_A^m (i,j)\}_{i \in \Z}$ coincide. 
\end{proof}

\begin{prop}\label{prop: Laplaceterminatemsteps}
Let $P: \Z^2 \to \R \mathrm{P}^n$ be a Q-net, and $m <n$. Generically,  $\mathcal{L}_A^mP$ is  Laplace degenerate if and only if
$\cap_{j\in \Z}P_{i,j} \vee P_{i+1,j} \vee \ldots \vee P_{i+m, j}$ is a point for each $i$. 
\end{prop}

\begin{proof}
For all $(i,j) \in \Z^2$, $P_{i,j} \vee P_{i+1,j} \vee \ldots \vee P_{i+m, j}$ is generically $m$-dimensional. Because $P$ is a Q-net, the intersection $(P_{i,j} \vee P_{i+1,j} \vee \ldots \vee P_{i+m, j})\cap (P_{i,j+1} \vee P_{i+1,j+1} \vee \ldots \vee P_{i+m, j+1})$ is $(m-1)$-dimensional, and $\cap_{l\in \{0, 1,\ldots,m\}}P_{i,j+l} \vee P_{i+1,j+l} \vee \ldots \vee P_{i+m, j+l}$ is a point. The point coincides with $\mathcal{L}_A^mP(i,j)$ since the latter is contained in $\cap_{l\in \{0, 1,\ldots,m\}}P_{i,j+l} \vee P_{i+1,j+l} \vee \ldots \vee P_{i+m, j+l}$. So, $\mathcal{L}_A^m P(i,j)$ is independent of $j$ if and only if $\cap_{j\in \Z}P_{i,j} \vee P_{i+1,j} \vee \ldots \vee P_{i+m, j}$ is a point for each $i$.
\end{proof}

Propositions~\ref{prop: Goursatterminatemsteps} and \ref{prop: Laplaceterminatemsteps} have counterparts in the smooth theory \cite{lane1942treatise}. For instance, the case $m=2$ of Proposition~\ref{prop: Laplaceterminatemsteps} can be interpreted in terms of osculating planes.


Recall, $l: \Z^2 \to \{\text{lines in } \R \mathrm{P}^n\}$ is a discrete line congruence if and only if any two neighbouring lines intersect.

\begin{defn}
\label{def:laplace_line congruence}
Let $l: \Z^2 \to \{\text{lines in } \R \mathrm{P}^n\}$ be a discrete line congruence, where $n \ge 3$. Define the \emph{Laplace transforms} $\mathcal{L}_Al:\Z^2 \to \{ \text{lines in } \R \mathrm{P}^n\}$ and $\mathcal{L}_Bl:\Z^2 \to \{ \text{lines in } \R \mathrm{P}^n\}$ as
\begin{align*}
\mathcal{L}_A l (i,j) &:= (l_{i,j} \vee l_{i+1,j}) \cap (l_{i,j+1} \vee l_{i+1,j+1}), \\
\mathcal{L}_B l (i,j) &:= (l_{i,j} \vee l_{i,j+1}) \cap (l_{i+1,j} \vee l_{i+1,j+1}).
\end{align*} 
\end{defn}

It is straightforward to check that the Laplace transforms $\mathcal{L}_A l$ and $\mathcal{L}_B l$ are generically discrete line congruences. Consequently, the Laplace transformations  $\mathcal{L}_A$ and $\mathcal{L}_B$ of discrete line congruences can be iterated. Let $\mathcal{L}_A^d$ and $\mathcal{L}_B^d$  denote the $d$-times iterated Laplace transformations, where $d \in \N$.

Analogously to the Laplace sequences of Q-nets, discrete line congruences also determine Laplace sequences. Goursat degenerations and Laplace degenerations are defined for discrete line congruences analogously to Definition~\ref{defn: GoursatandLaplace}.
\section{Q-Nets that are Inscribed in Quadrics}\label{section: gridinquadric}

Lemmas~\ref{lem: twoquads} and \ref{lem: planarquadpolarlaplace} are used in the ensuing proof of Theorem~\ref{thm: mxmiteratedlaplaceconjugate} which describes the structure of Q-nets that are inscribed in quadrics.

\begin{lem}\label{lem: twoquads}
Let $O:= \Box(O_{1,1}, O_{2,1}, O_{2,2}, O_{1,2})$ and $P:= \Box(P_{1,1}, P_{2,1}, P_{2,2}, P_{1,2})$ be two planar quads in $\mathbb{R}\mathrm{P}^n$. Let $\mathcal{L}_AO$, $\mathcal{L}_BO$, $\mathcal{L}_AP$ and $\mathcal{L}_BP$ be their Laplace points. (See Figure~\ref{figure: lemmaAB}.) Any five of the following six conjugacy conditions relative to a quadric $\mathcal{Q}$ in $\mathbb{R}\mathrm{P}^n$ imply the sixth one.
\begin{gather*}
 O_{1,2} \perp P_{1,2}, \  O_{1,1} \perp P_{1,1}, \ O_{2,1} \perp P_{2,1},
O_{2,2} \perp P_{2,2}, \ 
\mathcal{L}_{B}O \perp \mathcal{L}_AP,\  \mathcal{L}_{A}O \perp \mathcal{L}_{B}P.
\end{gather*}
\end{lem}

\begin{figure}[htbp]
\begin{center}
\includegraphics[width =0.7\textwidth]{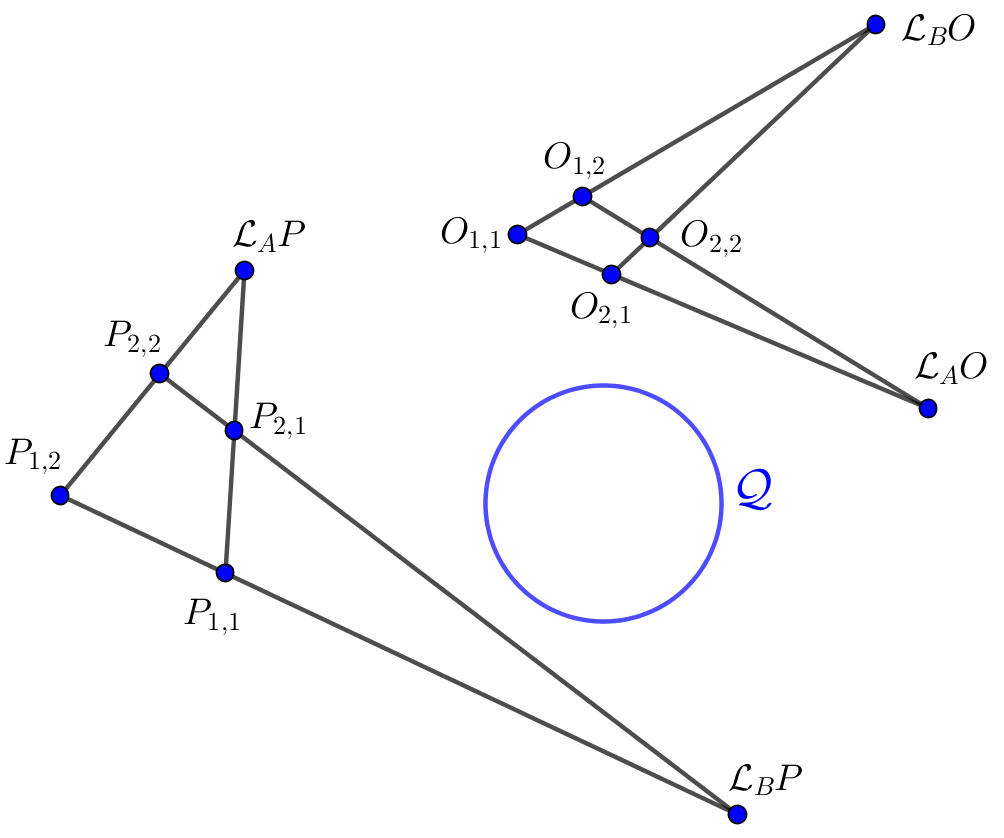}
\end{center}
\caption{Two planar quads  in $\R\mathrm{P}^2$ that satisfy the six conjugacy conditions in Lemma~\ref{lem: twoquads}.} 
\label{figure: lemmaAB}
\end{figure}

\begin{proof}
Denote by $[x] \in \R \mathrm{P}^n$ a point with representative vector $x \in \R^{n+1}$. Let $O_{1,2}= [u]$, $O_{1,1} = [v]$, $O_{2,1} = [w]$. Choose the representative vectors $u, v$ and $w$ so that $O_{2,2}= [u+v+w]$. Then, $\mathcal{L}_{A}O = [v+w]$ and $\mathcal{L}_BO = [u+v]$. Analogously, $P_{1,2}= [x]$, $P_{1,1} = [y]$, $P_{2,1} = [z]$, $P_{2,2}= [x+y+z]$. Then, $\mathcal{L}_AP = [y+z]$ and $\mathcal{L}_BP = [x+y]$.   Let $\p$ be a symmetric bilinear form representing the quadric $\mathcal{Q}$. The six conditions in Lemma~\ref{lem: twoquads} are equivalent to the following equations.
\begin{align*}
& \p(u,x)=0,\ \p(v,y)=0,\ \p(w,z)=0,  \\
&   \p(u+v, y+z)=0,\ \p(v+w,  x+y )=0,\  \p(u+v+w,x+y+z)=0.
\end{align*}
Any five of these equations imply the sixth one.
\end{proof}

\begin{lem}\label{lem: planarquadpolarlaplace}
Let three of the vertices of a planar quad $\Box(P_{1,1},P_{2,1},P_{2,2},P_{1,2})\subset \mathbb{R}\mathrm{P}^n$ be contained in a quadric $\mathcal{Q}$. Then, the fourth vertex also lies in the quadric $Q$ if and only if the Laplace points $A_{1,1}:= (P_{1,1} \vee P_{2,1})\cap (P_{1,2} \vee P_{2,2})$ and $B_{1,1}:= (P_{1,1} \vee P_{1,2})\cap (P_{2,1} \vee P_{2,2})$ are conjugate relative to $\mathcal{Q}$.
\end{lem}

\begin{proof}
Analogously to the proof of Lemma~\ref{lem: twoquads}, the claim follows easily from the normalisation $P_{1,2} = [u]$, $P_{1,1} = [v]$, $P_{2,1}= [w]$, $P_{2,2}= [u+v+w]$.
\end{proof}

\begin{proof}[Proof of Theorem~\ref{thm: mxmiteratedlaplaceconjugate}]
The case $m=2$ of Theorem~\ref{thm: mxmiteratedlaplaceconjugate} is Lemma~\ref{lem: planarquadpolarlaplace}. Let $m=3$. Then, $\mathcal{L}_A P: [2]\times [2] \to \R \mathrm{P}^n$ and $\mathcal{L}_B P: [2]\times [2] \to \R \mathrm{P}^n$ are the two planar quads $\Box(A_{1,1}, A_{2,1}, A_{2,2}, A_{1,2})$ and $\Box(B_{1,1}, B_{2,1}, B_{2,2}, B_{1,2})$. Suppose that the points $\{P_{i,j} \mid (i,j) \in [3]\times [3] \text{ s.t.\ } (i,j)\neq(3,3)\}$ are contained in a quadric $\mathcal{Q}$. By Lemma~\ref{lem: planarquadpolarlaplace}, the following conjugacy conditions are satisfied: 
\begin{align}\label{eq: conj}
A_{1,1} \perp B_{1,1}, \qquad A_{2,1} \perp B_{2,1}, \qquad
A_{1,2} \perp B_{1,2}.
\end{align}
Consider the following iterated Laplace points:
\begin{align*}
\mathcal{L}_A \circ \mathcal{L}_AP := (A_{11}\vee A_{21}) \cap (A_{12} \vee A_{22}),\\
\mathcal{L}_B \circ \mathcal{L}_AP := (A_{11}\vee A_{12}) \cap (A_{21} \vee A_{22}),\\
\mathcal{L}_B \circ \mathcal{L}_BP := (B_{11}\vee B_{12}) \cap (B_{21}\vee B_{22}),\\
\mathcal{L}_A \circ \mathcal{L}_BP := (B_{11}\vee  B_{21}) \cap (B_{12}\vee B_{22}).
\end{align*}
The points $\mathcal{L}_A \circ \mathcal{L}_BP$ and $\mathcal{L}_B \circ \mathcal{L}_AP$ coincide and equal $P_{2,2}$. So, $\mathcal{L}_A \circ \mathcal{L}_BP$ and $\mathcal{L}_B \circ \mathcal{L}_AP$ are conjugate relative to $\mathcal{Q}$ because $P_{2,2} \in \mathcal{Q}$. Because of the foregoing conjugacy and the three conjugacies in (\ref{eq: conj}), Lemma~\ref{lem: twoquads} ensures that $\mathcal{L}^2_AP \perp \mathcal{L}_B^2P$ if and only if $A_{2,2} \perp B_{2,2}$. By Lemma~\ref{lem: planarquadpolarlaplace}, $A_{2,2}\perp B_{2,2}$ if and only if $P_{3,3}\in \mathcal{Q}$. This proves the case $m=3$ of Theorem~\ref{thm: mxmiteratedlaplaceconjugate}. By the principle of mathematical induction, it suffices to prove Theorem~\ref{thm: mxmiteratedlaplaceconjugate} for $m \geq 4$ under the induction hypothesis that Theorem~\ref{thm: mxmiteratedlaplaceconjugate} holds for $m-1$ and $m-2$.

Let $m\geq 4$. Consider the Laplace transforms $\mathcal{L}^{m-2}_A P, \mathcal{L}^{m-2}_B P: [2]\times [2] \to \R \mathrm{P}^n$. They are two planar quads. The case $m-1$ of Theorem~\ref{thm: mxmiteratedlaplaceconjugate} implies the following conjugacy conditions:

\begin{gather}
\mathcal{L}^{m-2}_A P(1,1) \perp \mathcal{L}^{m-2}_B P(1,1)\label{eq: conj1}\\
\mathcal{L}^{m-2}_A P(2,1) \perp \mathcal{L}^{m-2}_B P(2,1)\label{eq: conj2}\\
\mathcal{L}^{m-2}_A P(1,2) \perp \mathcal{L}^{m-2}_B P(1,2\label{eq: conj3})
\end{gather}
 For instance, applying the case $m-1$ of Theorem~\ref{thm: mxmiteratedlaplaceconjugate} to the points $\{P_{i,j}\mid (i,j)\in [m]\times [m], i\neq m, j\neq 1\}$ implies the conjugacy condition $(\ref{eq: conj3})$.

The point $\mathcal{L}_B\circ \mathcal{L}^{m-2}_AP$ coincides with $\mathcal{L}_A^{m-3}P(2,2)$. The point $\mathcal{L}_A\circ \mathcal{L}^{m-2}_BP$ coincides with $\mathcal{L}_B^{m-3} P(2,2)$. By applying the case $m-2$ of Theorem~\ref{thm: mxmiteratedlaplaceconjugate} to the points $\{P_{i,j}\mid (i,j)\in [m]\times [m], i\neq 1, i\neq m, j\neq 1,j\neq m\}$, we obtain $\mathcal{L}_A^{m-3}P(2,2) \perp \mathcal{L}_B^{m-3}P(2,2)$. Equivalently,  $\mathcal{L}_B\circ \mathcal{L}^{m-2}_AP \perp \mathcal{L}_A\circ \mathcal{L}^{m-2}_BP$. Because of the foregoing conjugacy and the conjugacies (\ref{eq: conj1}), (\ref{eq: conj2}) and (\ref{eq: conj3}), Lemma~\ref{lem: twoquads} ensures that $\mathcal{L}^{m-2}_A P(2,2) \perp \mathcal{L}^{m-2}_B P(2,2)$ if and only if $\mathcal{L}^{m-1}_AP \perp \mathcal{L}_B^{m-1}P$. By applying the case $m-1$ of Theorem~\ref{thm: mxmiteratedlaplaceconjugate} to the points $\{P_{i,j}\mid (i,j)\in [m]\times [m], i\neq 1, j\neq 1\}$, the point $P_{m,m}$ is contained in $\mathcal{Q}$ if and only if $\mathcal{L}^{m-2}_A P(2,2) \perp \mathcal{L}^{m-2}_B P(2,2)$. Therefore, $P_{m,m}$ is contained in $\mathcal{Q}$ if and only if $\mathcal{L}_A^{m-1}P \perp \mathcal{L}_B^{m-1}P$.
\end{proof}

\begin{cor}\label{cor: mxmgridthm}
Let $P: [m]\times [m] \to \mathbb{R}\mathrm{P}^n$ be a Q-net, where $n \ge m \geq 1$.  For each $i\in [m]$, let $V_i:= \mathrm{join}\{P_{i,j}\}_{j \in [m]}$. For each $j\in [m]$, let $H_j:= \mathrm{join}\{P_{i,j}\}_{i \in [m]}$. Suppose that the points $\{P_{i,j} \mid (i,j)\in [m] \times [m] \text{ s.t.\ } (i,j)\neq (m,m)\}$ are contained in a non-degenerate quadric $\mathcal{Q}$. Generically, the spaces $X:= \cap_{j\in [m] }H_j$ and $Y:= \cap_{i\in [m]}V_i$ are points. The point $P_{m,m}$ is contained in $\mathcal{Q}$ if and only if the points $X$ and $Y$ are conjugate with respect to $\mathcal{Q}$.
\end{cor}

\begin{proof}
A point $P$ is contained in the quadric $\mathcal{Q}$ if and only if $P$ is self-conjugate relative to $\mathcal{Q}$. This proves the case $m=1$. The case $m=2$ is Lemma~\ref{lem: planarquadpolarlaplace}. For the remainder of this proof, let $m\geq 3$. The genericity assumption ensures that $\mathcal{L}_A^d P, \mathcal{L}_B^d P : [m-d]\times [m-d] \to \mathbb{R}\mathrm{P}^n$ are well defined Q-nets for all $d \le m-1$. Then, $\mathcal{L}^{m-1}_AP$ and $\mathcal{L}^{m-1}_BP$ are points such that $\mathcal{L}^{m-1}_AP \in X:= \cap_{j\in [m]}H_j$ and $\mathcal{L}^{m-1}_BP \in Y:= \cap_{i\in [m]}V_i$. Generically, the spaces $\{V_i\}_{i\in [m]}$ and $\{H_j\}_{j\in [m]}$ are $(m-1)$-dimensional. $V_i \vee V_{i+1}$ is $m$-dimensional because $P$ is a Q-net and $V_i \neq V_{i+1}$ since $\mathcal{Q}$ is non-degenerate. Equivalently, $V_i \cap V_{i+1}$ is $(m-2)$-dimensional. It follows that $Y:= \cap_{i\in [m]}V_i$ is $0$-dimensional. Similarly, $X:= \cap_{j\in [m]}H_j$ is $0$-dimensional. Therefore, $X$ equals $\mathcal{L}^{m-1}_AP$ and $Y$ equals $\mathcal{L}^{m-1}_BP$. Suppose that the points $\{P_{i,j}: (i,j)\in [m] \times [m], (i,j)\neq (m,m)\}$ are contained in a quadric $\mathcal{Q}$. By Theorem~\ref{thm: mxmiteratedlaplaceconjugate}, $P_{m,m} \in \mathcal{Q}$ if and only if $\mathcal{L}^{m-1}_AP \perp \mathcal{L}^{m-1}_BP$. Equivalently, $P_{m,m} \in \mathcal{Q}$ if and only if $X$ and $Y$ are conjugate relative to $\mathcal{Q}$.
\end{proof}

\begin{thm}\label{thm: GoursatimpliesLaplaceQnetquadric}
Let $P: \Z^2 \to \mathcal{Q} \subset \R \mathrm{P}^n$ be a Q-net that is inscribed in a quadric, and all lines $P_{i,j} \vee P_{i+1,j}$ are not isotropic.
Suppose that $\mathcal{L}_A^mP$ is Goursat degenerate for some $m$, and $\mathcal{L}_B^d P$ is well defined for all $d \le m$. Then $\mathcal{L}_B^mP$ is Laplace degenerate. 
\end{thm}

\begin{proof}
Suppose that $\mathcal{L}_B^d P$ is well defined for all $d \le m$ and $\mathcal{L}_A^mP$ is Goursat degenerate. Then, for each $j\in \Z$, the point $\mathcal{L}_A^mP(j) := \mathcal{L}_A^mP(i,j)$ is well defined because it is independent of $i$. By Theorem~\ref{thm: mxmiteratedlaplaceconjugate}, for each $j\in \Z$, $\mathcal{L}_A^mP(j)$ is conjugate to the points $\{\mathcal{L}_B^m P(i,j)\}_{i\in \Z}$.  We will prove that this is in fact a unique point by showing that $\mathcal{L}_B^mP(i,j)=\mathcal{L}_B^mP(i+1,j)$. Indeed, suppose that $\mathcal{L}_B^mP(i,j) \neq\mathcal{L}_B^mP(i+1,j)$. Then, the line $\mathcal{L}_B^mP(i,j) \vee \mathcal{L}_B^mP(i+1,j)$ is conjugate to $\mathcal{L}_A^mP(j)$. Notice that the line $\mathcal{L}_B^mP(i,j) \vee \mathcal{L}_B^mP(i+1,j)$ coincides with the line $\mathcal{L}_B^{m-1}P(i+1,j)\vee \mathcal{L}_B^{m-1}P(i+1,j+1)$. If $P_{i+m+1, j+m}$ is replaced by any point in the line $P_{i+m,j+m} \vee P_{i+m+1,j+m}$, then the Laplace point $\mathcal{L}_A^mP(i,j)$ does not change, whereas the Laplace point $\mathcal{L}_B^mP(i+1,j)$ does change but it remains in the line $\mathcal{L}_B^{m-1}P(i+1,j)\vee \mathcal{L}_B^{m-1}P(i+1,j+1)$ which is conjugate to $\mathcal{L}_A^m P(j)$. Therefore, by Theorem~\ref{thm: mxmiteratedlaplaceconjugate}, any point in the line $P_{i+m,j+m} \vee P_{i+m+1,j+m}$ is contained in $\mathcal{Q}$. This contradicts the assumption that $P_{i+m,j+m} \vee P_{i+m+1,j+m}$ is not an isotropic line. 
\end{proof}

The assumption of Theorem~\ref{thm: GoursatimpliesLaplaceQnetquadric} that the lines $P_{i,j} \vee P_{i+1,j}$ are not isotropic is essential. It is reasonable to assume that $\forall (i,j)\in \Z^2$, $P_{i,j} \vee P_{i+1,j}$ is not an isotropic line of $\mathcal{Q}$. For example, suppose that $P: \Z^2 \to \R \mathrm{P}^3$ is a Q-net inscribed in a non-degenerate quadric $\mathcal{Q}$ of signature $(++--)$ and $\mathcal{L}_AP$ is Goursat degenerate. Then, for each $j\in \Z$, the points $\{P_{i,j}\}_{i\in \Z}$ are contained in an isotropic line of $\mathcal{Q}$. However, $\mathcal{L}_BP$ is generically non-degenerate.

\begin{thm}\label{thm: LaplaceimpliesGoursatQnetquadric}
Let $P: \Z^2 \to \mathcal{Q} \subset \R \mathrm{P}^n$ be a Q-net that is inscribed in a non-degenerate quadric.
Suppose that $\mathcal{L}_A^mP$ is Laplace degenerate for some $m$, and $\mathcal{L}_B^d P$ is well defined for all $d \le m+n-1$. Then $\mathcal{L}_B^{m+n-1}P$ is Goursat degenerate. 
\end{thm}

\begin{proof}
Suppose that $\mathcal{L}_B^d P$ is well defined for all $d \le m+n-1$. Let $\mathcal{L}^m_AP$ be Laplace degenerate. Then, for each $i\in \Z$, the point $\mathcal{L}_A^mP(i) := \mathcal{L}_A^mP(i,j)$ is well defined, i.e. it is independent of $j$.  By Theorem~\ref{thm: mxmiteratedlaplaceconjugate}, $\mathcal{L}_A^mP(i)$ is conjugate to the points $\{\mathcal{L}_B^m P(i,j)\}_{j\in \Z}$. Thus, these points are contained in the $(n-1)$-dimensional polar hyperplane $\mathcal{L}_A^mP(i)^\perp$. Generically, $\mathrm{join}\{\mathcal{L}_A^mP(i), \ldots, \mathcal{L}_A^mP(i+n-1)\}$ is $(n-1)$-dimensional. Equivalently, the intersection $\mathcal{L}_A^mP(i)^\perp \cap \ldots \cap \mathcal{L}_A^mP(i+n-1)^\perp$ is $0$-dimensional. This implies, $\mathcal{L}_B^{m+n-1}P(i,j)=\mathcal{L}_B^{m+n-1}P(i,j+1)$ because they are both contained in $\mathcal{L}_A^mP(i)^\perp \cap \ldots \cap \mathcal{L}_A^mP(i+n-1)^\perp$. Thus, the points $\{\mathcal{L}_B^{m+n-1} P(i,j)\}$ are independent of $j$. 
\end{proof}


Theorems~\ref{thm: GoursatimpliesLaplaceQnetquadric} and \ref{thm: LaplaceimpliesGoursatQnetquadric} have analogues in the smooth theory. A theorem of Goursat for conjugate-line parametrisations of non-degenerate quadric surfaces says that if the Laplace sequence terminates in one direction, then it terminates in both directions \cite{goursat1896equations}.

Theorem~\ref{thm: mxmincidencethm} can be used to construct Q-nets $P: \Z^2 \to \mathbb{R}\mathrm{P}^n$ that are inscribed in non-degenerate quadrics and such that the parameter lines are contained in $d$-dimensional projective subspaces of $\mathbb{R}\mathrm{P}^n$, where $d <  n$.  Let us describe initial data that determine such nets.

\begin{thm}\label{thm: mxmincidencethm}
Let $P: [d+1]\times [d+1] \to \mathcal{Q}\subset \mathbb{R}\mathrm{P}^n$ be a Q-net that is inscribed in a non-degenerate quadric $Q$, and $d<n$. Denote by 
$V_i:= \mathrm{join} \{P_{i,j}\}_{j \in [d+1]}$, $H_j:= \mathrm{join} \{P_{i,j}\}_{i \in [d+1]}$ the corresponding $d$-dimensional projective subspaces. Then, this finite Q-net together with additional points $P_{1,j}\in \mathcal{Q}\cap V_1, P_{i,1}\in \mathcal{Q}\cap H_1$, $i,j\in\Z$ can be uniquely extended to a Q-net $P: \Z^2 \to \mathcal{Q}$ that is inscribed in the quadric $\mathcal{Q}$ such that all the parameter lines are contained in $d$-dimensional projective subspaces $V_i$ and $H_j$. Moreover, both $\cap_{i \in \Z}V_i$ and $\cap_{j \in \Z}H_j$ are points. Equivalently, $\mathcal{L}_A^{d}P$  and $\mathcal{L}_B^{d}P$ are both Laplace degenerate and also Goursat degenerate. 
\end{thm}

%

\begin{proof}
We show how the net $P: [d+1]\times [d+1] \to \mathcal{Q}$ with arbitrarily chosen $P_{1,d+2}\in \mathcal{Q}\cap V_1$, $P_{d+2,1}\in \mathcal{Q}\cap H_1$ can be uniquely extended to a Q-net $P: [d+2]\times [d+2] \to \mathcal{Q}$ with $d$-dimensional parameter lines.  By Corollary~\ref{cor: mxmgridthm},  $X:=\cap_{j\in [d+1]} H_j$,  $Y:=\cap_{i\in [d+1]} V_i$ are two points that are conjugate with respect to $\mathcal{Q}$. For $2\le i\le d$, the line $V_i\cap {\rm join}\{ P_{i-1,d+1},  P_{i-1,d+2}, P_{i,d+1}\}$ intersects $\mathcal{Q}$ in the point $P_{i,d+1}$ and one additional point, which determines $P_{i,d+2}$. The points $P_{d+2,j}$, $2\le j\le d$, are similarly determined. Define $V_{d+2}:={\rm join}\{ Y,  P_{d+2,j}\mid j\in [d]\}$, $H_{d+2}:={\rm join}\{ X,   P_{i,d+2}\mid i \in [d]\} $. By Theorem~\ref{thm: mxmiteratedlaplaceconjugate}, the point $P_{d+1,d+2}:=H_{d+2}\cap V_{d+1}\cap {\rm join}\{ P_{d,d+1},  P_{d+1,d+1}, P_{d,d+2}\}$ is contained in $\mathcal{Q}$. Similarly, $P_{d+2,d+1} := V_{d+2} \cap H_{d+1} \cap {\rm join}\{P_{d+1,d}, P_{d+1,d+1}, P_{d+2, d}\}\in \mathcal{Q}$. Applying Theorem~\ref{thm: mxmiteratedlaplaceconjugate} to the net $\{ P_{i,j} \mid 2\le  i,j\le d+2\}$, we obtain $P_{d+2,d+2}\in \mathcal{Q}$. Note, $X$ equals $\cap_{j\in [d+2]} H_j$ and $Y$ equals $\cap_{i\in [d+2]} V_i$. Then, by Proposition~\ref{prop: Laplaceterminatemsteps}, $\mathcal{L}_A^{d}P$  and $\mathcal{L}_B^{d}P$ are both Laplace degenerate and also Goursat degenerate.
\end{proof}

An isotropic subspace $G$ of maximal dimension in a quadric $\mathcal{Q}\subset \R \mathrm{P}^n$  is called a \emph{generator} of  $\mathcal{Q}$. If $\mathcal{Q}$ is a quadric of signature $(p,q,r)$, then any generator is $({\rm min}\{p,q\} -1 +r)$-dimensional. The dimension of the projective subspace consisting of the singular points of the quadric $\mathcal{Q}$ is $r-1$. If $\mathcal{Q}$ is non-degenerate, then $r=0$ and the signature is $(p,q)$.


\begin{prop}~\label{prop: hyperbolicquadrictwoclasses}
Let $\mathcal{Q}\subset \R \mathrm{P}^{2m-1}$ be a non-degenerate quadric  of neutral signature $(m,m)$. By definition, two generators $G_1$ and $G_2$ (of dimension $m-1$) of $\mathcal{Q}$ belong to the same \emph{system of generators} if the dimension of $G_1 \cap G_2$ has different parity than the parity of $m$. Then, there are exactly two systems of generators.
\end{prop}

A proof of Proposition~\ref{prop: hyperbolicquadrictwoclasses} can be found in \cite{odehnal2020universe}. 

One can naturally adapt the foregoing definition of systems of generators to quadrics $\mathcal{V}\subset \R \mathrm{P}^{2m}$ of signature $(m,m,1)$. Let $E$ be a hyperplane not containing the apex $A$ of $\mathcal{V}$. Then, $E\cap\mathcal{V}$ has signature $(m,m)$ and Proposition~\ref{prop: hyperbolicquadrictwoclasses} implies that it has two systems of generators of dimension $m-1$. Extending them by the apex $A$, we obtain two systems of generators of $\mathcal{V}$ of dimension $m$.



Consider a Q-net $P: \Z^2 \to \mathbb{R}\mathrm{P}^{n}$ with the property that each parameter line of $P$ is contained in an $m$-dimensional projective subspace of $\mathbb{R}\mathrm{P}^{n}$ where $n > m >1$. The vertices of $P$ span at most a $2m$-dimensional space. 

\begin{thm}\label{thm: pencilofquadrics}
Let $P: \Z^2 \to \mathbb{R}\mathrm{P}^{2m}$, where $m \geq 2$, be a Q-net inscribed in a non-degenerate quadric $\mathcal{Q}$. Assume that the vertices $\{P_{i,j}\}_{j\in \Z}$ are contained in $m$-dimensional projective subspaces $V_i$, and the vertices $\{P_{i,j}\}_{i \in \Z}$ are contained in $m$-dimensional projective subspaces $H_j$. Then, generically, the following properties hold:
\begin{itemize}
\item $V_i$ are isotropic subspaces of a quadric $\mathcal{V}$ of signature $(m,m,1)$.
\item $H_j$ are isotropic subspaces of a quadric $\mathcal{H}$ of signature $(m,m,1)$.
\item  $\{V_{2i}\}_{i\in \Z}$ and $\{V_{2i+1}\}_{i\in \Z}$ are two systems of generators of $\mathcal{V}$.
\item $\{H_{2j}\}_{j\in \Z}$ and $\{H_{2j+1}\}_{j\in \Z}$ are two systems of generators of $\mathcal{H}$.
\item The quadrics $\mathcal{Q}$, $\mathcal{V}$ and $\mathcal{H}$ belong to a pencil of quadrics.
\end{itemize}
\end{thm}

We prove Theorem~\ref{thm: pencilofquadrics} with three simple lemmas, the first of which can be found in \cite[Lemma 3.2]{bobenko2020checkerboard}.

\begin{lem}\label{lem: 2pointsandpencilofquadrics}
Let $\mathcal{Q}_1$ and $\mathcal{Q}_2$ be two distinct quadrics in $\mathbb{R}\mathrm{P}^{n}$. Let $A, B, C, D$ be distinct points in $\mathcal{Q}_1 \cap \mathcal{Q}_2$ such that the lines $A \vee B$ and $C \vee D$ are not contained in $\mathcal{Q}_1\cap \mathcal{Q}_2$. Then, there are unique quadrics $\mathcal{Q}_{A,B}$ and $\mathcal{Q}_{C,D}$ in the pencil of quadrics spanned by $\mathcal{Q}_1$ and $\mathcal{Q}_2$ such that $A \vee B$ is an isotropic line of $\mathcal{Q}_{A,B}$ and $C\vee D$ is an isotropic line of $\mathcal{Q}_{C,D}$.  If the lines $A\vee B$ and $C \vee D$ are coplanar, then $\mathcal{Q}_{A,B} = \mathcal{Q}_{C,D}$.
\end{lem}


\begin{lem}\label{lem: 2quadricsdefinepencil}
Let $\mathcal{U}  \subset E_\mathcal{U}$ and $\mathcal{W}  \subset E_\mathcal{W}$ be $(n-2)$-dimensional quadrics in two distinct hyperplanes of $\mathbb{R}\mathrm{P}^{n}$.  Suppose that $\mathcal{U}$ and $\mathcal{W}$ coincide in $E_\mathcal{U} \cap E_\mathcal{W}$. Then, there is a pencil of $(n-1)$-dimensional quadrics that contain $\mathcal{U}$ and  $\mathcal{W}$. If  $E_\mathcal{U} \cap E_\mathcal{W}$ is not an isotropic subspace (of  $\mathcal{U}$ and $\mathcal{W}$), then the pencil of quadrics is uniquely determined.

\end{lem}

\begin{proof}
Let $F_\mathcal{U}$ and $F_\mathcal{W}$ be $n$-dimensional linear subspaces of $\R^{n+1}$ such that $E_\mathcal{U} = [F_\mathcal{U}]$, $E_\mathcal{W}=[F_\mathcal{W}]$. Let $b_1, b_2, \ldots, b_{n+1}$ be a basis of $\R^{n+1}$ such that $F_\mathcal{U} = \mathrm{span}\{b_1, \ldots, b_n\}$,  $F_\mathcal{W} = \mathrm{span}\{b_2, \ldots, b_{n+1}\}$. Denote $F:=F_\mathcal{U}\cap F_\mathcal{W}= \mathrm{span}\{b_2, \ldots, b_n\}$. Let $\p: F\times F\to \R$ be a symmetric bilinear form representing $\mathcal{U} \cap E_\mathcal{W}$. Choose representative symmetric bilinear forms $\p_\mathcal{U}:F_\mathcal{U}\times F_\mathcal{U}\to \R$, $\p_\mathcal{W}:F_\mathcal{W}\times F_\mathcal{W}\to \R$ of $\mathcal{U}$ and $\mathcal{W}$ as extensions of $\p$. Then, define $\p_\lambda:\R^{n+1}\times \R^{n+1}\to \R$, $\lambda\in \R$, as 
\[
 \p_\l(b_i,b_j) = 
  \begin{cases} 
   \p_\mathcal{U}(b_i, b_j) & \text{if } i,j< n+1 \\
   \p_\mathcal{W}(b_i,b_j) & \text{if } i,j> 1 \\
   \l       & \text{if } \{i,j\}=\{1, n+1\} .
  \end{cases}
\]
It determines a pencil of quadrics containing $\mathcal{U}$ and $\mathcal{W}$. The pencil is unique if $\p$ does not vanish identically.
\end{proof}

Let $P: \Z \times [m] \to \mathbb{R}\mathrm{P}^{n}$,  $m\geq 2$, be a Q-net such that, $\forall j \in [m]$, $H_j:=\mathrm{join}\{P_{i,j}\}_{i\in \Z}$ is $d$-dimensional, $d\geq 2$. The vertices of $P$ span at most a $(d+m-1)$-dimensional space. 

\begin{lem}\label{lem: mxZstripisoquadric}
Let $P: \Z \times [m] \to\mathcal{Q}\subset \mathbb{R}\mathrm{P}^{d+m-1}$, where $m, d\geq 2$, be a Q-net inscribed in a quadric $\mathcal{Q}$ such that the lines $P_{i,j}\vee P_{i,j+1}$ are not isotropic. Suppose that the spaces $H_j:=\mathrm{join}\{P_{i,j}\}_{i\in \Z}$ are $d$-dimensional non-isotropic subspaces. 
Then, generically, there exists a unique quadric $\mathcal{V}$ such that the $(m-1)$-dimensional spaces   $V_i:= \mathrm{join}\{P_{i,j}\}_{j\in [m]}$ are isotropic subspaces of $\mathcal{V}$.
\end{lem}

\begin{proof}
We first prove the base case $m=2$ and then the case  $m\geq 3$ under the induction hypothesis that Lemma~\ref{lem: mxZstripisoquadric} holds for  $m - 1$. Let $m=2$. The points $\{P_{i,j}\}_{(i,j)\in \Z \times [2]}$ are contained in the intersection of a pair of quadrics $\mathcal{Q}\cap (H_1\cup H_2)$. These quadrics are distinct because $H_1$ and $H_2$ are not isotropic subspaces of $\mathcal{Q}$. Lemma~\ref{lem: 2pointsandpencilofquadrics} uniquely determines a quadric $\mathcal{Q}_i$ from this pencil, such that $P_{i,1}\vee P_{i,2}$ is an isotropic line of $\mathcal{Q}_i$. The lines $P_{i,1}\vee P_{i,2}$ and $P_{i+1,1} \vee P_{i+1, 2}$ are coplanar because $P$ is a Q-net. Applying Lemma~\ref{lem: 2pointsandpencilofquadrics}, we get $\mathcal{Q}_i = \mathcal{Q}_{i+1}$. This proves the case $m=2$.

Let $m \geq 3$. For each $i\in \Z$, define $U_i:= \mathrm{join}\{P_{i,j}\}_{j=1,\ldots, m-1}$ and $W_i := \mathrm{join}\{P_{i,j}\}_{j=2,\ldots, m}$. Because $P$ is a Q-net and each space $H_j$ is $d$-dimensional, $\mathrm{join}\{H_1$, $H_2, \ldots,  H_{m-1}\}$ and \linebreak $\mathrm{join}\{H_2, H_3, \ldots,  H_{m}\}$ are both generically $(d+m-2)$-dimensional hyperplanes in $\R \mathrm{P}^{d+m-1}$. The two foregoing hyperplanes determine a degenerate quadric that contains all the points \linebreak $\{P_{i,j}\}_{(i,j)\in \Z \times [m]}$. The quadric $\mathcal{Q}$ is a distinct quadric that also contains all these points. Applying the induction hypothesis to $\mathrm{join}\{H_1, H_2, \ldots,  H_{m-1}\}$, there is a unique quadric $\mathcal{U}$ such that the spaces $U_i$ are  isotropic subspaces of $\mathcal{U}$. Applying the induction hypothesis to $\mathrm{join}\{H_2, H_3, \ldots,  H_{m}\}$, there is a unique quadric $\mathcal{W}$ such that the spaces $W_i$ are  isotropic subspaces of $\mathcal{W}$. By Lemma~\ref{lem: 2quadricsdefinepencil}, there is a unique pencil of quadrics  in $\mathbb{R}\mathrm{P}^{d+m-1}$ that contain $\mathcal{U}$ and $\mathcal{W}$. Let $O$ be any point in $V_0 \setminus \{U_0 \cup W_0\}$. There is a unique quadric $\mathcal{V}$ in the pencil containing $O$. The intersection $V_0 \cap \mathcal{V}$ is either a quadric or an isotropic space of $\mathcal{V}$. However, $V_0 \cap \mathcal{V}$ cannot be a quadric because it contains $U_0\cup W_0$, which is a degenerate quadric in $V_0$, and it also contains $O \notin U_0\cup W_0$. Therefore, $V_0$ must be an isotropic subspace of $\mathcal{V}$. Further, $V_{1}$ must also be an isotropic subspace of  $\mathcal{V}$. Indeed, $U_1$ and $W_1$ are isotropic subspaces of $\mathcal{V}$ and $V_0\cap V_1$ is also isotropic because $V_0$ is isotropic. However, $V_0\cap V_1$ is not contained in the degenerate quadric $U_1 \cup W_1 \subset V_1$. So, $V_1\cap\mathcal{V}$ is not a quadric but an isotropic subspace of $\mathcal{V}$. Iteratively, all $V_i$ are isotropic subspaces of $\mathcal{V}$.
\end{proof}

\begin{proof}[Proof of Theorem~\ref{thm: pencilofquadrics}]
By applying Lemma~\ref{lem: mxZstripisoquadric} to the restriction of $P :\Z^2 \to \mathbb{R}\mathrm{P}^{2m}$ to $\Z\times [m+1]$, there exists a quadric $\mathcal{V}$ such that the projective subspaces $\{V_i\}_{i\in \Z}$ are isotropic subspaces of $\mathcal{V}$. By applying Lemma~\ref{lem: mxZstripisoquadric} to the restriction of $P :\Z^2 \to \mathbb{R}\mathrm{P}^{2m}$ to $[m+1]\times \Z$, there exists a quadric $\mathcal{H}$ such that the projective subspaces $\{H_j\}_{j \in \Z}$ are isotropic subspaces of $\mathcal{H}$. Non-degenerate quadrics with $m$-dimensional isotropic spaces can only exist in $\mathbb{R}\mathrm{P}^{n}$ where $n \geq 2m+1$. The quadrics $\mathcal{V}$ and $\mathcal{H}$ must be degenerate because they are contained in $\R \mathrm{P}^{2m}$. Let $(p,q,r)$ be the signature of $\mathcal{V}$. Looking for a contradiction, suppose that $r \geq 3$. Then, $\mathcal{V}$ has at least a $2$-dimensional subspace of singular points. By the construction of $\mathcal{V}$, which is explained in the proof of Lemma~\ref{lem: mxZstripisoquadric}, $\mathcal{V}$ is in the pencil of quadrics spanned by $\mathcal{Q}$ and a degenerate quadric, say $\mathcal{D}$, that consists of a pair of hyperplanes in $\R \mathrm{P}^{2m}$. The locus of singular points of $\mathcal{D}$ is $(2m-2)$-dimensional, which intersects the locus of singular points of $\mathcal{V}$ in at least one point. Such an intersection point must also be a singular point of $\mathcal{Q}$ because $\mathcal{Q}$ is in the pencil of quadrics spanned by $\mathcal{V}$ and $\mathcal{D}$. This contradicts the assumption that $\mathcal{Q}$ is non-degenerate. Therefore, $r \leq 2$. The generators of $\mathcal{V}$ are $(min\{p,q\}-1+r)$-dimensional. Because $\mathcal{V}$ contains $m$-dimensional isotropic spaces, $\mathcal{V}$ has signature $(m,m,1)$. Similarly, $\mathcal{H}$ also has signature $(m,m,1)$.

To prove the claim about the systems of generators, observe that $V_i\cap V_{i+1}$ is $(m-1)$-dimensional. By intersecting with a generic hyperplane $E$, the dimension of $V_i\cap V_{i+1}\cap E$ is $m-2$ which has the same parity as $m$. Thus, by the explanation at Proposition~\ref{prop: hyperbolicquadrictwoclasses}, we have two systems of generators $\{V_{2i}\}_{i\in \Z}$ and $\{V_{2i+1}\}_{i\in \Z}$. 


Generically, the apex of $\mathcal{H}$ is not contained in $\mathcal{Q}$. Then, there is a unique quadric $\mathcal{R}$ that contains the apex of $\mathcal{H}$ and that is a member of the pencil of quadrics spanned by $\mathcal{V}$ and $\mathcal{Q}$. We will show that $\mathcal{R}$ equals $\mathcal{H}$. The $m$-dimensional spaces $H_j$ are not isotropic subspaces of $\mathcal{Q}$ because $\mathcal{Q}$ is non-degenerate. So, $H_j \cap \mathcal{Q}$ is a quadric that contains the points $\{P_{i,j}\}_{i\in \Z}$. Generically, $H_j \cap \mathcal{Q}$ is the unique quadric in $H_j$ that contains these points. So, $\mathcal{R}$ contains the quadric $H_j\cap \mathcal{Q}$ because $\mathcal{R}$ contains the points $\{P_{i,j}\}_{i\in \Z}$ that are in the base locus of the pencil of quadrics spanned by $\mathcal{Q}$ and $\mathcal{V}$. However, $\mathcal{R}$ also contains the apex of $\mathcal{H}$ which is not contained in $\mathcal{Q}$. So, $H_j \cap \mathcal{R}$ cannot be a quadric. Instead, $H_j$ is an isotropic subspace of $\mathcal{R}$. The quadrics $\mathcal{R}$ and $\mathcal{H}$ are equal because they have the same generators $H_j$.
\end{proof}

\section{Circular Nets with Spherical Parameter Lines}
\label{section: Mobiusapproach}

Circular nets in $\R^3$  provide a M\"obius-geometric discretisation of curvature-line parametrisations. See \cite{DDG, bobenko2018curvature, bobenko2003discrete, cieslinski1997}.

A map $P: \Z^2 \to \R^n$ is a circular net if  the vertices $P_{i,j}$, $P_{i+1,j}$, $P_{i+1,j+1}$, $P_{i,j+1}$ lie on circles, i.e. $\Box_{i,j}:= \Box(P_{i,j}, P_{i+1,j},P_{i+1,j+i}, P_{i,j+1})$ are  circular quads. The 
points $\{P_{i,j}\}_{i \in \Z}$ and $\{P_{i,j}\}_{j \in \Z}$ build two families of parameter lines. If $n=3$, then the parameter lines are interpreted as discrete curvature lines.  

Throughout this Section~\ref{section: Mobiusapproach}, the lift of a circular net $P: \Z^2 \to \R^n$ to the M{\"o}bius quadric $\mathcal{M}^n:=\{ [x]\in \R \mathrm{P}^{n+1}\mid x_1^2+x_2^2+\ldots + x_{n+1}^2 -x_{n+2}^2 =0 \}$ is denoted  $M: \Z^2 \to \mathcal{M}^n\subset \R \mathrm{P}^{n+1}$. The lift is a Q-net that is inscribed in the M\"obius quadric such that $\sigma(M_{i,j}) = P_{i,j}$, where $\sigma: \R \mathrm{P}^{n+1} - N \to \overline{\R^n}$ is the central projection with centre $N$ that is defined in Appendix~\ref{appendix: mobiusgeo}. Here $\overline{\R^n}$ is the projective closure of $\R^n$. See Appendix~\ref{appendix: mobiusgeo} for a brief introduction to M\"obius geometry.

\begin{defn}
Let $P:\Z^2 \to \R^n$ be a circular net and let $M: \Z^2 \to \mathcal{M}^n \subset \R \mathrm{P}^{n+1}$ be its lift. A parameter line $\{P_{i,j}\}_{i \in \Z}$ is \emph{$m$-spherical} or \emph{$m$-planar} if its vertices are contained in an $m$-dimensional sphere or $m$-dimensional plane in $\R^n$, respectively. Equivalently, the points $\{M_{i,j}\}_{i\in \Z}$ of the lift $M: \Z^2 \to \mathcal{M}^n \subset \R \mathrm{P}^{n+1}$ to the M\"obius quadric are contained in a $(m+1)$-dimensional projective subspace of $\R \mathrm{P}^{n+1}$.
\end{defn}

A parameter line $\{P_{i,j}\}_{i \in \Z}$ is $m$-planar when the corresponding $(m+1)$-dimensional projective subspace of $\R \mathrm{P}^{n+1}$ contains the point $N$.

Some of the proofs in this Section~\ref{section: Mobiusapproach} repeatedly use the same arguments. To avoid excessive repetition, some of them are omitted.

\subsection{One Family Spherical or Circular}\label{section: Mobiusonefamily}
The following proposition is a direct corollary of Proposition~\ref{prop: Goursatterminatemsteps}. 

\begin{prop}\label{prop: sphericaliffgoursat}
Let $M: \Z^2 \to \mathcal{M}^n \subset \R \mathrm{P}^{n+1}$ be a circular net,  and $m <n$. Generically, the parameter lines  $\{ M_{i,j}\}_{i \in \Z}$ are $m$-spherical if and only if $\mathcal{L}_A^{m+1} M$ is Goursat degenerate.
\end{prop}

\begin{prop}\label{prop: sphericalterminatbothsides}
Let $M: \Z^2 \to \mathcal{M}^n \subset \R \mathrm{P}^{n+1}$ be a circular net, where $n \geq 2$. Suppose that the parameter lines  $\{(M_{i,j})_{i \in \Z}\}_{j\in \Z}$ are $m$-spherical, where $m < n$. Generically, $\mathcal{L}_B^{m+1}M$ is Laplace degenerate. For each $(i,j) \in \Z^2$, let $\mathcal{S}_{i,j}$  be the $m$-dimensional sphere that contains the points $M_{i,j}, \ldots, M_{i, j+m+1}$. For each $j \in \Z$, the $m$-spheres $\{\mathcal{S}_{i,j}\}_{i \in \Z}$ have a common orthogonal hypersphere.
\end{prop}

\begin{proof}
By Proposition~\ref{prop: sphericaliffgoursat},  $\mathcal{L}_A^{m+1} M$ is generically Goursat degenerate. Generically, by Theorem~\ref{thm: GoursatimpliesLaplaceQnetquadric}, $\mathcal{L}_B^{m+1} M$ is Laplace degenerate. For each $j \in \Z$, the point $\mathcal{L}_B^{m+1} M(j) :=\mathcal{L}_B^{m+1} M(i,j)$ is well defined because it is  independent of $i \in \Z$.  For each $(i,j) \in \Z^2$, let $S_{i,j}$  be the $(m+1)$-dimensional projective subspace that contains the points $M_{i,j}, \ldots, M_{i, j+m+1}$. Then, $S_{i,j} \cap \mathcal{M}^n = \mathcal{S}_{i,j}$. For each $j \in \Z$, the spaces $\{S_{i,j}\}_{i \in \Z}$ contain the point $\mathcal{L}_B^{m+1} M(j)$. Let $\mathcal{L}_B^{m+1} M(j)^\perp$ be the polar hyperplane of $\mathcal{L}_B^{m+1} M(j)$ relative to $\mathcal{M}^n$. The (possibly imaginary) hypersphere  $\mathcal{L}_B^{m+1} M(j)^\perp \cap \mathcal{M}^n$ is orthogonal to the $m$-dimensional spheres $\{\mathcal{S}_{i,j}\}_{i \in \Z}$.
\end{proof}

Let $P: \Z^2 \to \R^3$ be a circular net with spherical parameter lines $\{(P_{i,j})_{i \in \Z}\}_{j\in \Z}$. Let $M: \Z^2 \to \mathcal{M}^{3}$ be its lift. By the case $m=2$ of Proposition~\ref{prop: sphericaliffgoursat}, $\mathcal{L}_A^{3} M$ is Goursat degenerate. For each $(i,j) \in \Z^2$ and $d \in [3]$, $\sigma(\mathcal{L}_A^{d} M (i,j)) = \mathcal{L}_A^{d} P(i,j)$. So, $\mathcal{L}_A^{3} P$ is Goursat degenerate. This agrees with the smooth theory of surfaces in $\R^3$ with one family of spherical curvature lines \cite{goursat1896equations}. By the case $m=2$ of Proposition~\ref{prop: sphericalterminatbothsides}, $\mathcal{L}_B^{3} M$ is Laplace degenerate. So, $\mathcal{L}_B^{3} P$ is Laplace degenerate. This differs to the smooth theory.  Along any of the spherical curvature lines of a smooth surface with one family of spherical curvature lines, the osculating planes of the non-spherical curvature lines are generically concurrent \cite{picart1863}. Then, the case $m=2$ of Proposition~\ref{prop: Laplaceterminatemsteps} implies that the circular net $P: \Z^2 \to \R^3$ with one family of spherical parameter lines $\{(P_{i,j})_{i \in \Z}\}_{j\in \Z}$ would better mimic a smooth surface with one family of spherical curvature lines if $\mathcal{L}_B^2P$ were Laplace degenerate. Let $\nu_{i,j}$ be the discrete normal through the centre of the circle containing $P_{i,j}$, $P_{i+1,j}$, $P_{i+1,j+1}$, $P_{i,j+1}$ and that is orthogonal to the plane of the circle. It can readily be shown that the planes $\{\nu_{i,j} \vee \nu_{i,j+1}\}_{i \in \Z}$ are concurrent at the centre of the sphere through the points $\{P_{i,j+1}\}_{i \in \Z}$. This feature is comparable to the condition that the osculating planes are concurrent.

In Section~\ref{Lie: onefamilyspherical}, an alternation phenomenon is exhibited that is used to define an improved discretisation of surfaces with spherical curvature lines.

Let $P: \Z^2 \to \R^3$ be a circular net with circular parameter lines $\{(P_{i,j})_{i \in \Z}\}_{j\in \Z}$. By the case $m=1$ of Proposition~\ref{prop: sphericaliffgoursat}, $\mathcal{L}_A^{2} M$ is Goursat degenerate.  For each $(i,j) \in \Z^2$ and $d \in [2]$, $\sigma(\mathcal{L}_A^{d} M (i,j)) = \mathcal{L}_A^{d} P(i,j)$.  So, $\mathcal{L}_A^{2} P$ is Goursat degenerate. This agrees with the smooth theory of surfaces in $\R^3$ with one family of circular curvature lines \cite{goursat1896equations}, which are classical surfaces that are known as \emph{channel surfaces}. By the case $m=1$ of Proposition~\ref{prop: sphericalterminatbothsides}, $\mathcal{L}_B^{2} M$ is Laplace degenerate. So, $\mathcal{L}_B^{2} P$ is Laplace degenerate. This differs to the smooth theory. Joachimsthal's theorem ensures that the tangent planes of a channel surface along any circular curvature line are the tangent planes of a cone (or cylinder) of revolution. The tangent lines of the non-circular curvature lines are rulings of the foregoing cone. Therefore, any channel surface has a Laplace transform that is Laplace degenerate. It consists of the locus of the apices of the cones of revolution that are tangent to the channel surface along the circular curvature lines. Therefore, the circular net $P: \Z^2 \to \R^3$ with one family of circular parameter lines $\{(P_{i,j})_{i \in \Z}\}_{j\in \Z}$ would better mimic a channel surface if $\mathcal{L}_B P$ were Laplace degenerate. The condition that $\mathcal{L}_B P$ is Laplace degenerate is present in the Lie-geometric discretisation of channel surfaces in \cite{hertrich2020}. Let $\nu_{i,j}$ be the discrete normal through the centre of the circle containing $P_{i,j}$, $P_{i+1,j}$, $P_{i+1,j+1}$, $P_{i,j+1}$ and that is orthogonal to the plane of the circle. It can readily be shown that the planes $\{\nu_{i,j} \vee \nu_{i,j+1}\}_{i \in \Z}$ are concurrent at the line through the centre of the circle containing the points $\{P_{i,j+1}\}_{i \in \Z}$ and that is orthogonal to the plane of the circle.

\begin{prop}\label{prop: onefamilyplanar}
Let $P: \Z^2 \to \R^n$ be a circular net, where $n\geq 2$. For each $j\in \Z$, suppose that the points $\{P_{i,j}\}_{i \in \Z}$ are contained in an $m$-dimensional plane $\mathcal{H}_j$, where $m<n$.  Generically, $\mathcal{L}^{m}_AP$ is Goursat degenerate and $\mathcal{L}^{m+1}_BP$ is Laplace degenerate. The discrete curve $\mathcal{L}^{m+1}_BP$ is contained in the hyperplane at infinity of $\overline{\R^{n}}$. For each $j\in \Z$, the lines $\{\mathcal{L}^m_BP(i,j)\vee \mathcal{L}^m_BP(i,j+1)\}_{i \in \Z}$ are parallel.
\end{prop}

\begin{proof}
By Proposition~\ref{prop: Goursatterminatemsteps}, $\mathcal{L}_A^{m} P$ is Goursat degenerate. Let $M: \Z^2 \to \mathcal{M}^n$ be the lift of $P$. For each $j \in \Z$, $H_j := \mathrm{join}\{M_{i,j}\}_{i \in \Z}$ is a $(m+1)$-dimensional projective subspace of $\R \mathrm{P}^{n+1}$ that contains the point $N$. By Proposition~\ref{prop: Laplaceterminatemsteps},  $\mathcal{L}_A^{m+1} M$ is Laplace degenerate. By Proposition~\ref{prop: Goursatterminatemsteps},  $\mathcal{L}_A^{m+1} M$ is Goursat degenerate. So, $N = \mathcal{L}_A^{m+1} M(i,j)$ for all $(i,j) \in \Z^2$. By Theorem~\ref{thm: GoursatimpliesLaplaceQnetquadric}, $\mathcal{L}_B^{m+1} M$ is Laplace degenerate. So,  $\mathcal{L}_B^{m+1} M(j):=  \mathcal{L}_B^{m+1} M(i,j)$ is well defined. By Theorem~\ref{thm: mxmiteratedlaplaceconjugate}, $N$ and $\mathcal{L}_B^{m+1} M(j)$ are conjugate for all $j \in \Z$. Then, the points $\{\mathcal{L}_B^{m+1} M(j)\}_{j \in\Z}$ are contained in the polar hyperplane $N^\perp$. For each $(i,j) \in \Z^2$, $\sigma(\mathcal{L}_B^{m+1} M (i,j)) = \mathcal{L}_B^{m+1} P(i,j)$. The polar hyperplane $N^\perp$ projects under $\sigma$ to the hyperplane at infinity of $\overline{\R^n}$. Therefore, $\mathcal{L}_B^{m+1} P$ is Laplace degenerate and the points $\{\mathcal{L}_B^{m+1} P(j)\}_{j \in\Z}$ are contained in the hyperplane at infinity of $\overline{\R^{n}}$.
\end{proof}

\subsection{Both Families Circular}\label{section: circularcircular}

\begin{prop}\label{thm: circularcircular4x4}
Consider a function $M: [4]\times [4] - \{(4,4)\} \to \mathcal{M}^n \subset \R\mathrm{P}^{n+1}$ where $n\in \{2,3\}$.  For each $j \in [3]$, let the points $\{M_{i,j}\}_{i\in [4]}$ be contained in a circle $\mathcal{H}_j$. Let $\mathcal{H}_4$ be the circle containing the points $M_{1,4}, M_{2,4}, M_{3,4}$. For each $i \in [3]$, let the points $\{M_{i,j}\}_{j\in [4]}$ be contained in a circle $\mathcal{V}_i$. Let $\mathcal{V}_4$ be the circle containing the points $M_{4,1}, M_{4,2}, M_{4,3}$. Suppose that $\forall (i,j) \neq (3,3)$ the quad $\Box_{i,j}:= \Box(M_{i,j}, M_{i+1,j}, M_{i+1,j+i}, M_{i,j+1})$ is circular. There exists a unique point $M_{4,4}\in \mathcal{M}^n$ such that the quad $\Box_{3,3}$ is circular and such that $M_{4,4} \in \mathcal{V}_4\cap \mathcal{H}_4$. (See Figure~\ref{figure: 4x4circulargridincidence}.)
\end{prop}

\begin{proof}
Let $V_i$ and $H_j$ denote the $2$-dimensional projective subspaces of $\R\mathrm{P}^{n+1}$ such that $\mathcal{V}_i = V_i \cap \mathcal{M}^n$ and $\mathcal{H}_j = H_j \cap \mathcal{M}^n$. Define $M_{4,4}:= (M_{3,3} \vee M_{4,3} \vee M_{3,4}) \cap V_4 \cap H_4$. By Theorem~\ref{thm: mxmincidencethm}, $M_{4,4} \in \mathcal{M}^n$.
\end{proof}

\begin{proof}[Alternative proof of Proposition~\ref{thm: circularcircular4x4}]
By applying Miquel's theorem to the points $\{M_{i,j} \mid i \in  [4], \, j \in [2]\}$, it follows that the points $M_{1,1}, M_{1,2}, M_{4,1},M_{4,2}$ are contained in a circle. For instance, see \cite{DDG} for more information regarding Miquel's theorem. In other words, the points $\{M_{i,j} \mid i\in [4],  j\in [2]\}$ are the eight vertices of a cube with planar faces. By Miquel's theorem, the following quadruples of points are also contained in circles.
\begin{gather*}
\{M_{1,2},M_{1,3},M_{4,2},M_{4,3}\} \; \;\;\; \; \{ M_{1,1}, M_{2,1}, M_{1,4}, M_{2,4}\} \; \;\;\;\; \{ M_{2,1}, M_{3,1}, M_{2,4}, M_{3,4}\}
\end{gather*}
By the $4$-dimensional consistency of Q-nets \cite[Theorem 2.5]{DDG}, the following six planes are concurrent because the vertices have the structure of a hypercube.
\begin{gather*}
M_{3,3} \vee M_{4,3} \vee M_{3,4} \qquad M_{1,1}\vee M_{4,1}\vee M_{1,4}\\
M_{1,3}\vee M_{1,4}\vee M_{4,3} \qquad M_{3,1}\vee M_{4,1}\vee M_{3,4}\\
M_{4,1}\vee M_{4,2}\vee M_{4,3} \qquad M_{1,4}\vee M_{2,4}\vee M_{3,4}
\end{gather*}

Let $M_{4,4}$ be the concurrency point of the six planes. By a classical theorem concerning associated points, which can be found in \cite[Thm. 3.13]{DDG}, the point $M_{4,4}$ is contained in $\mathcal{M}^n$. So, the following quadruples of points are contained in circles.
\begin{gather*}
\{M_{3,3}, M_{4,3}, M_{3,4}, M_{4,4} \}\qquad \{M_{1,1}, M_{4,1}, M_{1,4},M_{4,4}\}\\
\{M_{1,3}, M_{1,4}, M_{4,3}, M_{4,4}\} \qquad \{M_{3,1} , M_{4,1}, M_{3,4}, M_{4,4}\}\\
\{M_{4,1}, M_{4,2}, M_{4,3}, M_{4,4}\} \qquad \{M_{1,4}, M_{2,4}, M_{3,4},M_{4,4}\}
\end{gather*}\end{proof}
In Proposition~\ref{thm: circularcircular4x4}, the point $M_{4,4}$ is uniquely determined as the concurrency point of three circles. In the alternative proof of Proposition~\ref{thm: circularcircular4x4}, it is shown that the point $M_{4,4}$ is actually the concurrency point of six circles that are determined by quadruples of vertices of the circular net $M: [4]\times [4] \to \mathcal{M}^n$.

In Proposition~\ref{thm: circularcircular4x4}, it is superfluous to consider the case $n > 3$. In the alternative proof of Proposition~\ref{thm: circularcircular4x4}, the sixteen vertices $\{M_{i,j} \mid (i,j) \in [4]\times [4]\}$ are identified with the vertices of a hypercube that is inscribed in $\mathcal{M}^n$. The sixteen vertices of any hypercube span at most a $4$-dimensional projective subspace of $\R \mathrm{P}^{n+1}$. Therefore, the vertices $\{M_{i,j} \mid (i,j) \in [4]\times [4]\}$ must be contained in a $3$-dimensional sphere if $n >3$. 

Proposition~\ref{thm: circularcircular4x4} can be used iteratively to construct circular nets $\Z^2 \to \mathcal{M}^n$, where $n \in \{2,3\}$, such that both families of parameter lines are circular.

Proposition~\ref{prop: circularcircularR2R3} involves Darboux cyclides and cyclics. These are defined in Appendix~\ref{appendix: mobiusgeo}.

\begin{prop}\label{prop: circularcircularR2R3}
Let $M: \Z^2 \to \mathcal{M}^n\subset \R \mathrm{P}^{n+1}$, where $n\in \{2,3\}$, be a circular net. For all $i\in \Z$, let the points $\{M_{i,j}\}_{j\in \Z}$ be contained in a circle $\mathcal{V}_i$. For all $j \in \Z$, let the points $\{M_{i,j}\}_{i\in \Z}$ be contained in a circle $\mathcal{H}_j$. 
\begin{enumerate}[label=(\roman*)]
\item
Let $n=3$. The circles $\{\mathcal{V}_i\}_{i\in \Z}$ and $\{\mathcal{H}_j\}_{j\in \Z}$ are contained in a $2$-dimensional Darboux cyclide $\mathcal{D}$. (See Figure~\ref{figure: 4x4circulargridincidence}.)
\item
Let $n=2$. The circles $\{\mathcal{V}_i\}_{i\in \Z}$ and $\{\mathcal{H}_j\}_{j\in \Z}$ envelop a cyclic $\mathcal{D}$. Each circle $\mathcal{V}_i$ is twice tangent to $\mathcal{D}$ and each circle $\mathcal{H}_j$ is twice tangent to $\mathcal{D}$. (See Figure~\ref{figure: bircircularquarticCASa}.)
\end{enumerate}
\end{prop}

\begin{proof}
Let $V_i$ and $H_j$ denote the $2$-dimensional projective subspaces of $\R\mathrm{P}^{n+1}$ such that $\mathcal{V}_i = V_i \cap \mathcal{M}^n$ and $\mathcal{H}_j = H_j \cap \mathcal{M}^n$.

Let $n=3$. The case $m=2$ of Theorem~\ref{thm: pencilofquadrics} implies that $\{V_i\}_{i \in \Z}$ are isotropic subspaces of a quadric $\mathcal{V}$. Similarly, $\{H_j\}_{j \in \Z}$ are isotropic subspaces of a quadric $\mathcal{H}$. By Theorem~\ref{thm: pencilofquadrics}, $\mathcal{H}$ and $\mathcal{V}$ span a pencil of quadrics that contains $\mathcal{M}^3$. The base locus of the pencil of quadrics is the required Darboux cyclide.

Let $n=2$. The first step is to show that the circles $\{\mathcal{V}_i\}_{i\in \Z}$ envelop a cyclic. By Theorem~\ref{thm: mxmincidencethm}, $\cap_{i\in [4]}V_i$ is generically a point. It follows that $Y:= \cap_{i\in \Z}V_i$ is a point. By Lemma~\ref{lem: mxZstripisoquadric}, the lines $\{M_{i,1}\vee M_{i,2}\}_{i \in \Z}$ are isotropic lines of a $2$-dimensional quadric, say $\mathcal{Q}$. Generically, $\mathcal{Q}$ has signature $(++--)$. Each plane $V_i$ can be described as $(M_{i,1}\vee M_{i,2})\vee Y$. So, the planes $\{V_i\}_{i\in \Z}$ are tangent planes of the quadratic cone $\mathcal{V}$ that is determined by the join of the point $Y$ with the conic $Y^{\perp_{\mathcal{Q}}} \cap \mathcal{Q}$, where $Y^{\perp_{\mathcal{Q}}}$ is the polar plane of $Y$ relative to $\mathcal{Q}$. The cone $\mathcal{V}$ intersects the sphere $\mathcal{M}^2\subset \R \mathrm{P}^3$ in a cyclic $\mathcal{D}:= \mathcal{M}^2\cap \mathcal{V}$ such that, $\forall i \in \Z$, the circle $\mathcal{V}_i$ is twice tangent to $\mathcal{D}$. The plane $V_i$ is tangent to the cone $\mathcal{V}$ along an isotropic line that intersects the sphere $\mathcal{M}^2$ in two (possibly imaginary) points that are the tangency points of $\mathcal{V}_i$ with the cyclic $\mathcal{D}$. The plane $H_1$ intersects the plane $Y^{\perp_{\mathcal{Q}}}$ in a line that intersects $\mathcal{M}^2$ in two (possibly imaginary) points. The circle $\mathcal{H}_1$ is contained in $\mathcal{Q}$ and also in $\mathcal{M}^2$. So, the two intersection points $H_1\cap Y^{\perp_{\mathcal{Q}}} \cap \mathcal{M}^2$ are the two (possibly imaginary) tangency points of $\mathcal{H}_1$ with the cyclic $\mathcal{D}$. Similarly, the circle $\mathcal{H}_2$ is twice tangent to $\mathcal{D}$. For any $j\in \Z$, if the quadric $\mathcal{Q}$ is replaced by the quadric that contains the lines $\{M_{i,j}\vee M_{i,j+1}\}_{i \in \Z}$, the same reasoning shows that the circles $\mathcal{H}_j$ and $\mathcal{H}_{j+1}$ are twice tangent to the cyclic $\mathcal{D}$. Therefore, the circles $\{\mathcal{H}_j\}_{j\in \Z}$ are twice tangent to the cyclic $\mathcal{D}$.
\end{proof}

\begin{cor}\label{cor: circularcircularR2confocaldeferents}
Let $P: \Z^2 \to \mathbb{R}^2$ be a circular net. For all $i\in \Z$, let the points $\{P_{i,j}\}_{j\in \Z}$ be contained in a circle $\mathcal{V}_i$. For all $j \in \Z$, let the points $\{P_{i,j}\}_{i\in \Z}$ be contained in a circle $\mathcal{H}_j$. Generically, the following properties hold. (See Figure~\ref{figure: bircircularquarticCASa}.)
\begin{itemize}
\item The centres of the circles $\{\mathcal{H}_j\}_{j\in\Z}$ are contained in a conic $\mathcal{Q}_{\mathcal{H}}$.
\item The centres of the circles $\{\mathcal{V}_i\}_{i\in \Z}$ are contained in a conic $\mathcal{Q}_{\mathcal{V}}$.
\item The conics $\mathcal{Q}_{\mathcal{H}}$ and $\mathcal{Q}_{\mathcal{V}}$ are confocal.
\item The circles $\{\mathcal{H}_j\}_{j\in\Z}$ have a common orthogonal circle $\mathcal{K}_{\mathcal{H}}$.
\item The circles $\{\mathcal{V}_i\}_{i\in\Z}$ have a common orthogonal circle $\mathcal{K}_{\mathcal{V}}$.
\item The circles $\mathcal{K}_{\mathcal{H}}$ and $\mathcal{K}_{\mathcal{V}}$ are orthogonal. 
\end{itemize}
\end{cor}

\begin{figure}[htbp]
\[\includegraphics[width=0.7\textwidth]{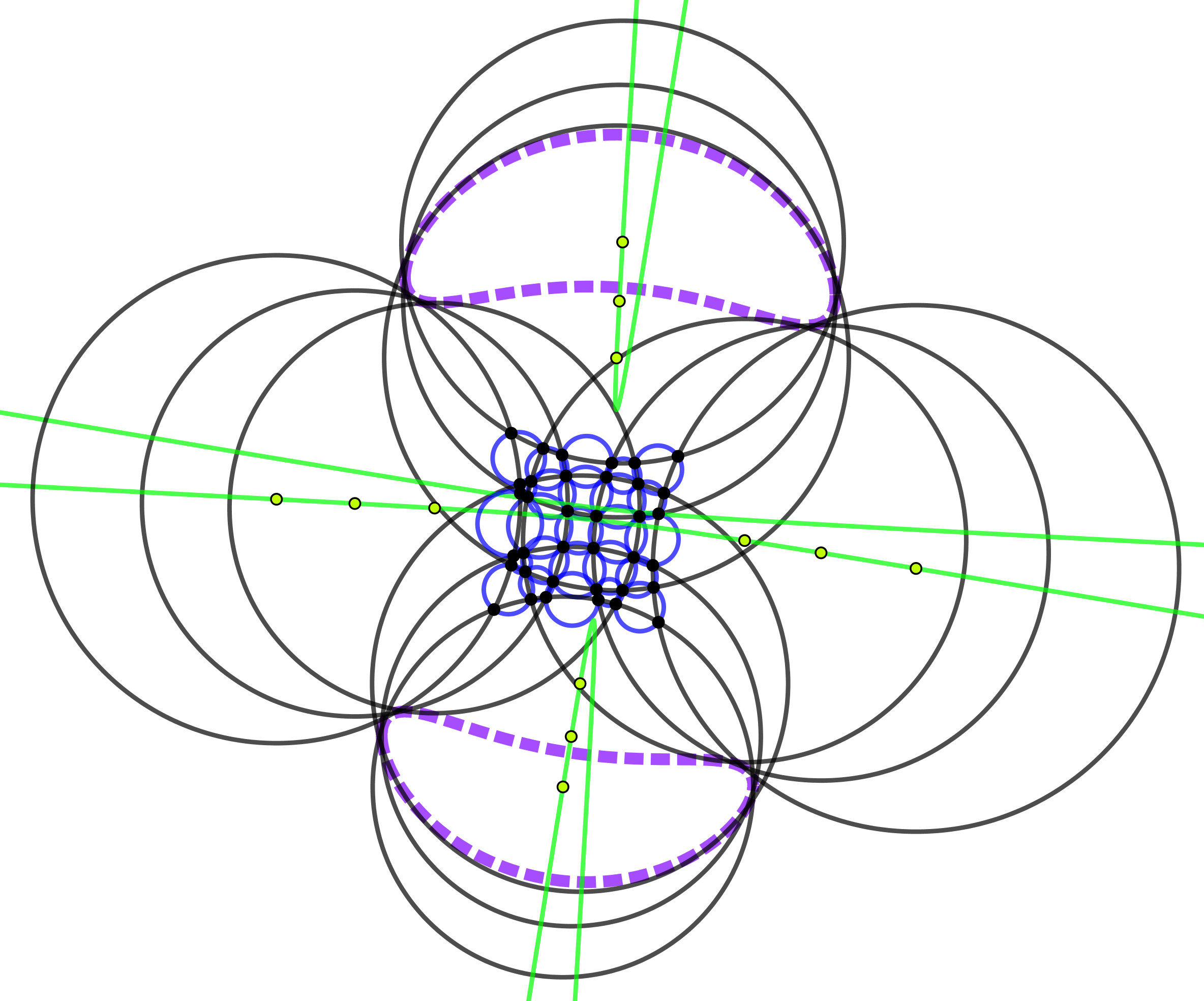}\]
\caption{A circular net $[6]\times [6]\to \R^2$ with circular parameter lines. The circles of the circular parameter lines are twice tangent to a cyclic, which is the dashed curve with two components. The centres of the circles of the circular parameter lines are contained in a pair of confocal conics.}
\label{figure: bircircularquarticCASa}
\end{figure}

\begin{proof}
By the case $n=2$ of Theorem~\ref{prop: circularcircularR2R3}, the circles of the circular parameter lines envelop a cyclic. The circles $\{\mathcal{H}_j\}_{j\in \Z}$ provide one generation of the cyclic. The circles $\{\mathcal{V}_i\}_{i\in \Z}$ provide another generation of the cyclic. Therefore, by Theorem~\ref{thm: cyclide}, the six properties are satisfied.
\end{proof}

In Corollary~\ref{cor: circularcircularR2confocaldeferents}, it can happen that one of the circles $\mathcal{K}_\mathcal{H}$ or $\mathcal{K}_\mathcal{V}$ is imaginary. However, $\mathcal{K}_\mathcal{H}$ and $\mathcal{K}_\mathcal{V}$ cannot both be imaginary.

In classical differential geometry, smooth surfaces in $\R^3$ such that both families of curvature lines are circles are known as \emph{Dupin cyclides}. The Darboux cyclides in the case $n=3$ of Proposition~\ref{prop: circularcircularR2R3} are generically not Dupin cyclides. Darboux cyclides are algebraic surfaces that are more general than Dupin cyclides.


In the proof of the case $n=3$ of Proposition~\ref{prop: circularcircularR2R3}, Theorem~\ref{thm: pencilofquadrics} implies that the spaces $\{V_i\}_{i\in\Z}$ decompose alternatingly into two systems of generators $\{V_{2i}\}_{i\in\Z}$ and $\{V_{2i+1}\}_{i\in\Z}$ of a quadric $\mathcal{V}$ that generically has signature $(++--0)$. Similarly, the spaces $\{H_j\}_{j\in\Z}$ decompose alternatingly into two systems of generators $\{H_{2j}\}_{j\in\Z}$ and $\{H_{2j+1}\}_{j\in\Z}$ of a quadric $\mathcal{H}$ that generically has signature $(++--0)$. Therefore, both families of circular parameter lines have an alternation phenomenon. For the circular net $M: \Z^2 \to \mathcal{M}^3$ to admit a smooth limit, one should remove the alternation phenomena. This can be achieved by requiring that $\mathcal{V}$ and $\mathcal{H}$ both degenerate to quadrics of signature $(++-00)$. Then, $\cap_{i\in \Z}V_i$ and $\cap_{j\in \Z}H_j$ are both $1$-dimensional. Then, Proposition~\ref{prop: 3x3X1D} implies that $\mathcal{L}_A M$ and $\mathcal{L}_B M$ are both Laplace degenerate. Using the terminology in \cite{bobenko2019multi}, this means that $M$ is a \emph{multi-circular net} with circular parameter lines.

\begin{prop}\label{prop: 3x3X1D}
Let $P: [3]\times [3] \to \R \mathrm{P}^n$ be a Q-net, where $n \geq 3$. For all $j\in [3]$, let $H_j:= \mathrm{join}\{P_{i,j}\}_{i\in [3]}$ be $2$-dimensional. If $\cap_{j\in [3]}H_j$ is $1$-dimensional, then $\mathcal{L}_A P$ is Laplace degenerate.
\end{prop}

\begin{proof}
Let $l:= \cap_{j\in [3]}H_j$ be $1$-dimensional. Consider the Laplace points $\{A_{i,j}\mid (i,j)\in [2]\times [2]\}$.   Looking for a contradiction, suppose that $A_{1,1}\neq A_{1,2}$ and $A_{2,1}\neq A_{2,2}$. Then, $l= A_{1,1}\vee A_{1,2}= A_{2,1} \vee A_{2,2}$. Evidently, $A_{1,1}\vee A_{1,2}= P_{1,2} \vee P_{2,2}$ and $A_{2,1} \vee A_{2,2}= P_{2,2} \vee P_{3,2}$. Then, the points $P_{1,2}$, $P_{2,2}$ and $P_{3,2}$ are collinear. This contradicts the assumption that $H_2$ is $2$-dimensional. Therefore, at least one of the equations $A_{1,1}= A_{1,2}$ and $A_{2,1}= A_{2,2}$ must be valid. Looking for a contradiction, suppose that only one of the equations is valid. Without loss of generality, suppose that $A_{1,1}= A_{1,2}$ and $A_{2,1} \neq A_{2,2}$. Then, $l= A_{2,1}\vee A_{2,2}$. Equivalently, $l = P_{2,2}\vee P_{3,2}$ because generically $P_{2,2}\neq P_{3,2}$. $A_{1,1}=A_{1,2}$ is contained in the line $l$ because $A_{1,1} \in H_1 \cap H_2$ and $A_{1,2}\in H_2 \cap H_3$. Therefore, $A_{1,1}\in l$ and $P_{2,2} \in l$. Generically, $A_{1,1} \neq P_{2,2}$. Then, $l=A_{1,1}\vee P_{2,2}$. The line $A_{1,1}\vee P_{2,2}$ contains $P_{1,2}$. Therefore, the points $P_{1,2}$, $P_{2,2}$ and $P_{3,2}$ are contained in the line $l$. This contradicts the assumption that $H_2$ is $2$-dimensional. Therefore, $A_{1,1}= A_{1,2}$ and $A_{2,1}= A_{2,2}$.
\end{proof}

\subsection{One Family Linear and the Other Family Circular}

\begin{prop}\label{thm: linearcircular}
Consider a function $P: [4]\times [3]-\{(4,3)\} \to \R^n$ where $n\in \{2,3\}$. For each $j \in [2]$, let the points $\{P_{i,j}\}_{i\in [4]}$ be contained in a circle $\mathcal{H}_j$. Let $\mathcal{H}_3$ be the circle containing the points $P_{1,3},P_{2,3},P_{3,3}$. For each $i \in [3]$, let the points $\{P_{i,j}\}_{j\in [3]}$ be contained in a line $\mathcal{V}_i$. Let $\mathcal{V}_4$ be the line $P_{4,1}\vee P_{4,2}$. Suppose that $\forall (i,j) \neq (3,2)$ the quad $\Box_{i,j}$  is circular. There exists a unique point $P_{4,3}$ such that the quad $\Box_{3,2}$ is circular and such that $P_{4,3}$ is contained in $\mathcal{V}_4 \cap \mathcal{H}_3$. (See Figure~\ref{figure: linearcircular4x3}.)
\end{prop}

\begin{figure}[htbp]
\[\includegraphics[width=0.4\textwidth]{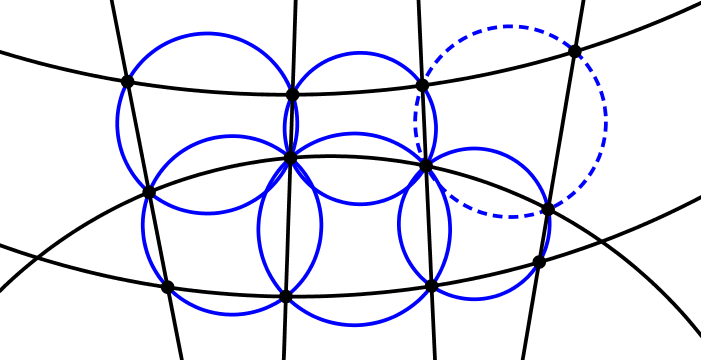}\]
\caption{A circular net $[4]\times [3]\to \R^2$ with one family of circular parameter lines and one family of linear parameter lines. Proposition~\ref{thm: linearcircular} ensures the existence of the dashed circle.}
\label{figure: linearcircular4x3}
\end{figure}

\begin{proof}
Let $M: [4]\times [3]-\{(4,3)\} \to \mathcal{M}^n \subset \R \mathrm{P}^{n+1}$ be the lift of $P$. Let $V_i$ and $H_j$ denote the $2$-dimensional projective subspaces of $\R\mathrm{P}^{n+1}$ such that $\sigma (V_i \cap \mathcal{M}^n) = \mathcal{V}_i$ and $\sigma (H_j \cap \mathcal{M}^n) = \mathcal{H}_j$. Generically, $V_1\cap V_2 \cap V_3 \cap V_4$ is a point. It must be the point $N$ because the $2$-dimensional spaces $\{V_i\}_{i\in [4]}$ contain $N$. Define $X:= H_1 \cap H_2 \cap H_3$. Generically, $X$ is a point. By applying Corollary~\ref{cor: mxmgridthm} to the restricted Q-net $M: [3]\times [3] \to \mathcal{M}^n$, the points $X$ and $N$ are conjugate relative to $\mathcal{M}^n$. Let $M_{4,3}:= V_4\cap H_3 \cap (M_{3,2} \vee M_{4,2} \vee M_{3,3})$. By applying Corollary~\ref{cor: mxmgridthm} to the points $\{M_{i,j} \mid i\in \{2,3,4\}, j\in [3]\}$, the point $M_{4,3}$ is contained in $\mathcal{M}^n$ because $X$ and $N$ are conjugate relative to $\mathcal{M}^n$. The point $\sigma(M_{4,3})$ is the required point $P_{4,3}$.
\end{proof}

Proposition~\ref{thm: linearcircular} can be used iteratively to construct circular nets $\Z^2 \to \R^n$, where $n\in \{2,3\}$, such that the parameter lines are linear in one family and circular in the other family.

\begin{prop}\label{prop: linearcircularquadricconcic}
Let $P: \Z^2 \to \R^n$, where $n\in \{2,3\}$, be a circular net. For all $i \in \Z$, let the points $\{P_{i,j}\}_{j\in\Z}$ be contained in a line $\mathcal{V}_i$. For all $j \in \Z$, let the points $\{P_{i,j}\}_{i\in\Z}$ be contained in a circle $\mathcal{H}_j$. 
\begin{enumerate}[label=(\roman*)]
\item
Let $n=3$. The lines $\{\mathcal{V}_i\}_{i\in \Z}$ are isotropic lines of a $2$-dimensional quadric $\mathcal{R}$ and the circles $\{\mathcal{H}_j\}_{j\in \Z}$ are contained in $\mathcal{R}$. 
\item
Let $n=2$. The lines $\{\mathcal{V}_i\}_{i\in \Z}$ are tangent lines of a conic $\mathcal{R}$ and the circles $\{\mathcal{H}_j\}_{j\in \Z}$ are twice tangent to $\mathcal{R}$. The centres of the circles $\{\mathcal{H}_j\}_{j\in \Z}$ are collinear. (See Figure~\ref{figure: dfbn}.)
\end{enumerate}
\end{prop}

\begin{proof}
Let $M: \Z^2 \to \mathcal{M}^n$ be the lift of $P$. For each $j\in \Z$, let $H_j \subset \R \mathrm{P}^{n+1}$ be the $2$-plane such that $\sigma(H_j \cap \mathcal{M}^n) = \mathcal{H}_j$. For each $i\in \Z$, let $V_i \subset \R \mathrm{P}^{n+1}$ be the $2$-plane such that $\sigma(V_i \cap \mathcal{M}^n) = \mathcal{V}_i$. By the case $m=2$ and $d=2$ of Lemma~\ref{lem: mxZstripisoquadric}, the lines $\{M_{i,1} \vee M_{i,2}\}_{i\in\Z}$ are isotropic lines of a $2$-dimensional quadric $\mathcal{Q}$ that is contained in the $3$-dimensional projective subspace $H_1 \vee H_2$ of $\R\mathrm{P}^{n+1}$. Generically, the quadric $\mathcal{Q}$ has signature $(++--)$.

Let $n=3$. Generically, the point $N$ is not contained in the hyperplane $H_1 \vee H_2 \subset \R\mathrm{P}^4$. Then, the projection $\sigma: H_1 \vee H_2 \to \overline{\R^3}$ with centre $N$ is a projective transformation. The signature of $\mathcal{Q}$ is preserved under any projective transformation. So, $\sigma(\mathcal{Q})$ has signature $(++--)$. The quadric  $\sigma(\mathcal{Q})$ is the required quadric $\mathcal{R}$.

Let $n=2$. Let $N^{\perp_\mathcal{Q}}$ be the polar plane of $N$ relative to $\mathcal{Q}$. Generically, $N$ is not contained in $\mathcal{Q}$. Then, $N^{\perp_\mathcal{Q}} \cap \mathcal{Q}$ is a conic of signature $(++-)$. Thus, $\sigma (N^{\perp_\mathcal{Q}} \cap \mathcal{Q})$ is a conic in $\overline{\R^2}$ of signature $(++-)$. The isotropic lines $\{M_{i,1} \vee M_{i,2}\}_{i\in\Z}$ project under $\sigma$ to tangent lines of the conic $\sigma (N^{\perp_\mathcal{Q}} \cap \mathcal{M}^2)$. In other words, the lines $\{\mathcal{V}_i\}_{i\in \Z}$ are tangent lines of $\sigma (N^{\perp_\mathcal{Q}} \cap \mathcal{Q})$. For $j\in [2]$, $H_j$ intersects $N^{\perp_\mathcal{Q}}$ in a line that intersects $\mathcal{M}^2$ is a pair of (possibly imaginary) points that are contained in $\mathcal{Q}$ since $\mathcal{Q}$ contains $H_1 \cap \mathcal{M}^2$ and $H_2 \cap \mathcal{M}^2$. So, for $j\in [2]$, the circle $\mathcal{H}_j$ is tangent to the conic $\mathcal{R}:= \sigma (N^{\perp_\mathcal{Q}} \cap \mathcal{Q})$ at the two points $\sigma (H_j \cap N^{\perp_{\mathcal{Q}}}\cap \mathcal{M}^2)$. Similarly, all of the circles $\{\mathcal{H}_j\}_{j\in \Z}$ are twice tangent to the conic $\mathcal{R}$. Generically, $X:= H_1 \cap H_2 \cap H_3$ is a point and $N= \cap_{i\in \Z}V_i$. By Corollary~\ref{cor: mxmgridthm}, $X$ and $N$ are conjugate relative to $\mathcal{M}^2$. So, the poles $H_1^\perp, H_2^\perp$ and $H_3^\perp$ are points that are contained in the polar plane $X^{\perp_{\mathcal{M}^2}}$ that contains $N$. Thus, $\sigma(H_1^\perp), \sigma(H_2^\perp)$ and $\sigma(H_3^\perp)$ are collinear points in $\R^3$. Then, Lemma~\ref{lem: stereoproj} implies that the centres of the circles $\mathcal{H}_1$, $\mathcal{H}_2$ and $\mathcal{H}_3$ are collinear. It follows that the centres of the circles $\{\mathcal{H}_j\}_{j\in \Z}$ are collinear.
\end{proof}

\begin{figure}[htbp]
\begin{center}
  \begin{subfigure}[t]{0.65\textwidth}
   \includegraphics[width=\textwidth]{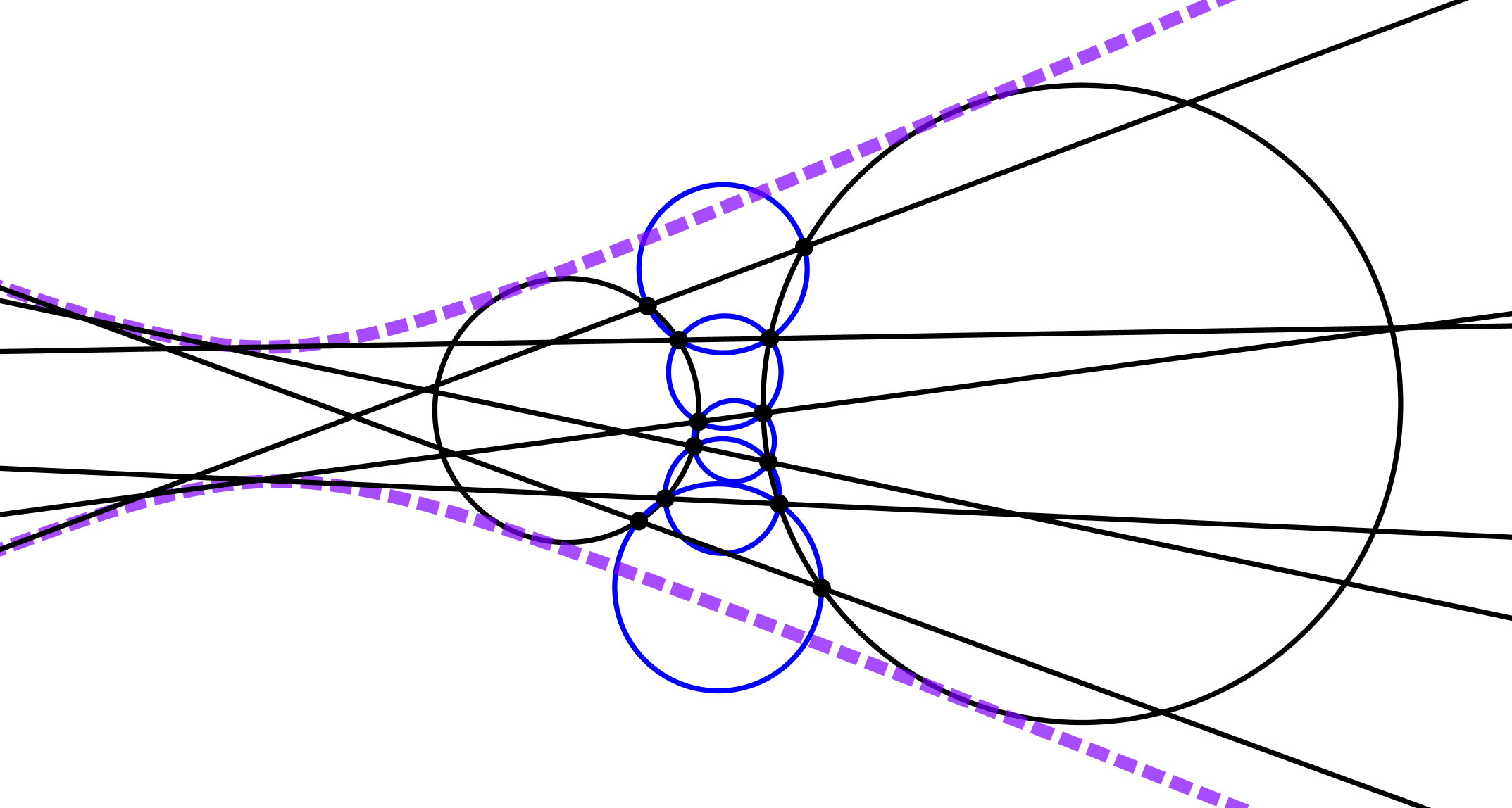}
 \caption{The six lines are tangent to the dashed hyperbola. The two black circles are twice tangent to the hyperbola.}
 \end{subfigure}
\hspace{0.7cm}%

  \begin{subfigure}[t]{0.4\textwidth}
   \includegraphics[width=\textwidth]{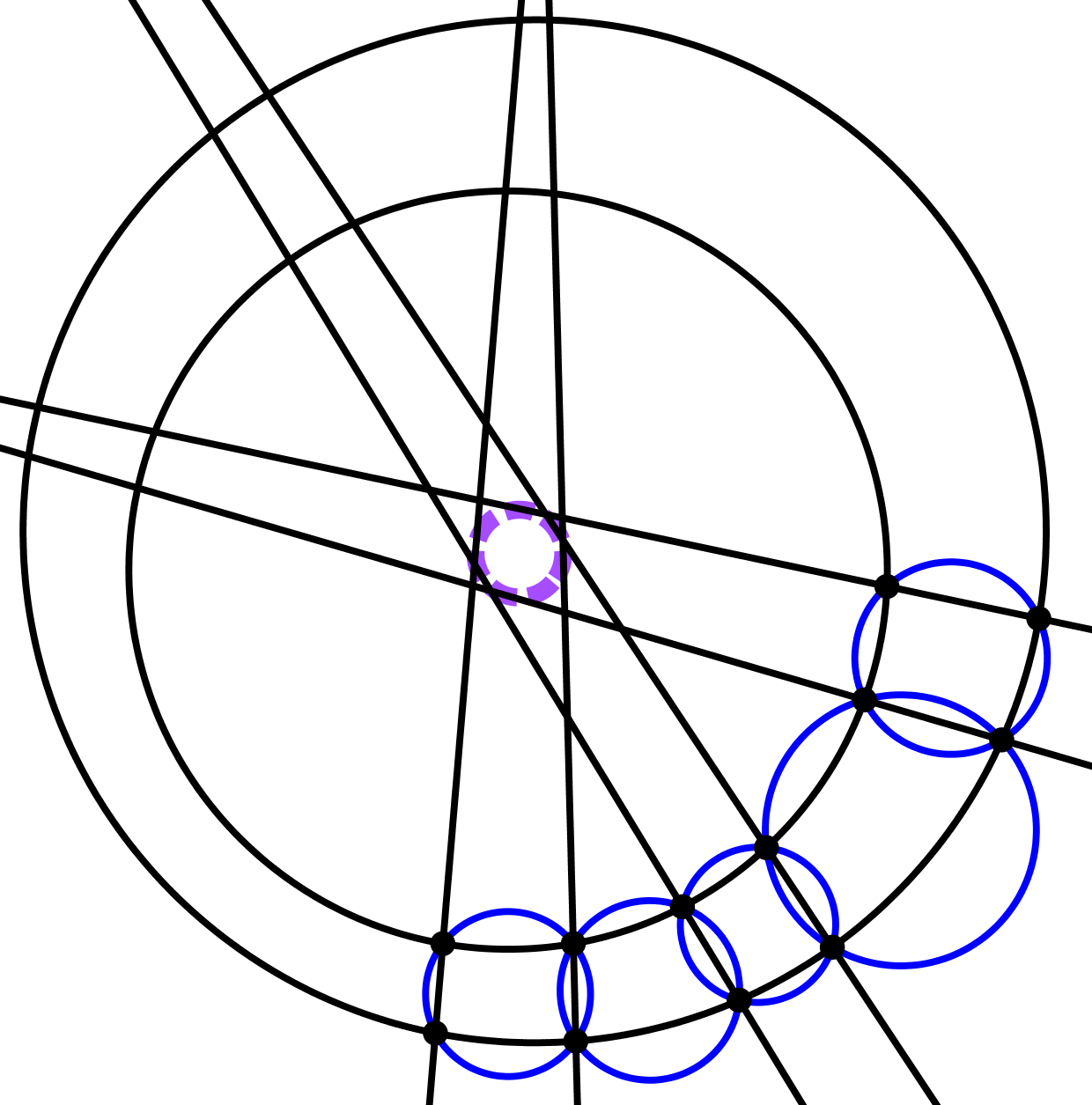}
 \caption{The six lines are tangent to the dashed ellipse. The two black circles have imaginary double contact with the ellipse.}
 \end{subfigure}
\hspace{0.7cm}%
  \begin{subfigure}[t]{0.4\textwidth}
    \includegraphics[width=\textwidth]{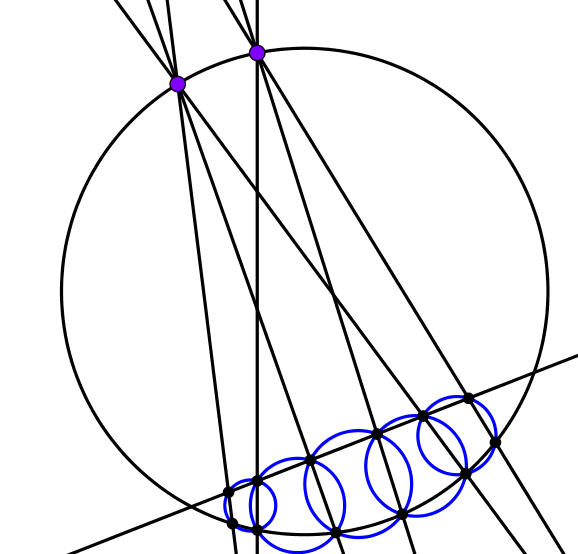}
 \caption{The six lines alternate between two systems of concurrent lines.}
    \end{subfigure}
   \end{center}
  \caption{Three examples of circular nets $P: [6]\times [2] \to \R^2$ with circular parameter lines $\{(P_{i,j})_{i \in [6]}\}_{j\in [2]}$. In case (a), the six lines are tangent to a hyperbola. In case (b), the six lines are tangent to an ellipse. In the case (c), one of the circular parameter lines is a straight line. The six lines are tangent to a degenerate conic.} 
  \label{figure: dfbn}
\end{figure}

As shown in Figure~\ref{figure: dfbn} (b), it is possible that some of the circles $\{\mathcal{H}_j\}_{j \in \Z}$ in the case $n=2$ of Proposition~\ref{prop: linearcircularquadricconcic} have imaginary double contact with the conic $\mathcal{R}$. The conic $\mathcal{R}$ degenerates to a pair of points if the point $N$ is contained in the quadric $\mathcal{Q}$ that is defined in the proof of the case $n=2$ of Proposition~\ref{prop: linearcircularquadricconcic}. An example is shown in Figure~\ref{figure: dfbn} (c).

Developable surfaces are ruled surfaces in $\R^3$ such that the rulings are curvature lines. Joachimsthal's theorem implies that any developable surface in $\R^3$ with one family of circular curvature lines must be a cone (or cylinder) of revolution. However, the circular net $P: \Z^2 \to \R^3$ in the case $n=3$ of Proposition~\ref{prop: linearcircularquadricconcic} is generically contained in a non-degenerate quadric rather than a cone of revolution. 

Let $P: \Z^2 \to \mathbb{R}^3$ be a circular net with linear parameter lines $\{\mathcal{V}_i\}_{i\in \Z}$ and circular parameter lines $\{\mathcal{H}_j\}_{j\in \Z}$. By Proposition~\ref{prop: linearcircularquadricconcic}, the lines  $\{\mathcal{V}_i\}_{i\in \Z}$ and the circles $\{\mathcal{H}_j\}_{j\in \Z}$ are generically contained in a non-degenerate quadric $\mathcal{R}$ of signature $(++--)$. By Proposition~\ref{prop: hyperbolicquadrictwoclasses}, the quadric $\mathcal{R}$ has two systems of $1$-dimensional generators. Any two generators of $\mathcal{R}$ are in different systems if and only if their intersection is $0$-dimensional, i.e. the intersection is a point. Thus, the generators $\{\mathcal{V}_{2i}\}_{i\in \Z}$ and $\{\mathcal{V}_{2i+1}\}_{i\in \Z}$ are in different systems of generators. The circular sections $\{\mathcal{H}_j\}_{j\in \Z}$ of the quadric $\mathcal{R}$ also have an alternation phenomenon. Proposition~\ref{prop: linearlinearparallel} implies that the circles $\{\mathcal{H}_{2j}\}_{j\in \Z}$ are contained in parallel planes and that the circles $\{\mathcal{H}_{2j+1}\}_{j\in \Z}$ are also contained in parallel planes. Generically, the planes of the circles $\{\mathcal{H}_{j}\}_{j\in \Z}$ are not parallel. For the circular net $P$ to admit a smooth limit to a developable surface with one family of circular curvature lines, one should remove both of the alternation phenomena. This can be achieved by requiring that the circles $\{\mathcal{H}_{j}\}_{j\in \Z}$ are contained in parallel planes and that $\mathcal{R}$ degenerates to a quadric of signature $(++-0)$. Then, the quadric $\mathcal{R}$ must be a quadric of revolution. This is consistent with the smooth theory. 

\subsection{Both Families Linear}

\begin{prop}\label{thm: linearlinear}

Consider a function $P: [3]\times [3] \to \R^2$.  For each $j \in [3]$, let the points $\{P_{i,j}\}_{i\in [3]}$ be contained in a line $\mathcal{H}_j$. For each $i \in [3]$, let the points $\{P_{i,j}\}_{j\in [3]}$ be contained in a line $\mathcal{V}_i$. Suppose that $\forall (i,j) \neq (2,2)$ the quad $\Box_{i,j}$ is circular. Then, the quad $\Box_{2,2}$ is also circular. 
\end{prop}

The proof of Proposition~\ref{thm: linearlinear} is analogous to the proof of Proposition~\ref{prop: 4x4planarplanar}.

Proposition~\ref{thm: linearlinear} can be used iteratively to construct circular nets $\Z^2 \to \R^2$ such that both families of parameter lines are straight lines.  The following Proposition~\ref{prop: linearlinearparallel} implies that, for both families of parameter lines, the straight lines alternate between two classes of parallel lines. It is the simplest instance of the alternation phenomena that appears throughout Section~\ref{section: Mobiusapproach}.

\begin{prop}\label{prop: linearlinearparallel}
Let $P: \Z \times [2] \to \R^2$ be a circular net. For all $j \in [2]$, let the points $\{P_{i,j}\}_{i\in \Z}$ be contained in a line $\mathcal{H}_j$. For all $i \in \Z$, let $\mathcal{V}_i$ be the line $P_{i,1} \vee P_{i,2}$. The lines $\{\mathcal{V}_{2i}\}_{i\in \Z}$ are parallel and the lines $\{\mathcal{V}_{2i + 1}\}_{i\in \Z}$ are parallel.
\end{prop}

Proposition~\ref{prop: linearlinearparallel} can be seen as a degenerate case of the geometry in Figure~\ref{figure: dfbn} (c).

\subsection{One Family Circular and the Other Family Spherical}\label{section: circularspherical}

\begin{prop}\label{prop: circularspherical5x4}
Consider a function $M: [5]\times [4] - \{(5,4)\} \to \mathcal{M}^n \subset \R\mathrm{P}^{n+1}$ where $n\in \{3,4\}$.  For each $j \in [3]$, let the points $\{M_{i,j}\}_{i\in [5]}$ be contained in a $2$-dimensional sphere $\mathcal{H}_j$. Let $\mathcal{H}_4$ be the $2$-dimensional sphere containing the points $M_{1,4}, M_{2,4}, M_{3,4}, M_{4,4}$. For each $i \in [4]$, let the points $\{M_{i,j}\}_{j\in [4]}$ be contained in a circle $\mathcal{V}_i$. Let $\mathcal{V}_5$ be the circle containing the points $M_{5,1}, M_{5,2}, M_{5,3}$. Suppose that $\forall (i,j) \neq (4,3)$ the quad $\Box_{i,j}$ is circular. There exists a unique point $M_{5,4}\in \mathcal{M}^n$ such that the quad $\Box_{4,3}$ is circular and such that $M_{5,4} \in \mathcal{V}_5\cap \mathcal{H}_4$.
\end{prop}

Proposition~\ref{prop: 4x4X1Dor2D} is used in the ensuing proof of Proposition~\ref{prop: circularspherical5x4}.

\begin{prop}\label{prop: 4x4X1Dor2D}
Let $P: [4]\times [4] \to \R \mathrm{P}^n$ be a Q-net, where $n \geq 4$. For each $j\in [4]$, let $H_j:= \mathrm{join}\{P_{i,j}\}_{i\in [4]}$ be $3$-dimensional. 
\begin{enumerate}[label=(\roman*)]
\item If $\cap_{j \in [4]} H_j$ is $1$-dimensional, then $\mathcal{L}^2_A P$ is Laplace degenerate. 
\item If $\cap_{j \in [4]} H_j$ is $2$-dimensional, then $\mathcal{L}_A P$ is Laplace degenerate.
\end{enumerate}
\end{prop}

The proof of Proposition~\ref{prop: 4x4X1Dor2D} is analogous to the proof of Proposition~\ref{prop: 3x3X1D}.

\begin{proof}[Proof of Proposition~\ref{prop: circularspherical5x4}]
Let $V_i$ be the $2$-dimensional projective subspace of $\R\mathrm{P}^{n+1}$ such that $V_i \cap \mathcal{M}^n = \mathcal{V}_i$. Let $H_j$ be the $3$-dimensional projective subspace of $\R\mathrm{P}^{n+1}$ such that $H_j \cap \mathcal{M}^n = \mathcal{H}_j$. Define the point $M_{5,4}:= V_5\cap H_4 \cap (M_{4,3} \vee M_{5,3} \vee M_{4,4})$ and the following spaces.
\begin{align*}
X&:= H_1 \cap H_2 \cap H_3\cap H_4 &X_1 &:= \bigcap_{1\leq j \leq 4} \mathrm{join}\{M_{i,j}\}_{1\leq i \leq 3} \\ 
X_2 &:= \bigcap_{1\leq j \leq 4} \mathrm{join}\{M_{i,j}\}_{2\leq i \leq 4} 
&X_3 &:= \bigcap_{1\leq j \leq 4} \mathrm{join}\{M_{i,j}\}_{3\leq i \leq 5}
\end{align*}
By the case $m=3$ and $d=2$ of Lemma~\ref{lem: mxZstripisoquadric}, the planes $\{M_{1,j} \vee M_{2,j} \vee M_{3,j}\}_{j\in [4]}$ are isotropic planes of a quadric of signature $(++--0)$. Then, the planes $\{M_{1,j} \vee M_{2,j} \vee M_{3,j}\}_{j\in [4]}$ are concurrent at the apex of the cone. Equivalently, $X_1$ is a point. Similarly, $X_2$ is a point. The points $X_1$ and $X_2$ are distinct and they are contained in $X$. So, $X$ is generically $1$-dimensional. By Proposition~\ref{prop: 4x4X1Dor2D}, $\mathcal{L}^2_A M$ is Laplace degenerate. By Proposition~\ref{prop: Laplaceterminatemsteps}, $X_3$ is $0$-dimensional. By applying Corollary~\ref{cor: mxmgridthm} to the points $\{M_{i,j} \mid i\in \{3,4,5\}, j\in [3]\}$, the points $X_3$ and $V_3\cap V_4\cap V_5$ are conjugate relative to $\mathcal{M}^n$. Therefore, by applying Corollary~\ref{cor: mxmgridthm} to the points $\{M_{i,j} \mid  i\in \{3,4,5\}, j\in \{2,3,4\}\}$, the point $M_{5,4}$ is contained in $\mathcal{M}^n$.
\end{proof}

In Proposition~\ref{prop: circularspherical5x4}, it is superfluous to consider the case $n > 4$. The vertices $\{M_{i,j}\mid (i,j) \in [5]\times [4]\}$ span at most a $5$-dimensional projective subspace of $\R \mathrm{P}^{n+1}$. So, the circular net $M$ must be contained in a $4$-sphere if $n>4$.

Proposition~\ref{prop: circularspherical5x4} can be used iteratively to construct circular nets $\Z^2 \to \mathcal{M}^n$, where $n\in \{3,4\}$, such that the parameter lines are circular in one family and $2$-spherical in the other family. 

Lemma~\ref{lem: isotropiclinepencil} is used in the proof of Proposition~\ref{prop: circularsphericalR3R4}. Lemma~\ref{lem: isotropiclinepencil} is almost equivalent to Lemma~\ref{lem: 2pointsandpencilofquadrics}. However, Lemma~\ref{lem: isotropiclinepencil} allows for $\mathcal{Q}_1 \cap l = \mathcal{Q}_2\cap l$ to be empty.

\begin{lem}\label{lem: isotropiclinepencil}
Let $\mathcal{Q}_1 = \{[x] \in \R \mathrm{P}^n \mid \p_1(x,x)=0\}$ and $\mathcal{Q}_2 = \{[x] \in \R \mathrm{P}^n \mid \p_2 (x,x)=0\}$ be distinct quadric hypersurfaces in $\R \mathrm{P}^n$. Let $l$ be a line that is not an isotropic line of $\mathcal{Q}_1$ and $\mathcal{Q}_2$. Suppose that $\exists \mu \in \R$ such that $\mu \p_1(x,x) = \p_2 (x,x)$ for all $[x] \in l$. So, $\mathcal{Q}_1 \cap l = \mathcal{Q}_2\cap l$. Then, there is a unique quadric $\mathcal{Q} = \{[x] \in \R\mathrm{P}^n \mid \p(x,x) = 0\}$ that contains the line $l$ and that belongs to the pencil of quadrics spanned by $\mathcal{Q}_1$ and $\mathcal{Q}_2$.
\end{lem}

\begin{proof}
Let $P= [p] \in l$ be a point that is not contained in $\mathcal{Q}_1$ and $\mathcal{Q}_2$. There is a unique quadric $\mathcal{Q}:= \{[x] \in \R\mathrm{P}^n \mid \p(x,x) =0\}$ that contains $P$ and that belongs to the pencil spanned by $\mathcal{Q}_1$ and $\mathcal{Q}_2$. Let $\p = \l_1 \p_1 + \l_2 \p_2$, where $\l_1, \l_2 \in \R$. So, $\l_1\p_1(p,p) + \l_2 \p_2 (p,p)=0$ because $\mathcal{Q}$ contains $P= [p]$. Without loss of generality, the representative symmetric bilinear form  $\p_1$ can be chosen such that $\p_1(x,x) = \p_2 (x,x)$ for all $[x] \in l$. Then, $\l_1 \p_1(p,p) + \l_2 \p_1(p,p) = 0$. So, $\l_2 = -\l_1$. Let $[u]$ be any point in the line $l$. Then, $\p(u,u) := \l_1 \p_1(u,u) + \l_2 \p_2(u,u)$ equals $0$ because $\l_2 = -\l_1$ and $\p_1(u,u) =\p_2(u,u)$ for any $[u]\in l$. Therefore, $\mathcal{Q}$ contains the line $l$.
\end{proof}

\begin{prop}\label{prop: circularsphericalR3R4}
Let $M: \Z^2 \to \mathcal{M}^n\subset \R \mathrm{P}^{n+1}$, where $n\in \{3,4\}$, be a circular net. For all $i\in \Z$, let the points $\{M_{i,j}\}_{j\in \Z}$ be contained in a circle $\mathcal{V}_i$. Let $V_i$ be the $2$-dimensional projective subspace of $\R \mathrm{P}^{n+1}$ such that $V_i \cap \mathcal{M}^n=\mathcal{V}_i$. For all $j \in \Z$, let the points $\{M_{i,j}\}_{i\in \Z}$ be contained in a $2$-sphere $\mathcal{H}_j$. Let $H_j$ be the $3$-dimensional projective subspace of $\R \mathrm{P}^{n+1}$ such that $H_j \cap \mathcal{M}^n=\mathcal{H}_j$.
\begin{enumerate}[label=(\roman*)]
\item
Let $n=4$. The $2$-spheres $\{\mathcal{H}_j\}_{j\in \Z}$ and the circles $\{\mathcal{V}_i\}_{i \in \Z}$ are contained in a $3$-dimensional Darboux cyclide  $\mathcal{D}$. Generically, $\{H_j\}_{j\in \Z}$ are $3$-dimensional isotropic subspaces of a quadric $\mathcal{H}$ of signature $(++--00)$ and $\{V_i\}_{i\in \Z}$ are $2$-dimensional isotropic subspaces of a quadric $\mathcal{V}$ of signature $(+++---)$. The quadrics $\mathcal{H}$, $\mathcal{V}$ and $\mathcal{M}^4$ belong to a pencil of quadrics. 
\item
Let $n=3$. The $2$-spheres $\{\mathcal{H}_j\}_{j\in \Z}$ and the circles $\{\mathcal{V}_i\}_{i\in \Z}$ envelop a $2$-dimensional Darboux cyclide $\mathcal{D}$. Each sphere $\mathcal{H}_j$ touches $\mathcal{D}$ along the points of a circle and each circle $\mathcal{V}_i$ is twice tangent to $\mathcal{D}$. Generically, the planes $\{V_i\}_{i\in \Z}$ are tangent along isotropic lines of a quadric $\mathcal{V}$ of signature $(+++--)$ and the spaces $\{H_j\}_{j\in \Z}$ are tangent along isotropic planes of a quadric $\mathcal{H}$ of signature $(++-00)$. The quadrics $\mathcal{H}$, $\mathcal{V}$ and $\mathcal{M}^3$ belong to a pencil of quadrics.
\end{enumerate}
\end{prop}

\begin{proof}
Let $n=4$. By the case $m=4$ and $d=2$ of Lemma~\ref{lem: mxZstripisoquadric}, the spaces $\{H_j\}_{j\in \Z}$ are $3$-dimensional isotropic spaces of a $4$-dimensional quadric $\mathcal{H}$ in $\R \mathrm{P}^5$. Generically, $\mathcal{H}$ has signature $(++--00)$. The intersection $\mathcal{D} := \mathcal{H}\cap\mathcal{M}^4$ is a $3$-dimensional Darboux cyclide that contains the $2$-spheres $\{\mathcal{H}_j\}_{j\in \Z}$. For any $2$-plane in $\R\mathrm{P}^5$, its intersection with $\mathcal{M}^4$ is a circle that is either fully contained in $\mathcal{D}$ or else it is a circle that contains at most four points of $\mathcal{D}$. For each $i\in \Z$, the points $\{M_{i,j}\}_{j\in \Z}$ are contained in $\mathcal{D}$. Therefore, the circles $\{\mathcal{V}_i\}_{i\in \Z}$ are contained in $\mathcal{D}$. By the case $m=3$ and $d=3$ of Lemma~\ref{lem: mxZstripisoquadric}, the spaces $\{V_i\}_{i\in \Z}$ are $2$-dimensional isotropic spaces of a $4$-dimensional quadric $\mathcal{V}$ in $\R \mathrm{P}^5$. Generically, $\mathcal{V}$ has signature $(+++---)$. The quadrics $\mathcal{M}^4$, $\mathcal{V}$ and $\mathcal{H}$ belong to a pencil of quadrics because $\mathcal{M}^4 \cap \mathcal{V} = \mathcal{M}^4 \cap \mathcal{H}$.

Let $n=3$. As shown in the proof of Proposition~\ref{prop: circularspherical5x4}, $X:= \cap_{j\in \Z}H_j$ is generically $1$-dimensional. By the case $m=3$ and $d=2$ of Lemma~\ref{lem: mxZstripisoquadric}, for each $i\in \Z$, $\{M_{i,j}\vee M_{i+1,j}\vee M_{i+2,j}\}_{j\in \Z}$ are $2$-dimensional isotropic spaces of a $3$-dimensional quadric $\mathcal{Q}_i$ in $\R \mathrm{P}^4$. Consider $\mathcal{Q}_1$. Generically, $\mathcal{Q}_1$ has signature $(++--0)$. The apex of $\mathcal{Q}_1$ is contained in the line $X$. Let $Z_1$ be any point in the line $X$ that is distinct from the apex of $\mathcal{Q}_1$. The point $Z_1$ is not contained in $\mathcal{Q}_1$. Then, $Z_1^{\perp_{\mathcal{Q}_1}}\cap \mathcal{Q}_1$ is a quadric of signature $(++-0)$, where $Z_1^{\perp_{\mathcal{Q}_1}}$ denotes the polar hyperplane of $Z_1$ with respect to $\mathcal{Q}_1$. Let $\mathcal{H}$ be the join of the quadric $Z_1^{\perp_{\mathcal{Q}_1}}\cap \mathcal{Q}_1$ with the point $Z_1$. $\mathcal{H}$ has signature $(++-00)$. For each $j\in \Z$, the hyperplane $H_j$ can be described as the join of the point $Z_1$ with the plane $M_{1,j}\vee M_{2,j} \vee M_{3,j}$.  The $2$-plane $M_{1,j}\vee M_{2,j} \vee M_{3,j}$ intersects the $3$-plane $Z_1^{\perp_{\mathcal{Q}_1}}$ in a line, say $l_j$, which is an isotropic line of the quadric $Z_1^{\perp_{\mathcal{Q}_1}} \cap \mathcal{Q}_1$. Then, $l_j \vee Z_1$ is an isotropic plane of $\mathcal{H}$. For each $j \in \Z$, $H_j$ is tangent to $\mathcal{H}$ along the points of the isotropic plane $l_j \vee Z_1$. Let $\mathcal{D}:=\mathcal{H}\cap \mathcal{M}^3$, which is a $2$-dimensional Darboux cyclide. For each $j\in \Z$, the sphere $\mathcal{H}_j$ is tangent to $\mathcal{D}$ along the points of a (possibly imaginary) circle. Indeed, the tangency points of $\mathcal{H}_j$ are exactly the points in the intersection of $\mathcal{M}^3$ with the isotropic $2$-plane $l_j\vee Z_1$ of $\mathcal{H}$. 

For each $i\in [3]$, the $2$-plane $V_i$ intersects $Z_1^{\perp_{\mathcal{Q}_1}}$ in a $1$-dimensional line $k_i$ such that $k_i \cap \mathcal{M}^3$ is a pair of (possibly imaginary) points that are contained in $\mathcal{H}$ since $\mathcal{Q}_1$ contains the circles $\{\mathcal{V}_i\}_{i\in [3]}$. The two foregoing points are exactly the two tangency points of the circle $\mathcal{V}_i$ to the cyclide $\mathcal{D}$. The two (possibly imaginary) intersection points $k_1\cap \mathcal{M}^3$ are contained in $\mathcal{M}^3$ and also $\mathcal{H}$. By Lemma~\ref{lem: isotropiclinepencil}, there is a unique quadric, say $\mathcal{V}$, such that $k_1$ is an isotropic line of $\mathcal{V}$ and such that $\mathcal{V}$ belongs to the pencil of quadrics spanned by $\mathcal{H}$ and $\mathcal{M}^3$. Generically, $\mathcal{V}$ has signature $(+++--)$. Similarly, there is a unique quadric $\mathcal{W}$ in the pencil such that $k_2$ is an isotropic line of $\mathcal{W}$. The planes $V_1$ and $V_2$ span a $3$-dimensional space. So, $V_1 \vee V_2$ intersects the hyperplane $Z_1^{\perp_{\mathcal{Q}_1}}$ in a $2$-plane. So, the lines $k_1$ and $k_2$ are coplanar. There is a unique quadric in the pencil that contains the point $k_1 \cap k_2$. So, $\mathcal{V}= \mathcal{W}$. Similarly, the quadric $\mathcal{V}$ also contains the line $k_3$. Let $\mathcal{Q}_2$ be the quadric such that the spaces $\{M_{2,j}\vee M_{3,j}\vee M_{4,j}\}_{j\in \Z}$ are $2$-dimensional isotropic spaces of $\mathcal{Q}_2$. Generically, $\mathcal{Q}_2$ has signature $(++--0)$. Let $Z_2$ be any point in the line $X$ that is distinct from the apex of the quadric $\mathcal{Q}_2$. Repeating the same argument as above, it follows that the circles $\{\mathcal{V}_i\}_{i=2,3,4}$ are twice tangent to $\mathcal{D}$. Any circle can be twice tangent to $\mathcal{D}$ in at most two points. So, the line $Z_2^{\perp_{\mathcal{Q}_2}}\cap V_2$ coincides with $k_2$ and the line $Z_2^{\perp_{\mathcal{Q}_2}}\cap V_3$ coincides with $k_3$. Let $k_4 := Z_2^{\perp_{\mathcal{Q}_2}}\cap V_4$. The line $k_4$ is an isotropic line of $\mathcal{V}$. Iterating further, it follows that all of the circles $\{\mathcal{V}_i\}_{i\in \Z}$ are twice tangent to $\mathcal{D}$ and that the $2$-planes $\{V_i\}_{i\in \Z}$ are tangent to $\mathcal{V}$ along isotropic lines of $\mathcal{V}$. By construction, $\mathcal{V}$ is in the pencil of quadrics spanned by $\mathcal{H}$ and $\mathcal{M}^3$.
\end{proof}

In classical differential geometry, for any smooth surface in $\R^3$ with one family of circular curvature lines and one family of spherical curvature lines, the planes of the circular curvature lines are generically concurrent at a point in $\overline{\R^3}$ and the spheres containing the spherical curvature lines belong to a pencil of spheres, i.e.\ the spheres have a common (possibly imaginary) circle. For instance, see \cite{snyder1904developable, rouquet1882etude} for more information. However, in the case $n=3$ of Proposition~\ref{prop: circularsphericalR3R4}, the intersection $\cap_{j\in \Z}\mathcal{H}_j$ is a pair of (possibly imaginary) points rather than a (possibly imaginary) circle and the intersection of the planes containing the circles $\{\mathcal{V}_i\}_{i \in\Z}$ is empty rather than a point in $\overline{\R^3}$.

\subsection{One Family Linear and the Other Family Spherical}

\begin{prop}\label{thm: linearspherical5x3}
Consider a function $P: [5]\times [3]-\{(5,3)\} \to \R^n$ where $n\in \{3, 4\}$. For each $j \in [2]$, let the points $\{P_{i,j}\}_{i\in [5]}$ be contained in a $2$-dimensional sphere $\mathcal{H}_j$. Let $\mathcal{H}_3$ be the $2$-dimensional sphere containing the points $P_{1,3},P_{2,3},P_{3,3}, P_{4,3}$. For each $i \in [4]$, let the points $\{P_{i,j}\}_{j\in [3]}$ be contained in a line $\mathcal{V}_i$. Let $\mathcal{V}_5$ be the line $P_{5,1}\vee P_{5,2}$. Suppose that $\forall (i,j) \neq (4,2)$ the quad $\Box_{i,j}$ is circular. There exists a unique point $P_{5,3}$ such that the quad $\Box_{4,2}$ is circular and such that $P_{5,3}$ is contained in $\mathcal{V}_5 \cap \mathcal{H}_3$.
\end{prop}

\begin{proof}

Let $M: [5]\times [3]-\{(5,3)\} \to \mathcal{M}^n \subset \R \mathrm{P}^{n+1}$ be the lift of $P$. Let $V_i$ denote the $2$-dimensional projective subspace of $\R\mathrm{P}^{n+1}$ such that $\sigma (V_i \cap \mathcal{M}^n) = \mathcal{V}_i$. Let $H_j$ denote the $3$-dimensional projective subspace of $\R\mathrm{P}^{n+1}$ such that 
$\sigma (H_j \cap \mathcal{M}^n) = \mathcal{H}_j$. Define $M_{5,3}:= V_5\cap H_3 \cap (M_{4,2} \vee M_{4,3} \vee M_{5,2})$ and the following spaces.
\begin{align*}
X &:= H_1 \cap H_2 \cap H_3
&
X_1 &:= \bigcap_{1\leq j \leq 3} \mathrm{join}\{M_{i,j}\}_{1\leq i \leq 3}\\
X_2 &:=\bigcap_{1\leq j \leq 3} \mathrm{join}\{M_{i,j}\}_{2\leq i \leq 4} &  X_3 &:= \bigcap_{1\leq j \leq 3} \mathrm{join}\{M_{i,j}\}_{3\leq i \leq 5} 
\end{align*}
$X_1, X_2$ and $X_3$ are points that are contained in $X$, which is generically $1$-dimensional. The $2$-dimensional spaces $\{V_i\}_{ i \in [5]}$ contain $N$. By applying Corollary~\ref{cor: mxmgridthm} to the points $\{M_{i,j} \mid i\in [3], j\in [3]\}$, the points $X_1$ and $N$ are conjugate. By applying Corollary~\ref{cor: mxmgridthm} to the points $\{M_{i,j} \mid i\in \{2,3,4\}, j\in [3]\}$, the points $X_2$ and $N$ are conjugate. Therefore, the line $X = X_1 \vee X_2$ is contained in the polar hyperplane $N^{\perp}$. So, $X_3$ and $N$ are conjugate. By applying Corollary~\ref{cor: mxmgridthm} to the points $\{M_{i,j} \mid i\in \{3,4,5\}, j\in [3]\}$, the point $M_{5,3}$ is contained in $\mathcal{M}^n$ because $X_3$ and $N$ are conjugate. The point $\sigma (M_{5,3})$ is the required point $P_{5,3}$.
\end{proof}

Proposition~\ref{thm: linearspherical5x3} can be used iteratively to construct circular nets $\Z^2 \to \R^n$, where $n\in \{3,4\}$, such that the parameter lines are linear in one family and spherical in the other family.

\begin{prop}\label{prop: linearsphericalquadricofrevolution}
Let $P: \Z^2 \to \R^n$, where $n\in \{3,4\}$, be a circular net. For all $i\in \Z$, let the points $\{P_{i,j}\}_{j\in \Z}$ be contained in a line $\mathcal{V}_i$. For all $j \in \Z$, let the points $\{P_{i,j}\}_{i\in \Z}$ be contained in a $2$-dimensional sphere $\mathcal{H}_j$. 

\begin{enumerate}[label=(\roman*)]
\item
Let $n=4$. The lines $\{\mathcal{V}_i\}_{i\in \Z}$ are isotropic lines of a $3$-dimensional quadric $\mathcal{R}$ such that the $2$-spheres $\{\mathcal{H}_j\}_{j\in \Z}$ are contained in $\mathcal{R}$. The $2$-spheres $\{\mathcal{H}_{2j}\}_{j\in \Z}$ are contained in parallel $3$-planes of $\R^4$. The $2$-spheres $\{\mathcal{H}_{2j+1}\}_{j\in \Z}$ are contained in parallel $3$-planes of $\R^4$.
\item
Let $n=3$. The lines $\{\mathcal{V}_i\}_{i\in \Z}$ are tangent to a $2$-dimensional quadric $\mathcal{R}$ such that, for each $j\in \Z$, the sphere $\mathcal{H}_j$ touches $\mathcal{R}$ along the points of a circle. The centres of the spheres $\{\mathcal{H}_j\}_{j\in\Z}$ are collinear.
\end{enumerate}
\end{prop}

The proof of Proposition~\ref{prop: linearsphericalquadricofrevolution} is analogous to the proof of Proposition~\ref{prop: linearcircularquadricconcic}.

In non-generic cases, it can happen that the quadric $\mathcal{R}$ in the case $n=3$ of Proposition~\ref{prop: linearsphericalquadricofrevolution} degenerates to a conic. For example, as shown in Figure~\ref{figure: linearandsphericaldegenerate}, this happens if one of the parameter lines is planar rather than spherical.

\begin{figure}[htbp]
\[\includegraphics[width=0.5\textwidth]{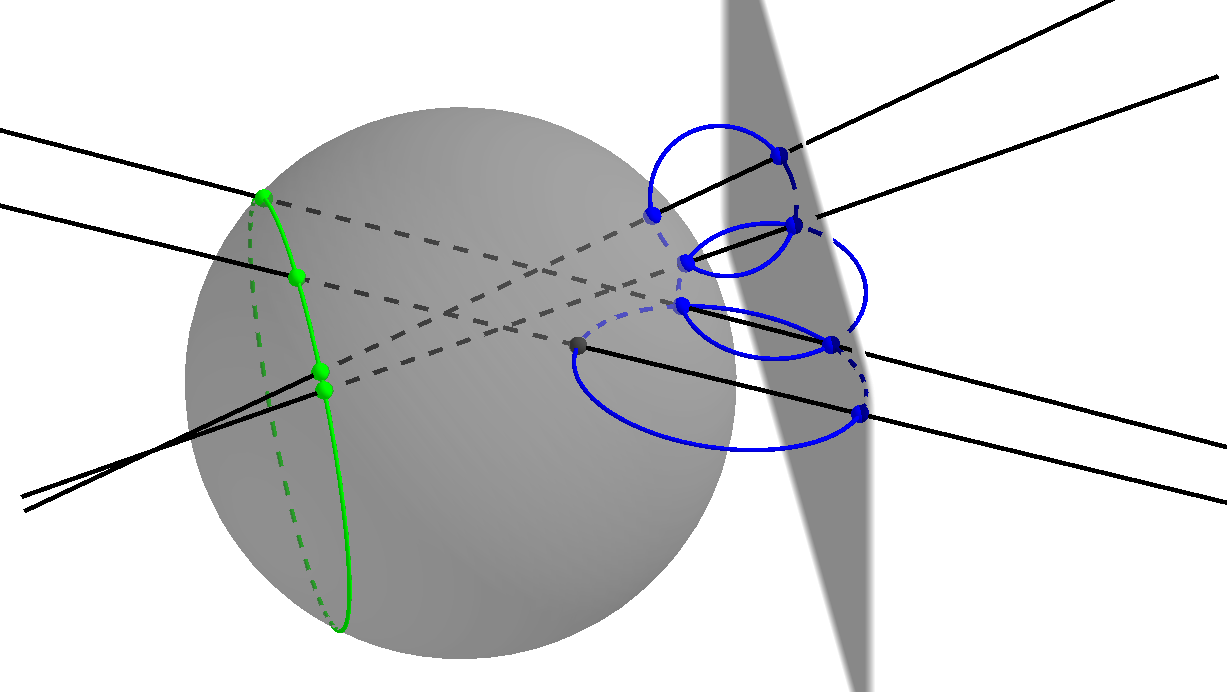}\]
\caption{A circular net $P: [4]\times [2] \to \R^3$. The points $\{P_{i,1}\}_{i \in [4]}$ are contained in a sphere. The points $\{P_{i,2}\}_{i\in [4]}$ are contained in a plane. The lines $\{P_{i,1}\vee P_{i,2}\}_{i \in [4]}$ intersect the sphere in four additional points that are coplanar.}
\label{figure: linearandsphericaldegenerate}
\end{figure}

In classical differential geometry, for any developable surface in $\R^3$ with spherical curvature lines, the spheres containing the spherical curvature lines are concentric \cite{Serret1853}. However, the spheres $\{\mathcal{H}_j\}_{j\in \Z}$ in the case $n=3$ of Proposition~\ref{prop: linearsphericalquadricofrevolution} are generically not concentric. Instead, they have collinear centres.

\subsection{One Family Circular and the Other Family Planar}

\begin{prop}\label{thm: circularplanar4x4}
Consider a function $P: [4]\times [4] - \{(4,4)\} \to \mathbb{R}^n$ where $n\in \{3,4\}$.  For each $j \in [3]$, let the points $\{P_{i,j}\}_{i\in [4]}$ be contained in a $2$-dimensional plane $\mathcal{H}_j$. Let $\mathcal{H}_4$ be the $2$-dimensional plane containing the points $P_{1,4}, P_{2,4}, P_{3,4}$. For each $i \in [3]$, let the points $\{P_{i,j}\}_{j\in [4]}$ be contained in a circle $\mathcal{V}_i$. Let $\mathcal{V}_4$ be the circle containing the points $P_{4,1}, P_{4,2}, P_{4,3}$. Suppose that $\forall (i,j) \neq (3,3)$ the quad $\Box_{i,j}$ is circular. Then, there exists a unique point $P_{4,4}\in \mathbb{R}^n$ such that the quad $\Box_{3,3}$ is circular and such that $P_{4,4} \in \mathcal{V}_4\cap \mathcal{H}_4$. 
\end{prop}

The proof of Proposition~\ref{thm: circularplanar4x4} is omitted. Proposition~\ref{thm: circularplanar4x4} can be seen as a degenerate case of Proposition~\ref{prop: 4x4planarplanar}, which is an incidence theorem for circular nets $P: [4]\times [4] \to \R^n$, $n\in \{3,4\}$, such that both families of parameter lines are $2$-planar. Suppose that the points $\{P_{1,1}, P_{1,2}, P_{1,3}, P_{1,4}\}$ are contained in a circle. Then, there is a unique $2$-sphere that contains the points $\{P_{1,1}, P_{1,2}$, $P_{1,3}, P_{1,4}\} \cup \{P_{2,1}, P_{2,2}, P_{2,3}, P_{2,4}\}$. The planarity of the points $\{P_{2,1}, P_{2,2}, P_{2,3}, P_{2,4}\}$ means that they are contained in a planar section of a $2$-sphere, i.e.\ a circle. Similarly, $\{P_{3,1}, P_{3,2}, P_{3,3}, P_{3,4}\}$ and $\{P_{4,1}, P_{4,2}, P_{4,3}, P_{4,4}\}$ are quadruples of points that are contained in circles because the circular net $P$ has $2$-planar parameter lines.

Proposition~\ref{thm: circularplanar4x4} can be used iteratively to construct circular nets $\Z^2 \to \R^n$, where $n\in \{3,4\}$, such that the parameter lines are circular in one family and $2$-planar in the other family.

\begin{prop}\label{prop: circularplanarR3R4}
Let $P: \Z^2 \to \mathbb{R}^n$, where $n\in \{3,4\}$, be a circular net. For all $i\in \Z$, let the points $\{P_{i,j}\}_{j\in \Z}$ be contained in a circle $\mathcal{V}_i$. For all $j \in \Z$, let the points $\{P_{i,j}\}_{i\in \Z}$ be contained in a $2$-dimensional plane $\mathcal{H}_j$. 
\begin{enumerate}[label=(\roman*)]
\item
Let $n=4$. The $2$-planes $\{\mathcal{H}_j\}_{j\in \Z}$ are isotropic planes of a $3$-dimensional quadratic cone $\mathcal{R}$. The circles $\{\mathcal{V}_i\}_{i \in \Z}$ are contained in the quadratic cone $\mathcal{R}$. 
\item
Let $n=3$. The $2$-planes $\{\mathcal{H}_j\}_{j\in \Z}$ are tangent along isotropic lines of a $2$-dimensional quadratic cone $\mathcal{R}$. The circles $\{\mathcal{V}_i\}_{i \in \Z}$ are twice tangent to the quadratic cone $\mathcal{R}$. 
\end{enumerate}
\end{prop}

The proof of Proposition~\ref{prop: circularplanarR3R4} is analogous to the proof of Proposition~\ref{prop: linearcircularquadricconcic}.

In classical differential geometry, if a smooth surface in $\R^3$ has one family of circular curvature lines and one family of planar curvature lines, then generically the planes of the planar curvature lines are concurrent at a line in $\R^3$ and the planes of the circular curvature lines are concurrent at a point in the plane at infinity of $\overline{\R^3}$. For instance, see \cite{rouquet1882etude} for more information. However, in the case $n=3$ of Proposition~\ref{prop: circularplanarR3R4}, the intersection $\cap_{j\in \Z}\mathcal{H}_j$ is generically a point rather than a line and the planes of the circles $\{\mathcal{V}_i\}_{i \in \Z}$ are generically not concurrent in $\overline{\R^3}$.

\subsection{One Family Linear and the Other Family Planar}

\begin{prop}\label{thm: linearplanar4x3}
Consider a function $P: [4]\times [3]-\{(4,3)\} \to \R^3$. For each $j \in [2]$, let the points $\{P_{i,j}\}_{i\in [4]}$ be contained in a plane $\mathcal{H}_j$. Let $\mathcal{H}_3$ be the plane containing the points $P_{1,3},P_{2,3},P_{3,3}$. For each $i \in [3]$, let the points $\{P_{i,j}\}_{j\in [3]}$ be contained in a line $\mathcal{V}_i$. Let $\mathcal{V}_4$ be the line $P_{4,1}\vee P_{4,2}$. Suppose that $\forall (i,j) \neq (3,2)$, the quad $\Box_{i,j}$
is circular. There exists a unique point $P_{4,3}$ such that the quad $\Box_{3,2}$ is circular and such that $P_{4,3}$ is contained in $\mathcal{V}_4 \cap \mathcal{H}_3$.
\end{prop}

The proof of Proposition~\ref{thm: linearplanar4x3} is analogous to the proof of Proposition~\ref{prop: 4x4planarplanar}.

Proposition~\ref{thm: linearplanar4x3} can be used iteratively to construct circular nets $\Z^2 \to \R^3$ such that the parameter lines are linear in one family and planar in the other family.

\begin{prop}\label{prop: linearplanarBonnetthereom}
Let $P: \Z^2 \to \R^3$ be a circular net. For all $i\in \Z$, let the points $\{P_{i,j}\}_{j\in \Z}$ be contained in a line $\mathcal{V}_i$. For all $j \in \Z$, let the points $\{P_{i,j}\}_{i\in \Z}$ be contained in a plane $\mathcal{H}_j$. The lines $\{\mathcal{V}_i\}_{i\in \Z}$ are tangent to a conic that is contained in the plane at infinity of $\overline{\R^3}$. The planes $\{\mathcal{H}_{2j}\}_{j\in \Z}$ are parallel and the planes $\{\mathcal{H}_{2j+1}\}_{j\in \Z}$ are parallel.
\end{prop}

\begin{proof}
Let $M: \Z^2 \to \mathcal{M}^3 \subset \R \mathrm{P}^4$ be the lift of $P$. Let $H_j$ be the $3$-dimensional projective subspace of $\R \mathrm{P}^4$ such that $\sigma(H_j \cap \mathcal{M}^3)=\mathcal{H}_j$. Let $V_i$ be the $2$-dimensional projective subspace of $\R \mathrm{P}^4$ such that $\sigma(V_i \cap \mathcal{M}^3)=\mathcal{V}_i$. By the case $m=2$ and $d=3$ of Lemma~\ref{lem: mxZstripisoquadric}, the lines $\{M_{i,1} \vee M_{i,2}\}_{i\in\Z}$ are isotropic lines of a $3$-dimensional quadric $\mathcal{Q}$. Generically, $\mathcal{Q}$ has signature $(+++--)$. The quadric $\mathcal{Q}$ contains $H_1 \cap \mathcal{M}^3$ and $H_2 \cap \mathcal{M}^3$, which both contain the point $N$ because $\mathcal{H}_1$ and $\mathcal{H}_2$ are planes in $\R^3$. So, the quadric $\mathcal{Q}$ contains the point $N$. Let $N^{\perp_\mathcal{Q}}$ be the polar hyperplane of $N$ relative to $\mathcal{Q}$. The quadric $\mathcal{R}:=N^{\perp_{\mathcal{Q}}}\cap \mathcal{Q}$ has signature $(++-0)$. It is a quadratic cone. Let $N^{\perp_{\mathcal{M}^3}}$ be the tangent hyperplane of $\mathcal{M}^3$ at $N$. The $2$-dimensional tangent space of the quadric $H_1 \cap \mathcal{M}^3$ at $N$ is contained in  $N^{\perp_{\mathcal{M}^3}}$. Similarly, the $2$-dimensional tangent space of the quadric $H_2 \cap \mathcal{M}^3$ at $N$ is contained in  $N^{\perp_{\mathcal{M}^3}}$. Generically, the two foregoing planes are distinct. Then, their span is the $3$-dimensional space $N^{\perp_{\mathcal{M}^3}}$. The $2$-spheres  $H_1 \cap \mathcal{M}^3$ and  $H_2 \cap \mathcal{M}^3$ are contained in $\mathcal{Q}$. Thus, $N^{\perp_\mathcal{Q}}$ equals $N^{\perp_{\mathcal{M}^3}}$. Under the central projection $\sigma$, $N^{\perp_{\mathcal{M}^3}}$ projects to the plane at infinity of $\overline{\R^3}$. So, the cone $\mathcal{R}$ is projected under $\sigma$ to a conic $\sigma(\mathcal{R})$ of signature $(++-)$ in the plane at infinity of $\overline{\R^3}$. Each isotropic line $M_{i,1}\vee M_{i,2}$ intersects the cone $\mathcal{R}$ in a unique point. So, each line $P_{i,1}\vee P_{i,2}$ intersects the conic $\sigma(\mathcal{R})$ in a unique point in the plane at infinity of $\overline{\R^3}$. Proposition~\ref{prop: linearlinearparallel} implies that, $\forall i \in \Z$, the lines $\{ P_{i,2j}\vee P_{i+1,2j}\}_{j \in \Z}$ are parallel and that the lines $\{ P_{i,2j+1}\vee P_{i+1,2j+1}\}_{j \in \Z}$ are parallel. Therefore, all of the planes $\{\mathcal{H}_{2j}\}_{j\in \Z}$ are parallel and all of the planes $\{\mathcal{H}_{2j+1}\}_{j\in \Z}$ are parallel.
\end{proof}

Proposition~\ref{prop: linearplanarBonnetthereom} can be seen as a degenerate case of the geometry that is shown in Figure~\ref{figure: linearandsphericaldegenerate}.

For any smooth surface in $\R^3$, Joachimsthal's theorem implies that a planar section of the surface is a curvature line if and only if the plane intersects the surface at a constant angle. For any developable surface in $\R^3$ with planar curvature lines, Joachimsthal's theorem implies that the planes of the planar curvature lines are parallel. However, in Proposition~\ref{prop: linearplanarBonnetthereom}, the planes $\{\mathcal{H}_j\}_{j\in \Z}$ are generically not parallel. Instead, the planes $\{\mathcal{H}_j\}_{j\in \Z}$ decompose into the two families $\{\mathcal{H}_{2j}\}_{j\in \Z}$ and $\{\mathcal{H}_{2j+1}\}_{j\in \Z}$ of parallel planes.

\subsection{Both Families Spherical}\label{section: sphericalspherical}

\begin{prop}\label{thm: sphericalspherical5x5}
Consider a function $M: [5]\times [5] - \{(5,5)\} \to \mathcal{M}^n \subset \R\mathrm{P}^{n+1}$ where $n\in \{3,4,5\}$.  For each $j \in [4]$, let the points $\{M_{i,j}\}_{i\in [5]}$ be contained in a $2$-sphere $\mathcal{H}_j$. Let $\mathcal{H}_5$ be the $2$-sphere containing the points $\{M_{i,5}\}_{i\in [4]}$. For each $i \in [4]$, let the points $\{M_{i,j}\}_{j\in [5]}$ be contained in a $2$-sphere $\mathcal{V}_i$. Let $\mathcal{V}_5$ be the $2$-sphere containing the points $\{M_{5,j}\}_{j\in [4]}$. Suppose that $\forall (i,j) \neq (4,4)$ the quad $\Box_{i,j}$ is circular. There exists a unique point $M_{5,5}\in \mathcal{M}^n$ such that the quad $\Box_{4,4}$ is circular and such that $M_{5,5} \in \mathcal{V}_5\cap \mathcal{H}_5$.
\end{prop}

\begin{proof}
Let $V_i$ and $H_j$ denote the $3$-dimensional projective subspaces of $\R\mathrm{P}^{n+1}$ such that $\mathcal{V}_i = V_i \cap \mathcal{M}^n$ and $\mathcal{H}_j = H_j \cap \mathcal{M}^n$. By Theorem~\ref{thm: mxmincidencethm}, the point $M_{5,5} := V_5\cap H_5 \cap (M_{4,4}\vee M_{4,5} \vee M_{5,4})$ is contained in $\mathcal{M}^n$.
\end{proof}

In Proposition~\ref{thm: sphericalspherical5x5}, it is superfluous to consider the case $n>5$. The vertices $\{M_{i,j} \mid (i,j) \in [5] \times [5]\}$ span at most a $6$-dimensional projective subspace of $\R \mathrm{P}^{n+1}$. So, the circular net $M: [5]\times [5] \to \mathcal{M}^n$ is contained in a $5$-sphere if $n>5$. 

Proposition~\ref{thm: sphericalspherical5x5} can be used iteratively to construct circular nets $M: \Z^2 \to \mathcal{M}^n$, where $n\in \{3,4,5\}$, such that both families of parameter lines are $2$-spherical.

\begin{prop}\label{prop: sphericalsphericalgrid}
Let $M: \Z^2 \to \mathcal{M}^n$, where $n\in \{3,4,5\}$, be a circular net. For all $i\in \Z$, let the points $\{M_{i,j}\}_{j\in \Z}$ be contained in a $2$-sphere $\mathcal{V}_i$. For all $j \in \Z$, let the points $\{M_{i,j}\}_{i\in \Z}$ be contained in a $2$-sphere $\mathcal{H}_j$.  
\begin{enumerate}[label=(\roman*)]
\item
Let $n=5$. There is a $4$-dimensional Darboux cyclide $\mathcal{D}$ that contains the spheres $\{\mathcal{V}_i\}_{i\in \Z}$ and $\{\mathcal{H}_j\}_{j\in \Z}$. 

\item
Let $n=4$. There is a $3$-dimensional Darboux cyclide $\mathcal{D}$ such that, for each $i \in\Z$, the sphere $\mathcal{V}_i$ touches $\mathcal{D}$ along the points of a circle and such that, for each $j \in\Z$, the sphere $\mathcal{H}_j$  touches $\mathcal{D}$ along the points of a circle. 
\item
Let $n=3$. There is a $2$-dimensional Darboux cyclide $\mathcal{D}$ such that, for each $i \in \Z$, the sphere $\mathcal{V}_i$ touches $\mathcal{D}$ in two points and such that, for each $j \in\Z$, the sphere $\mathcal{H}_j$ touches $\mathcal{D}$ in two points. 
\end{enumerate}
\end{prop}

\begin{proof}
Let $V_i$ and $H_j$ denote the $3$-dimensional projective subspaces of $\R \mathrm{P}^{n+1}$ such that $\mathcal{V}_i = V_i \cap \mathcal{M}^n$ and $\mathcal{H}_j = H_j \cap \mathcal{M}^n$.

Let $n=5$. By the case $m=3$ of Theorem~\ref{thm: pencilofquadrics}, the vertices of $M$ are generically contained in the base locus of a pencil of $5$-dimensional quadrics. The base locus is the required $4$-dimensional Darboux cyclide $\mathcal{D}$.

Let $n=4$. By the case $m=3$ and $d=3$ of Lemma~\ref{lem: mxZstripisoquadric}, the $2$-planes $\{M_{i,1}\vee M_{i,2} \vee M_{i,3} \}_{i\in \Z}$ are isotropic $2$-planes of a quadric, say $\mathcal{Q}$. By Theorem~\ref{thm: mxmincidencethm}, $Y:= \cap_{i\in \Z}V_i$ is a point. Generically, $\mathcal{Q}$ has signature $(+++---)$ and $Y$ is not contained in $\mathcal{Q}$. Let $Y^{\perp_{\mathcal{Q}}}$ be the polar hyperplane of $Y$ with respect to $\mathcal{Q}$. Then, $Y^{\perp_{\mathcal{Q}}} \cap \mathcal{Q}$ has signature $(+++--)$. The join of $Y^{\perp_{\mathcal{Q}}} \cap \mathcal{Q}$ with $Y$ is a quadric $\mathcal{V}$ of signature $(+++--0)$. For each $i \in \Z$, the $2$-plane $M_{i,1}\vee M_{i,2} \vee M_{i,3}$ intersects the hyperplane $Y^{\perp_{\mathcal{Q}}}$ in a line, say $k_i$, which is contained in $\mathcal{V}$. Then, $k_i \vee Y$ is an isotropic $2$-plane of $\mathcal{V}$. Each $3$-plane $V_i$ is tangent to $\mathcal{V}$ along the points of the $2$-plane $k_i\vee Y$. Let $\mathcal{D}:=\mathcal{V}\cap \mathcal{M}^4$. Then, $\mathcal{D}$ is a $3$-dimensional Darboux cyclide such that, for each $i \in \Z$, the sphere $\mathcal{V}_i$ touches $\mathcal{D}$ along the points of the (possibly imaginary) circle $\mathcal{M}^4 \cap (k_i \vee Y)$. For each $j\in [3]$, the $3$-plane $H_j$ intersects the hyperplane $Y^{\perp_{\mathcal{Q}}}$ in a $2$-plane that intersects $\mathcal{M}^4$ in a (possibly imaginary) circle that is contained in $\mathcal{D}$ since $\mathcal{Q}$ contains $\{\mathcal{H}_j\}_{j\in [3]}$. The $2$-sphere $\mathcal{H}_j$ is tangent to $\mathcal{D}$ along the points of the foregoing circle. Therefore, the three $2$-spheres $\{\mathcal{H}_j\}_{j\in [3]}$ are tangent to $\mathcal{D}$ along circles. By the case $m=3$ and $d=3$ of Lemma~\ref{lem: mxZstripisoquadric}, for each $j\in \Z$, the $2$-planes $\{M_{i,j}\vee M_{i,j+1} \vee M_{i,j+1} \}_{i\in \Z}$ are isotropic planes of a quadric, say $\mathcal{Q}_j$. Repeating the same argument as for $\mathcal{Q}$ (which equals $\mathcal{Q}_1$), it follows that, for each $j\in \Z$, the three $2$-spheres $\mathcal{H}_j, \mathcal{H}_{j+1}, \mathcal{H}_{j+2}$ are tangent to $\mathcal{D}$ along circles of $\mathcal{D}$. Therefore, all the $2$-spheres $\{\mathcal{H}_j\}_{j\in \Z}$ are tangent to $\mathcal{D}$ along circles. 

Let $n=3$. By the case $n=3$ of Proposition~\ref{prop: circularsphericalR3R4}, the $2$-planes $\{M_{i,1}\vee M_{i,2} \vee M_{i,3} \}_{i\in \Z}$ are tangent along isotropic lines of a quadric, say $\mathcal{Q}$. Generically, $\mathcal{Q}$ has signature $(+++--)$. By Theorem~\ref{thm: mxmincidencethm}, $Y:= \cap_{i\in \Z}V_i$ is a point. Generically, $Y$ is not contained in $\mathcal{Q}$. Then, $Y^{\perp_{\mathcal{Q}}} \cap \mathcal{Q}$ has signature $(+++-)$ or $(++--)$. The join of $Y^{\perp_{\mathcal{Q}}} \cap \mathcal{Q}$ with $Y$ is a quadric $\mathcal{V}$ of signature $(+++-0)$ or $(++--0)$. For each $i \in \Z$, let $k_i$ be the line along which the $2$-plane $M_{i,1}\vee M_{i,2} \vee M_{i,3}$ is tangent to $\mathcal{Q}$. The line $k_i$ intersects the hyperplane $Y^{\perp_{\mathcal{Q}}}$ in a point, say $T_i$, that is contained in $\mathcal{V}$. Then, $T_i \vee Y$ is an isotropic line of $\mathcal{V}$. Each $3$-plane $V_i$ is tangent to $\mathcal{V}$ along the points of $T_i \vee Y$. Let $\mathcal{D}:= \mathcal{V}\cap \mathcal{M}^3$. Then, $\mathcal{D}$ is a $2$-dimensional Darboux cyclide such that, for each $i\in \Z$, the $2$-sphere $\mathcal{V}_i$ is tangent to $\mathcal{D}$ at the (possibly imaginary) points $(T_i\vee Y)\cap \mathcal{M}^3$. By the case $n=3$ of Proposition~\ref{prop: circularsphericalR3R4}, each of the three $3$-planes $\{H_j\}_{j\in [3]}$ has a $2$-dimensional subspace, say $E_j$, such that the (possibly imaginary) circle $E_j\cap \mathcal{M}^3$ is contained in $\mathcal{Q}$. For each $j \in [3]$, the $2$-plane $E_j$ intersects the hyperplane $Y^{\perp_{\mathcal{Q}}}$ in a line, say $l_j$. The line $l_j$ intersects $\mathcal{M}^3$ in two (possibly imaginary) points that are contained in $\mathcal{D}$. For each $j \in [3]$, the sphere $\mathcal{H}_j$ is tangent to $\mathcal{D}$ at the two foregoing points. For each $j\in \Z$, by the case $n=3$ of Proposition~\ref{prop: circularsphericalR3R4}, the $2$-planes $\{M_{i,j}\vee M_{i,j+1} \vee M_{i,j+2} \}_{i\in \Z}$ are tangent along isotropic lines of a quadric, say $\mathcal{Q}_j$. By repeating the same argument as for $\mathcal{Q}$ (which equals $\mathcal{Q}_1$), it follows that for each $j\in \Z$, the three $2$-spheres $\mathcal{H}_j, \mathcal{H}_{j+1}$ and $\mathcal{H}_{j+2}$ are twice tangent to $\mathcal{D}$. Thus, all of the $2$-spheres $\{\mathcal{H}_j\}_{j\in \Z}$ are twice tangent to $\mathcal{D}$. 
\end{proof}

\begin{cor}\label{cor: sphericalsphericalR3confocaldeferents}
Let $P: \Z^2 \to \R^{3}$ be a circular net. For all $i \in \Z$, let the points $\{P_{i,j}\}_{j\in \Z}$ be contained in a $2$-sphere $\mathcal{V}_i$. For all $j \in \Z$, let the points $\{P_{i,j}\}_{i\in \Z}$ be contained in a $2$-sphere $\mathcal{H}_j$. Generically, the following properties hold.
\begin{itemize}
\item The centres of the $2$-spheres $\{\mathcal{H}_j\}_{j\in\Z}$ are contained in a $2$-dimensional quadric $\mathcal{Q}_{\mathcal{H}}$.
\item The centres of the $2$-spheres $\{\mathcal{V}_i\}_{i\in \Z}$ are contained in a $2$-dimensional quadric $\mathcal{Q}_{\mathcal{V}}$.
\item The quadrics $\mathcal{Q}_{\mathcal{H}}$ and $\mathcal{Q}_{\mathcal{V}}$ are confocal.
\item The $2$-spheres $\{\mathcal{H}_j\}_{j\in\Z}$ have a common orthogonal $2$-sphere $\mathcal{S}_{\mathcal{H}}$.
\item The $2$-spheres $\{\mathcal{V}_i\}_{i\in\Z}$ have a common orthogonal $2$-sphere $\mathcal{S}_{\mathcal{V}}$.
\item The $2$-spheres $\mathcal{S}_{\mathcal{H}}$ and $\mathcal{S}_{\mathcal{V}}$ are orthogonal.
\end{itemize}
\end{cor}

\begin{proof}
By the case $n=3$ of Proposition~\ref{prop: sphericalsphericalgrid}, the spheres of the spherical parameter lines envelop a $2$-dimensional Darboux cyclide. The spheres $\{\mathcal{H}_j\}_{j\in \Z}$ provide one generation of the cyclide. The spheres $\{\mathcal{V}_i\}_{i\in \Z}$ provide another generation of the cyclide. So, by Theorem~\ref{thm: cyclide}, the above six properties are satisfied.
\end{proof}

In Corollary~\ref{cor: sphericalsphericalR3confocaldeferents}, it can happen that one of the spheres $\mathcal{S}_\mathcal{H}$ or $\mathcal{S}_\mathcal{V}$ is imaginary. However, they cannot both be imaginary.

For smooth surfaces in $\R^3$ such that both families of curvature lines are spherical (and not circular), the centres of the spheres of the spherical curvature lines are generically contained in a pair of orthogonal planes and each family of spheres has two common (possibly imaginary) points. For instance, see \cite{Serret1853, lemonnier1868surfaces} for more information. However, in the case $n=3$ of Proposition~\ref{prop: sphericalsphericalgrid}, the centres of the spheres are contained in a pair of confocal quadrics rather than a pair of orthogonal planes. The spheres $\{\mathcal{V}_i\}_{i\in\Z}$ (resp. $\{\mathcal{H}_j\}_{j\in\Z}$) have the common orthogonal sphere $\mathcal{S}_{\mathcal{V}}$ (resp. $\mathcal{S}_{\mathcal{H}}$) rather than two common points.

\subsection{One Family Planar and the Other Family Spherical}\label{section: planarandspherical}

\begin{prop}\label{thm: 5x4planarspherical}
Consider a function $P: [5]\times [4] - \{(5,4)\} \to \mathbb{R}^n$ where $n\in \{3,4,5\}$.  For each $j \in [3]$, let the points $\{P_{i,j}\}_{i\in [5]}$ be contained in a $2$-sphere $\mathcal{H}_j$. Let $\mathcal{H}_4$ be the $2$-sphere containing the points $P_{1,4}, P_{2,4}, P_{3,4}, P_{4,4}$. For each $i \in [4]$, let the points $\{P_{i,j}\}_{j\in [4]}$ be contained in $2$-plane $\mathcal{V}_i$. Let $\mathcal{V}_5$ be the $2$-plane containing the points $P_{5,1}, P_{5,2}, P_{5,3}$. Suppose that $\forall (i,j) \neq (4,3)$ the quad $\Box_{i,j}$
is circular. There exists a unique point $P_{5,4}\in \mathbb{R}^n$ such that the quad $\Box_{4,3}$ is circular and such that $P_{5,4} \in \mathcal{V}_5\cap \mathcal{H}_4$.  
\end{prop}

The proof of Proposition~\ref{thm: 5x4planarspherical} is analogous to the proof of Proposition~\ref{thm: linearcircular}.

\begin{prop}\label{prop: planarsphericalgrid}
Let $P: \Z^2 \to \mathbb{R}^n$, where $n\in \{3,4,5\}$, be a circular net. For all $i\in \Z$, let the points $\{P_{i,j}\}_{j\in \Z}$ be contained in a $2$-plane $\mathcal{V}_i$. For all $j \in \Z$, let the points $\{P_{i,j}\}_{i\in \Z}$ be contained in a $2$-sphere $\mathcal{H}_j$. 
\begin{enumerate}[label=(\roman*)]
\item
Let $n=5$. There is $4$-dimensional quadric $\mathcal{R}$ that contains the $2$-spheres $\{\mathcal{H}_j\}_{j\in \Z}$ and that also contains the $2$-planes $\{\mathcal{V}_i\}_{i\in \Z}$.
\item
Let $n=4$. There is a $3$-dimensional quadric $\mathcal{R}$ such that, $\forall j \in \Z$, the $2$-sphere $\mathcal{H}_j$ is tangent to $\mathcal{R}$ along a circle and such that, $\forall i \in \Z$, the plane $\mathcal{V}_i$ is tangent to $\mathcal{R}$ along an isotropic line. 
\item
Let $n=3$. There is a $2$-dimensional quadric $\mathcal{R}$ such that, $\forall j \in \Z$, the  $2$-sphere $\mathcal{H}_j$ is tangent to $\mathcal{R}$ at two points and such that, $\forall i \in \Z$, the plane $\mathcal{V}_i$ is tangent to $\mathcal{R}$ at one point. The centres of the spheres $\{\mathcal{H}_j\}_{j\in \Z}$ are coplanar.
\end{enumerate}
\end{prop}

The proof of Proposition~\ref{prop: planarsphericalgrid} is analogous to the proof of Proposition~\ref{prop: linearcircularquadricconcic}.

In classical differential geometry, for smooth surfaces in $\R^3$ such that one family of curvature lines is planar and the other family of curvature lines is spherical, the planes are generically concurrent at a point in $\overline{\R^3}$ and the centres of the spheres are collinear \cite{bonnet1853memoire, Serret1853}. However, in the case $n=3$ of Proposition~\ref{prop: planarsphericalgrid}, the planes $\{\mathcal{V}_i\}_{i\in\Z}$ are generically tangent to a non-degenerate quadric and they are not concurrent. Moreover, the centres of the spheres $\{\mathcal{H}_j\}_{j\in\Z}$ are generically coplanar rather than collinear. 

\subsection{Both Families Planar}\label{section: bothplanar}

\begin{prop}\label{prop: 4x4planarplanar}
Consider a function $P: [4]\times [4] - \{(4,4)\} \to \mathbb{R}^n$ where $n \in \{3,4\}$.  For each $j \in [3]$, let the points $\{P_{i,j}\}_{i\in [4]}$ be contained in a $2$-plane $\mathcal{H}_j$. Let $\mathcal{H}_4$ be the $2$-plane containing the points $P_{1,4}, P_{2,4}, P_{3,4}$. For each $i \in [3]$, let the points $\{P_{i,j}\}_{j\in [4]}$ be contained in a $2$-plane $\mathcal{V}_i$. Let $\mathcal{V}_4$ be the $2$-plane containing the points $P_{4,1}, P_{4,2}, P_{4,3}$. Suppose that $\forall (i,j) \neq (3,3)$ the quad $\Box_{i,j}$ is circular. There exists a unique point $P_{4,4}\in \mathbb{R}^n$ such that the quad $\Box_{3,3}$ is circular and such that $P_{4,4} \in \mathcal{V}_4\cap \mathcal{H}_4$. 
\end{prop}
\begin{proof}
Let $M: [4] \times [4] - \{(4,4)\} \to \mathcal{M}^3 \subset \R \mathrm{P}^4$ be the lift of $P$. Let $V_i$ and $H_j$ be the $3$-dimensional projective subspaces of $\R\mathrm{P}^{n+1}$ such that $\sigma (V_i \cap \mathcal{M}^n) = \mathcal{V}_i$ and $\sigma (H_j \cap \mathcal{M}^n) = \mathcal{H}_j$. Generically, $\cap_{j\in [4]} H_{j}$ and $\cap_{i\in [4]} V_{i}$ are both $0$-dimensional. More precisely, $N= \cap_{j\in [4]} H_{j} = \cap_{i\in [4]} V_{i}$ because the spaces $V_i$ and $H_j$ contain $N$. Define $M_{4,4} := (M_{3,3}\vee M_{4,3} \vee M_{3,4}) \cap H_4 \cap V_4$. By Corollary~\ref{cor: mxmgridthm}, $M_{4,4} \in  \mathcal{M}^n$ because $N$ is conjugate to itself relative to $\mathcal{M}^n$. The point $\sigma (M_{4,4})$ is the required point $P_{4,4} \in \R^n$.
\end{proof}

In Proposition~\ref{prop: 4x4planarplanar}, it is superfluous to consider the case $n >4$. In the proof of Proposition~\ref{prop: 4x4planarplanar}, the 3-planes $V_1$ and $H_1$ both contain the distinct points $M_{1,1}$ and $N$. So, $V_1 \cap H_1$ is at least $1$-dimensional. Equivalently, $V_1 \vee H_1$ is at most $5$-dimensional. Because $M$ is a Q-net, the points $\{M_{i,j} \mid (i,j) \in [4]\times [4]\}$ are contained in $V_1 \vee H_1$. So, all the vertices of $M$ are contained in a $5$-dimensional space, i.e. they are contained in a $4$-dimensional sphere.

Proposition~\ref{prop: 4x4planarplanar} can be used iteratively to construct circular nets $\Z^2 \to \R^n$, $n \in \{3,4\}$, such that both families of parameter lines are $2$-planar.

\begin{prop}\label{prop: circularnetplanarplanar}
Let $P: \Z^2 \to \mathbb{R}^n$, where $n \in \{3,4\}$, be a circular net. For all $i\in \Z$, let the points $\{P_{i,j}\}_{j\in \Z}$ be contained in a $2$-plane $\mathcal{V}_i$. For all $j \in \Z$, let the points $\{P_{i,j}\}_{i\in \Z}$ be contained in a $2$-plane $\mathcal{H}_j$.
\begin{enumerate}[label=(\roman*)]
\item
Let $n=4$. The $2$-planes $\{\mathcal{H}_j\}_{j\in \Z}$ are tangent along isotropic lines of a $2$-dimensional quadric in the hyperplane at infinity of $\overline{\R^4}$. The $2$-planes $\{\mathcal{V}_i\}_{i\in \Z}$ are tangent along isotropic lines to a $2$-dimensional quadric in the hyperplane at infinity of $\overline{\R^4}$. 
\item
Let $n=3$. The $2$-planes $\{\mathcal{H}_j\}_{j\in \Z}$ are tangent to a conic in the plane at infinity of $\overline{\R^3}$. The $2$-planes $\{\mathcal{V}_i\}_{i\in \Z}$ are tangent to a conic in the plane at infinity of $\overline{\R^3}$. 
\end{enumerate}
\end{prop}

\begin{proof}
Let $M: \Z^2 \to \mathcal{M}^n \subset \R \mathrm{P}^{n+1}$ be the lift of $P$. Let $H_j$ and $V_i$ be the $3$-dimensional projective subspaces of $\R \mathrm{P}^{n+1}$ such that $\sigma(H_j \cap \mathcal{M}^n)=\mathcal{H}_j$ and $\sigma(V_i \cap \mathcal{M}^n)=\mathcal{V}_i$.

Let $n=4$. By the case $m=3$ and $d=3$ of Lemma~\ref{lem: mxZstripisoquadric}, the spaces $\{M_{1,j}\vee M_{2,j} \vee M_{3,j}\}_{j\in \Z}$ are $2$-dimensional isotropic spaces of a $4$-dimensional quadric $\mathcal{Q}$ in the space $V_1 \vee V_2 \vee V_3$ which is generically $5$-dimensional. Generically, $\mathcal{Q}$ has signature $(+++---)$. For each $i \in \Z$, $V_i$ contains $N$ because $\mathcal{V}_i=\sigma (V_i \cap \mathcal{M}^4)$ is a $2$-plane in $\R^4$.  So, $\mathcal{Q}$ contains $N$ because $\mathcal{Q}$ contains $\{V_i \cap \mathcal{M}^4\}_{i\in [3]}$. Then, $N^{\perp_\mathcal{Q}}\cap \mathcal{Q}$ is a $3$-dimensional quadric of signature $(++--0)$. For each $i \in [3]$, the $2$-dimensional tangent plane of the $2$-sphere $V_i \cap \mathcal{M}^4$ at $N$ is a tangent $2$-plane of $\mathcal{M}^4$ and it is also a tangent $2$-plane of $\mathcal{Q}$ because $V_i \cap \mathcal{M}^4$ is contained in $\mathcal{Q}$. So, $N^{\perp_\mathcal{Q}}$ equals $N^{\perp_{\mathcal{M}^4}}$, which is the tangent hyperplane of $\mathcal{M}^4$ at $N$. Then, $\sigma(N^{\perp_\mathcal{Q}}\cap \mathcal{Q})$ is a $2$-dimensional quadric of signature $(++--)$ that is contained in the hyperplane at infinity of $\overline{\R^4}$. For each $j\in \Z$, $(M_{i,j}\vee M_{2,j} \vee M_{3,j})\cap N^{\perp_{\mathcal{Q}}}$ is an isotropic line of $N^{\perp_{\mathcal{Q}}}\cap \mathcal{Q}$. The $2$-plane $\mathcal{H}_j$ is tangent to $\sigma(N^{\perp_{\mathcal{Q}}} \cap \mathcal{Q})$ along the isotropic line $\sigma((M_{1,j}\vee M_{2,j} \vee M_{3,j})\cap N^{\perp_{\mathcal{Q}}})$. Symmetrically, the planes $\{\mathcal{V}_i\}_{i\in \Z}$ are tangent along isotropic lines of a quadric of signature $(++--)$ in the hyperplane at infinity of $\overline{\R^4}$.  

Let $n=3$. By the case $n=3$ of Proposition~\ref{prop: circularsphericalR3R4}, the $2$-planes $\{M_{1,j}\vee M_{2,j} \vee M_{3,j}\}_{j\in \Z}$ are tangent to a quadric $\mathcal{Q}$ such that each $2$-plane $M_{1,j}\vee M_{2,j} \vee M_{3,j}$ is tangent to $\mathcal{Q}$ along the points of an isotropic line. Generically, $\mathcal{Q}$ has signature $(+++--)$. By Proposition~\ref{prop: circularspherical5x4}, the vertices $\{M_{i,j} \mid i\in [3], j\in \Z\}$ can be extended to a circular net $\til{M}: \Z^2\to \mathcal{M}^3$ such that, for each $j\in \Z$, the span of the points $\{\til{M}_{i,j}\}_{i\in \Z}$ is the $2$-plane $M_{1,j}\vee M_{2,j} \vee M_{3,j}$ and such that, for each $i\in \Z$, the span of the points $\{\til{M}_{i,j}\}_{j\in \Z}$ is $3$-dimensional. Then, by Proposition~\ref{prop: circularsphericalR3R4}, the $3$-planes $\{\mathrm{join}\{\til{M}_{i,j} \}_{j\in \Z}\}_{i \in \Z}$ are tangent along isotropic $2$-planes of a quadric, say $\mathcal{V}$. Generically, $\mathcal{V}$ has signature $(++-00)$. Then, $V_1\cap V_2\cap V_3$ is an isotropic line of $\mathcal{V}$. It is the line of singular points of $\mathcal{V}$. The line $V_1 \cap V_2\cap V_3$ intersects $\mathcal{M}^3$ in two points that are contained in $\mathcal{Q}$. Indeed, $\mathcal{Q}\cap\mathcal{V}$ equals $\mathcal{M}^3\cap \mathcal{V}$ because Proposition~\ref{prop: circularsphericalR3R4} ensures that the quadrics $\mathcal{Q}$, $\mathcal{V}$ and $\mathcal{M}^3$ belong to a pencil of quadrics. The line $V_1 \cap V_2\cap V_3$ contains $N$ because the $2$-spheres $\{V_{i}\cap \mathcal{M}^3\}_{i\in [3]}$ project stereographically to the planes $\{\mathcal{V}_{i}\}_{i\in [3]}$. Therefore, $N$ is contained in $\mathcal{Q}$. Then, $N^{\perp_{\mathcal{Q}}}\cap \mathcal{Q}$ is a quadric of signature $(++-0)$. Consider the pencil of quadrics containing $\mathcal{M}^3$, $\mathcal{V}$ and $\mathcal{Q}$. Any point in $V_1\cap V_2 \cap V_3$ is a singular point of $\mathcal{V}$. In particular, $N$ is a singular point of $\mathcal{V}$. It follows that, for any quadric in the foregoing pencil, its tangent space at $N$ must contain the tangent space $N^{\perp_{\mathcal{M}^3}}$. In particular, $N^{\perp_\mathcal{Q}}$ equals $N^{\perp_{\mathcal{M}^3}}$. The hyperplane $N^{\perp_{\mathcal{M}^3}}$ intersects $\overline{\R^3}$ in the plane at infinity. So, the quadratic cone $N^{\perp_{\mathcal{Q}}}\cap \mathcal{Q}$ intersects the plane at infinity of $\overline{\R^3}$ in a conic of signature $(++-)$. The $2$-planes $\{\mathcal{H}_j\}_{j\in \Z}$ are tangent to the foregoing conic. Symmetrically, the $2$-planes $\{\mathcal{V}_i\}_{i\in \Z}$ are tangent to a conic in the plane at infinity of $\overline{\R^3}$.
\end{proof}

For smooth surfaces in $\R^3$ such that both families of curvature lines are planar, the planes of one family of planar curvature lines are tangent to a cylinder and the planes of the second family of curvature lines are also tangent to a cylinder. In other words, each family of planes is concurrent at a point in the plane at infinity of $\overline{\R^3}$. The two cylinders have orthogonal directions. See \cite{Serret1853, bonnet1853memoire} for more information. The planes $\{\mathcal{H}_j\}_{j\in\Z}$ (resp. $\{\mathcal{V}_i\}_{i\in\Z}$) in the case $n=3$ of Proposition~\ref{prop: circularnetplanarplanar} are generically not concurrent at a point in the plane at infinity of $\overline{\R^3}$. Instead, they are tangent planes of a conic that is contained in the plane at infinity of $\overline{\R^3}$.

Recently, a discretisation of surfaces with planar curvature lines was introduced in \cite{Tellier2019}. It is a discretisation that is based on the classical fact that the Gauss map is a Combescure transformation of curvature-line parametrisations. The discretisation in \cite{Tellier2019} concerns circular nets with planar parameter lines that are more restrictive than the circular nets in Proposition~\ref{prop: 4x4planarplanar} and Proposition~\ref{prop: circularnetplanarplanar}.
\section{Principal Contact Element Nets with Spherical Parameter Lines}
\label{section: Lieapproach}

\subsection{Lie Geometry and Principal Contact Element Nets}
\label{ss: Lie}

We start with some standard facts about  the Lie geometry of oriented hyperspheres in Euclidean space $\R^n$. See \cite{cecil2008lie, blaschke1929, DDG} for more information on Lie geometry.

Let $e_1,\ldots,e_{n+1},e_{n+2},e_{n+3}$ be an orthonormal basis in $\R^{n+1,2}$:
\begin{eqnarray*}
\langle e_i,e_j\rangle =0,\ i\neq j, \langle e_i,e_i\rangle =1,\ i=1,\ldots, n+1, \langle e_{n+2},e_{n+2}\rangle =\langle e_{n+3},e_{n+3}\rangle=-1.
\end{eqnarray*}
The norm of a vector $x = (x_1, \ldots, x_{n+3})=\sum_{i=1}^{n+3} x_i e_i$ is given by 
$$
\langle x,x\rangle=\sum_{i=1}^{n+1} x_i^2-x_{n+2}^2-x_{n+3}^2.
$$
It is convenient to introduce two isotropic vectors
\begin{equation}
\label{eq:e_0-infty}
e_\infty=\frac{1}{2}(e_{n+2}+e_{n+1}), \quad e_0=\frac{1}{2}(e_{n+2}-e_{n+1}).
\end{equation}

The $(n+1)$-dimensional \emph{Lie quadric} 
$$\mathcal{L}^{n+1}:=\{[x]= [x_1 \ldots, x_{n+3}]\in \R \mathrm{P}^{n+2}\mid \langle x , x\rangle =0 \}$$
parametrises the set of non-empty oriented hyperspheres, oriented hyperplanes and points in Euclidean space $\R^n$. Models of the above elements in $\mathcal{L}^{n+1}$ are as follows.

\begin{itemize}
\item {\em Oriented hypersphere with center $c\in\R^n$ and
signed radius $r\in\R$:}
\begin{equation}
\label{eq: Lie sphere}
\hat{S}=[c+e_0+(|c|_{\R^n}^2-r^2)e_\infty+r e_{n+3}].
\end{equation}
\item {\em Oriented hyperplane $\langle v,x \rangle_{\R^n}=d$ with
$|v|_{\R^n} =1$ and $d\in\R$}:
\begin{equation}\label{eq: Lie plane}
\hat{E}=[v+0\cdot e_0+2d e_\infty+e_{n+3}].
\end{equation}
\item {\em Point $P\in\R^n$:}
\begin{equation}\label{eq: Lie point}
\hat{P}=[P+e_0+|P|_{\R^n}^2 e_\infty+0\cdot e_{n+3}].
\end{equation}
\end{itemize}

The M{\"o}bius quadric $\mathcal{M}^{n} :=\{ [x]\in \R \mathrm{P}^{n+1}\mid x_1^2+\ldots + x_{n+1}^2 -x_{n+2}^2=0 \}$ is identified with $\mathcal{L}^{n+1} \cap [e_{n+3}]^\perp$ where $[e_{n+3}]^\perp$ is the polar hyperplane relative to $\mathcal{L}^{n+1}$ of the point $[e_{n+3}]$. Points of the \emph{point complex} $\mathcal{L}^{n+1} \cap [e_{n+3}]^\perp$ are identified with spheres of radius zero, which are treated as points in $\R^n$. The point $[e_\infty]$ is interpreted as the point at infinity of the one-point compactification of  $\R^n$. Points of the \emph{hyperplane complex}   $\mathcal{L}^{n+1} \cap [e_\infty]^\perp$, where $ [e_\infty]^\perp$ is the polar hyperplane relative to $\mathcal{L}^{n+1}$ of the point $[e_\infty]$, are identified with oriented hyperplanes in $\R^n$.  


The Lie quadric provides a double cover of the points in the exterior of $\mathcal{M}^{n}$. In M\"obius geometry, the exterior of $\mathcal{M}^n \subset \R \mathrm{P}^{n+1}$ parametrises the set of non-oriented hyperspheres and hyperplanes in $\R^n$. Let $S$ be a point in the exterior of the M\"obius quadric $\mathcal{M}^{n},$ which is identified with the exterior of $\mathcal{L}^{n+1} \cap [e_{n+3}]^\perp$. Then, the line $S \vee [e_{n+3}]$ intersects the Lie quadric at two real points, which are interpreted as the two orientations of the hypersphere in $\R^n$ corresponding to the point $S$.

Two oriented hyperspheres are in \emph{oriented contact} if and only if the corresponding points $\hat{S}_1$ and $\hat{S}_2$ in $\mathcal{L}^{n+1}$ are conjugate relative to $\mathcal{L}^{n+1}$. Any two hyperspheres in oriented contact determine a \emph{contact element}, i.e. their point of contact and their
common tangent hyperplane. Then, the line $\hat{S}_1 \vee \hat{S}_2$  is {\em
isotropic}, i.e. lies entirely in $\mathcal{L}^{n+1}$. An isotropic line $l$ contains exactly one point $l\cap [e_{n+3}]^\perp $ with vanishing $e_{n+3}$-component (the common point $P$ of contact of all the oriented hyperspheres in $\R^n$ corresponding to the points of $l$), and, exactly one point $l \cap [e_\infty]^\perp $ with vanishing $e_0$-component (the common tangent oriented hyperplane $E$ in $\R^n$ of all the hyperspheres). We denote the space of isotropic lines by

$$
\mathcal{L}^{n+1}_0=\{\text{space of isotropic lines} \} = \{\text{space of contact elements}\}.
$$

%


Principal contact element nets provide a discretisation of curvature-line parametrisations of surfaces in $\R^3$ in terms of points and also tangent planes. They were introduced in \cite{bobenko2007organizing} to provide a Lie-geometric discretisation of curvature-line parametrisations, see also \cite{DDG} for more information on principal contact element nets.

Let 
$$
l : \Z^2 \to  \mathcal{L}^4_0
$$ 
be a \emph{principal contact element net}. This means that $l$ is a \emph{discrete line congruence}, i.e. any two neighboring 
lines intersect. By intersecting a principal contact element net by two hyperplanes one obtains two Q-nets. The first one
$$
\hat{P}:=l\cap [e_6]^\perp:\Z^2\to \mathcal{M}^3=\mathcal{L}^4  \cap [e_6]^\perp
$$
is a Q-net in the M\"obius quadric. Via stereographic projection, it determines a circular net $P: \Z^2 \to \R^3$. The second one 
$$
\hat{E}:= l\cap [e_\infty]^\perp:\Z^2\to \mathcal{L}^4  \cap [e_\infty]^\perp
$$
is a $Q$-net in the plane complex. Let $E_{i,j}$ be the oriented plane in $\R^3$ corresponding to $\hat{E}_{i,j}$. Then, $E: \Z^2\to \{\text{oriented planes in }\R^3\}$ is a conical net. The oriented planes $E_{i,j}$, $E_{i+1,j}$, $E_{i+1,j+1}$, $E_{i,j+1}$ are tangent to a cone of revolution. See \cite{liu2006geometry, pottmann2007geometry} for more information concerning conical nets and their applications in architectural geometry. Points of the circular net lie in the planes of the conical net, i.e. $P_{i,j}\in E_{i,j}$.  

In this elementary interpretation, a principal contact element net can be identified with a pair of maps $(P, E)$ where $P: \Z^2 \to \R^3$ is a circular net and $E: \Z^2 \to \{\text{oriented planes in } \R^3\}$ is a conical net such that $P_{i,j} \in E_{i,j}$. A contact element $(P_{i,j}, E_{i,j})$ has oriented contact with an oriented sphere $\mathcal{S}$ if $P_{i,j} \in \mathcal{S}$ and $E_{i,j}$ is an oriented tangent plane of $\mathcal{S}$. The defining property of a principal contact element net $(P, E)$ is that any two neighbouring contact elements have oriented contact with an oriented sphere. 

As explained in \cite{DDG, pottmann2008focal} any circular net $P: \Z^2 \to \R^3$ can be extended to a principal contact element net $(P,E)$ in a $2$-parameter family of ways. There is a $2$-parameter family of possible choices for the oriented plane $E_{0,0}$ containing $P_{0,0}$. The conical net $E: \Z^2 \to \{\text{oriented planes in } \R^3\}$ is uniquely determined by the choice of $E_{0,0}$. For instance, the plane $E_{1,0}$ is determined as the reflection of $E_{0,0}$ about the plane that is the perpendicular bisector of the segment joining $P_{0,0}$ and $P_{1,0}$. Similarly, any conical net $E: \Z^2 \to \{\text{oriented planes in } \R^3\}$ has a $2$-parameter family of extensions to a principal contact element net $(P,E)$.

\subsection{One Family Spherical}
\label{Lie: onefamilyspherical}

Let $l : \Z^2 \to \mathcal{L}_0^4$ be a principal contact element net. Let $\hat{P}: \Z^2 \to \mathcal{L}^4 \cap [e_6]^\perp$ be the corresponding circular net in the M{\"o}bius quadric $\mathcal{M}^3$ which is identified with $\mathcal{L}^4 \cap [e_6]^\perp$. Suppose that $\{ \hat{P}_{i,j} \}_{i\in \Z}$ is a spherical parameter line, i.e.\ $\vee_i \hat{P}_{i,j}$ is $3$-dimensional. Then, $\vee_i l_{i,j}$ is $4$-dimensional because $l$ is a discrete line congruence. This motivates Definition~\ref{defn: Liessphericalpara}.

\begin{defn}
\label{defn: Liessphericalpara}
Let $l : \Z^2 \to  \mathcal{L}_0^4$ be a principal contact element net. The parameter line $\{ l_{i,j}\}_{i \in \Z}$ is \emph{spherical} if $\vee_i l_{i,j}$ is $4$-dimensional.
\end{defn}

\begin{prop}
\label{prop: sphericalaparaLie}
Let $l: \Z^2 \to  \mathcal{L}_0^4$ be a principal contact element net. Let $P: \Z^2 \to \R^3$ and $E: \Z^2 \to \{\text{oriented planes in } \R^3\}$ be the corresponding circular and conical nets. Then, the following conditions are equivalent:
\begin{itemize}
\item $\{ l_{i,j}\}_{i \in \Z}$ is spherical,
\item $\{ P_{i,j}\}_{i \in \Z}$ is spherical, i.e. the points lie on a sphere $\mathcal{S}_1$,
\item $\{ E_{i,j}\}_{i \in \Z}$ are in oriented contact with a common oriented sphere $\mathcal{S}_2$. 
\end{itemize} 
The planes $\{E_{i,j}\}_{i \in \Z}$ intersect the sphere $\mathcal{S}_1$ at a constant angle. The spheres $\mathcal{S}_1$ and $\mathcal{S}_2$ are concentric.
\end{prop}

\begin{proof}
Since $l_{i,j}$ is a line congruence and $\hat{P}_{i,j}\in \mathcal{L}^4  \cap [e_6]^\perp$, $\hat{E}_{i,j}\in \mathcal{L}^4  \cap [e_\infty]^\perp$, the conditions $\dim \vee_i  l_{i,j} =4$, 
$\dim \vee_i  \hat{P}_{i,j} =3$ and $\dim \vee_i  \hat{E}_{i,j} =3$ are equivalent. This proves the equivalences in Proposition~\ref{prop: sphericalaparaLie}.

Suppose that the points $\{P_{i,j}\}_{i \in \Z}$ are contained in a sphere $\mathcal{S}_1$. The planes $E_{i,j}$ and $E_{i+1,j}$ are symmetric about the plane that perpendicularly bisects the segment through $P_{i,j}$ and $P_{i+1,j}$. The foregoing plane contains the centre of $\mathcal{S}_1$. So, the planes  $E_{i,j}$ and $E_{i+1,j}$ intersect $\mathcal{S}_1$ at a constant angle. 
\end{proof}



Proposition~\ref{prop: sphericalaparaLie} is a discrete analogue of Joachimsthal's theorem. The latter states that the tangent planes of a surface in $\R^3$ along a spherical curvature line intersect the sphere containing the spherical curvature line at a constant angle.
%

\begin{prop}
\label{prop: alternatingoneparapherical}
Let $l: \Z^2 \to \mathcal{L}_0^4$ be a principal contact element net with spherical parameter lines $\{l_{i,j}\}_{i\in \Z}$. There exist quadrics $\mathcal{Q}_j$ such that the planes $\{l_{i,j} \vee l_{i,j+1}\}_{i \in \Z}$ are isotropic planes of $\mathcal{Q}_j$. Generically, there is an alternation phenomenon:  the isotropic planes $\{l_{i,j} \vee l_{i,j+1}\}_{i \in \Z}$ decompose into the two systems of generators $\{l_{2i,j} \vee l_{2i,j+1}\}_{i \in \Z}$ and $\{l_{2i+1,j} \vee l_{2i+1,j+1}\}_{i \in \Z}$ of $\mathcal{Q}_j$.
\end{prop}

\begin{proof}
Consider the hyperplanes $\vee_i l_{i,j}$ and $\vee l_{i,j+1}$ in $\R \mathrm{P}^5$. By Lemma~\ref{lem: 2quadricsdefinepencil}, there is a unique pencil of $4$-dimensional quadrics that contain the $3$-dimensional quadrics $\vee_i l_{i,j} \cap \mathcal{L}^4$ and $\vee_i l_{i,j+1} \cap \mathcal{L}^4$. Let $\mathcal{Q}_j$ be the quadric in the pencil such that it contains a point in the plane $l_{0,j}\vee l_{0,j+1}$ that is not contained in $l_{0,j}\cup l_{0,j+1}$. Then, $l_{0,j}\vee l_{0,j+1}$ is an isotropic plane of $\mathcal{Q}_j$. Then, the intersection $(l_{1,j}\vee l_{1,j+1}) \cap \mathcal{Q}_j$ contains the line $(l_{0,j}\vee l_{0,j+1}) \cap (l_{1,j}\vee l_{1,j+1})$.  It also contains the lines $l_{1,j}$ and $l_{1,j+1}$. Any planar section of $\mathcal{Q}_j$ is either a conic or an isotropic plane. Therefore, $l_{1,j}\vee l_{1,j+1}$ is an isotropic plane of $\mathcal{Q}_j$. Iterating, we see that the planes $l_{i,j}\vee l_{i,j+1}$ are isotropic planes of $\mathcal{Q}_j$. Generically, the planes are not concurrent. So, they are isotropic planes of a quadric of signature $(+++---)$. Since $l$ is a discrete line congruence $(l_{i,j}\vee l_{i,j+1})\cap (l_{i+1,j}\vee l_{i+1,j+1})$ is $1$-dimensional. By Proposition~\ref{prop: hyperbolicquadrictwoclasses}, $l_{i,j}\vee l_{i,j+1}$ and $l_{i+1,j}\vee l_{i+1,j+1}$ are in different systems of generators.
\end{proof}

The alternation phenomenon indicates that additional specifications are required to model surfaces with spherical curvature lines, and to obtain them in a smooth limit.  This can be achieved by requiring that for any $j$ the planes $\{ l_{i,j}\vee l_{i,j+1} \}_{i\in \Z}$ are concurrent, i.e.  $\cap_i l_{i,j} \vee l_{i,j+1}$ is a point. 
This motivates the following definition.


\begin{defn}
\label{defn: 1paraspherical}
Let $l: \Z^2 \to \mathcal{L}_0^4$ be a principal contact element net. It has \emph{one family of spherical parameter lines} $\{l_{i,j}\}_{i\in \Z}$ if,  for any $j$ the parameter line $\{l_{i,j}\}_{i\in \Z}$ is spherical and the planes $\{l_{i,j} \vee l_{i,j+1}\}_{i\in \Z}$ are concurrent.
\end{defn}

\begin{prop}
\label{prop: discreteBlaschke1parasphericalLietransforms}
Let $l: \Z^2 \to  \mathcal{L}_0^4$ be a principal contact element net with one family of spherical parameter lines $\{ l_{i,j} \}_{i\in \Z}$. Then for any $j_0 $ and $j_1$, there is a projective transformation $f_{j_0,j_1} : \vee_i l_{i,j_0}\to \vee_i l_{i,j_1}$ such that $f_{j_0,j_1}(l_{i,j_0})= l_{i,j_1}$.
\end{prop}

\begin{proof}
Let $f_{j, j+1}: \vee_i l_{i,j}\to \vee_i l_{i,j+1}$ be the central projection with centre $\cap_i (l_{i,j} \vee l_{i,j+1})$. This is a projective transformation identifying the corresponding lines: $f_{j, j+1}(l_{i,j})= l_{i,j+1}$. The composition $f_{j_1-1,j_1}\circ \ldots \circ f_{j_0+1, j_0+2} \circ f_{j_0, j_0+1}$ is the required projective transformation $f_{j_0, j_1}$.
\end{proof}
%

Proposition~\ref{prop: discreteBlaschke1parasphericalLietransforms} sheds light on the classical theory of smooth surfaces with one family of spherical curvature lines.  Let $l: \Z^2 \to \{\text{isotropic lines of } \mathcal{L}^4\}$ be a principal contact element net with one family of spherical parameter lines. Each spherical parameter line corresponds to a $4$-dimensional space that intersects $\mathcal{L}^4$ in a quadric that generically has signature $(+++--)$. Each of these quadrics can be identified with the $3$-dimensional Lie quadric $\mathcal{L}^3$ that parametrises the set of oriented circles in the $2$-dimensional M\"obius quadric $\mathcal{M}^2$, which can be identified with each of the spheres in $\R^3$ that contain the spherical parameter lines. Then, Proposition~\ref{prop: discreteBlaschke1parasphericalLietransforms} says that any two spherical parameter lines of a principal contact element net with one family of spherical parameter lines are related by a Lie transformation. This observation has a counterpart in the smooth theory. If a surface has one family of spherical curvature lines, then any two spherical curvature lines are related by a Lie transformation \cite{blaschke1929}.


Proposition~\ref{prop: 1parasphericalgoursatlaplace} concerns discrete line congruences with terminating Laplace sequences. See Definition~\ref{def:laplace_line congruence}.

\begin{prop}
\label{prop: 1parasphericalgoursatlaplace}
Let $l: \Z^2 \to \mathcal{L}_0^4$ be a principal contact element net with one family of spherical parameter lines $\{l_{i,j} \}_{i\in \Z}$. Generically, $\mathcal{L}_A^3l$ is Goursat degenerate and $\mathcal{L}_B^2l$ is Laplace degenerate.
\end{prop}

\begin{proof}
Let $V_{i,j}$ be the $3$-dimensional space $\vee_i l_{i,j}\vee l_{i,j+1}\vee l_{i,j+2}$. The intersection $\cap_i V_{i,j}$ is generically $1$-dimensional. It is the line joining the points $\cap_i l_{i,j} \vee l_{i,j+1}$ and $\cap_i l_{i,j+1} \vee l_{i,j+2}$. The foregoing line equals $\mathcal{L}_B^2l(i,j)$ for all $i$. So, $\mathcal{L}_B^2l$ is Laplace degenerate. Let $H_j$ be the $4$-dimensional space $\vee_i l_{i,j}$. The intersection $H_j \cap H_{j+1} \cap H_{j+2} \cap H_{j+3}$ is generically $1$-dimensional. The foregoing line equals $\mathcal{L}_A^3l(i,j)$ for all $i$. So, $\mathcal{L}_A^3 l$ is Goursat degenerate.
\end{proof}

Let $\hat{P}: \Z^2 \to \mathcal{M}^3$ be the circular net corresponding to the principal contact element net $l$ with one family of spherical parameter lines. Intersecting $\mathcal{L}^4$ by $[e_6]^\perp$ we obtain from Proposition~\ref{prop: 1parasphericalgoursatlaplace} the corresponding statements for $\hat{P}$. Thus, for the corresponding circular net $P:\Z^2\to \R^3$ we obtain that $\mathcal{L}_A^3 P$ is Goursat degenerate and $\mathcal{L}_B^2 P$ is Laplace degenerate.  As already mentioned in Section~\ref{section: Mobiusapproach}, these properties are in complete agreement with the smooth theory. We use them to define circular nets with one family of spherical parameter lines.

Denote by $\Pi_{i,j}$ the plane containing the points $P_{i,j}, P_{i,j+1}, P_{i,j+2}$.

\begin{defn}
\label{def:1parSphericalmoebius}
A circular net $P:\Z^2\to \R^3$ has \emph{one family of spherical parameter lines} $\{ P_{i,j}\}_{i\in\Z}$ if
all parameter lines $\{ P_{i,j}\}_{i\in\Z}$ are spherical, and  the planes $\{ \Pi_{i,j} \}_{i\in\Z}$ are concurrent.
\end{defn}

Denote by $\mathcal{C}_{i,j}\subset \Pi_{i,j}$ the circle containing the points $P_{i,j}, P_{i,j+1}, P_{i,j+2}$, and by $\mathcal{S}_{i,j}$ the sphere containing the points $P_{i,j}, P_{i,j+1}, P_{i,j+2}, P_{i+1,j}, P_{i+1,j+1}, P_{i+1,j+2}$.

\begin{prop}
\label{prop:circular_one_family_spherical}
A circular net $P:\Z^2\to \R^3$ has one family of spherical parameter lines $\{P_{i,j}\}_{i\in \Z}$ if and only if all its parameter lines $\{P_{i,j}\}_{i\in \Z}$ are spherical and one of the following equivalent conditions is satisfied:
\begin{itemize}
\item the spheres $\{ \mathcal{S}_{i,j} \}_{i\in\Z}$ have a (possibly imaginary) common orthogonal sphere,
\item the circles $\{ \mathcal{C}_{i,j} \}_{i\in\Z}$ have a (possibly imaginary) common orthogonal sphere,
\item the Laplace transform $\mathcal{L}_B^2 \hat{P}$ is Laplace degenerate.
\end{itemize}
\end{prop}

\begin{proof}
Let $P_j=\cap_i \Pi_{i,j}$ be the intersection point of the planes $\{\Pi_{i,j}\}_{i\in\Z}$. Consider the (possibly imaginary) sphere $\mathcal{S}_j$ with the center $P_j$ and orthogonal to the sphere $\mathcal{S}_{1,j}$. Since $\mathcal{S}_j$ is orthogonal to all planes $\{ \Pi_{i,j}\}_{i\in\Z}$ and $\mathcal{S}_{i,j}, \mathcal{S}_{i+1,j}$ and $\Pi_{i+1,j}$ belong to a pencil, $\mathcal{S}_j$ is orthogonal to all the spheres $\{\mathcal{S}_{i,j}\}_{i\in\Z}$. Equivalently, $\mathcal{S}_j$ is orthogonal to all the circles $\{\mathcal{C}_{i,j}=\mathcal{S}_{i,j}\cap \Pi_{i,j}\}_{i \in \Z}$.

Lift $P: \Z^2 \to \R^3$ to the M\"obius quadric. Then, $\hat{P}: \Z^2 \to \mathcal{M}^3$ is a Q-net such that each $P_{i,j}$ is the stereographic projection of $\hat{P}_{i,j}$. By Proposition~\ref{prop: Laplaceterminatemsteps}, $\mathcal{L}_A^2 \hat{P}$ is Laplace degenerate if and only if the planes $\{\hat{P}_{i,j} \vee \hat{P}_{i,j+1} \vee \hat{P}_{i,j+2}\}_{i\in \Z}$ are concurrent. Equivalently, the circles $\{\mathcal{C}_{i,j}\}_{i \in \Z}$ have a common orthogonal sphere which corresponds to the intersection of $\mathcal{M}^3$ with the polar hyperplane of the foregoing concurrency point. 

\end{proof}

The next proposition reveals some of the geometric properties of principal contact element nets with one family of spherical parameter lines.

\begin{prop}
\label{cor: conicalandcircularLaplace2steps}
Let $l: \Z^2 \to \mathcal{L}_0^4$ be a principal contact element net with one family of spherical parameter lines $\{l_{i,j}\}_{i\in \Z}$. Let $P: \Z^2 \to \R^3$ and $E: \Z^2 \to \{\text{oriented planes in } \R^3\}$ be the corresponding circular and conical nets, and $n_{i,j}$ be the normal line of $E_{i,j}$ through $P_{i,j}$. Then:
\begin{itemize}
\item the circular  net $P$  has one family of spherical parameter lines in the sense of Definition~\ref{def:1parSphericalmoebius}.
\item for each $j$ the planes $\{n_{i,j} \vee n_{i,j+1}\}_{i\in \Z}$ are concurrent and the concurrency point is contained in the line joining the centres of the spheres through the points $\{P_{i,j}\}_{i \in \Z}$ and $\{P_{i,j+1}\}_{i \in \Z}$,
\item for each $j$ the intersection points $\{K_{i,j}:=E_{i,j} \cap E_{i,j+1} \cap E_{i,j+2}\}_{i \in \Z}$ are coplanar.
\end{itemize}
\end{prop}

\begin{proof}


By Proposition~\ref{prop: 1parasphericalgoursatlaplace}, $\mathcal{L}_A^2 l$ is Laplace degenerate. So, $\mathcal{L}_A^2\hat{P}$ is Laplace degenerate. Then, the first claim follows from Proposition~\ref{prop:circular_one_family_spherical}.

Let $\pi$ be the central projection with centre $[e_6]$ onto the hyperplane $[e_6]^\perp$. Recall, each point $\hat{S} \in \mathcal{L}^4$ corresponds to an oriented sphere in $\R^3$.  The projection $\pi(\hat{S})$ is a point outside $\mathcal{M}^3 = \mathcal{L}^4 \cap [e_6]^{\perp}$ that corresponds in M\"obius geometry to the same sphere in $\R^3$, without orientation. By Lemma~\ref{lem: stereoproj}, $\sigma \circ \pi(\hat{S})$ is the centre of the oriented sphere in $\R^3$ that corresponds to the point $\hat{S}$. Then, $n_{i,j} = \sigma \circ \pi(l_{i,j})$ because $n_{i,j}$ is the locus of the centres of oriented spheres of the contact element corresponding to $l_{i,j}$. Let $O_j$ be the concurrency point $\cap_{i \in \Z} l_{i,j}\vee l_{i,j+j}$. Then, $\sigma \circ \pi (O_j)$ is the concurrency point $\cap_{i \in \Z} n_{i,j} \vee n_{i,j+1}$.

Each space $H_j := \mathrm{join}\{l_{i,j}\}_{i \in \Z}$ is $4$-dimensional. The point $\pi(H_j^\perp)$ has a polar hyperplane relative to $\mathcal{L}^4$ that intersects $[e_6]^{\perp}$ in a $3$-dimensional space that equals $H_j \cap [e_6]^\perp$. So, $\pi(H_j^{\perp})$ is a point outside $\mathcal{M}^3 = \mathcal{L}^4 \cap [e_6]^{\perp}$ that corresponds in M\"obius geometry to the sphere $\mathcal{H}_j$ in $\R^3$ through the points $\{P_{i,j}\}_{i \in \Z}$. By Lemma~\ref{lem: stereoproj}, $\sigma \circ \pi (H_j^{\perp})$ is the centre of the sphere $\mathcal{H}_j$. The point $O_j$ is conjugate to the points $\{l_{i,j}\cap l_{i,j+1}\}_{i \in \Z}$ which span the $3$-dimensional space $H_j \cap H_{j+1}$. So, the point $O_j$ is contained in the line $H_j^{\perp} \vee H_{j+1}^{\perp}$. The concurrency point $\sigma \circ \pi (O_j)$ is contained in the line $\sigma \circ\pi (H_j^\perp \vee  H_{j+1}^\perp)$ which is the line joining the centres of the spheres  $\mathcal{H}_j$  and $\mathcal{H}_{j+1}$. 

Let $\hat{E}_{i,j}\in \mathcal{L}^4\cap [e_\infty]^\perp$ and $\hat{K}_{i,j}\in \mathcal{L}^4\cap [e_6]^\perp$ be the lifts of the planes $E_{i,j}$ and of their 
intersection points $K_{i,j}$ to the Lie quadric. Since $\hat{E}_{i,j}, \hat{E}_{i,j+1}, \hat{E}_{i,j+2} \in \hat{K}_{i,j}^\perp$ and the Laplace transform
 $\mathcal{L}_B^2 \hat{E}$ is independent of $i$, we obtain the conjugacies
$\hat{K}_{i,j} \perp \mathcal{L}_B^2 \hat{E}(j)$ and $\mathcal{L}_B^2 \hat{E}(j) \perp [e_\infty]$ relative to $\mathcal{L}^4$. Thus, the line $\mathcal{L}_B^2\hat{E}(j) \vee [e_\infty]$ is contained in the plane complex $[e_\infty]^\perp$ and intersects $\mathcal{L}^4$ in the point $[e_\infty]$ and a second real point which is conjugate to the points $\{\hat{K}_{i,j}\}_{i \in \Z}$. The second intersection point corresponds to a plane in $\R^3$ that contains the points $\{ K_{i,j} \}_{i\in \Z}$.

\end{proof}

\subsection{Both Families Spherical}\label{Lie: twofamilyspherical}

\begin{defn}\label{defn: Lie2familiesspherical}
Let $l: \Z^2 \to \mathcal{L}_0^4$ be a principal contact element net. It has \emph{two families of spherical parameter lines} if the following conditions are satisfied:
\begin{itemize}
\item for any $j$ the parameter line $\{l_{i,j}\}_{i\in \Z}$ is spherical and $\cap_i l_{i,j} \vee l_{i,j+1}$ is a point.
\item for any $i$ the parameter line $\{l_{i,j}\}_{j\in \Z}$ is spherical and $\cap_j l_{i,j} \vee l_{i+1,j}$ is a point.
\end{itemize}
\end{defn}

We start with Proposition~\ref{prop: isotropiclineslaplacetransform}, which is an analogue of Theorem~\ref{thm: mxmiteratedlaplaceconjugate} for line congruences. It is used later in the proof of Theorem~\ref{thm: twofamiliessphericalconjugateplanesLie}.

\begin{prop}
\label{prop: isotropiclineslaplacetransform}
Let $l: [m]\times [m] \to \{ \text{lines in }\mathbb{R}\mathrm{P}^n\}$ be a discrete line congruence, where $m \geq 2$ and $n\geq 3$. Suppose that $\mathcal{L}_A^d l : [m-d]\times [m-d] \to \mathbb{R}\mathrm{P}^n$ and $\mathcal{L}_B^d l: [m-d]\times [m-d] \to \mathbb{R}\mathrm{P}^n$ are well defined discrete line congruences for all $d \le m-1$. Then, $\mathcal{L}_A^{m-1} l$ and $\mathcal{L}_B^{m-1} l$ are two lines. Suppose that the lines $\{l_{i,j} \mid (i,j) \in [m]\times [m] \text{ s.t.\ } (i,j)\neq(m,m)\}$ are contained in a quadric  $\mathcal{Q} \subset \mathbb{R}\mathrm{P}^n$. Then, the line $l_{m,m}$ is contained in $\mathcal{Q}$ if and only if $\mathcal{L}_A^{m-1} l$ and $\mathcal{L}_B^{m-1} l$ are conjugate with respect to $\mathcal{Q}$.
\end{prop}

\begin{proof}
Let $\Pi$ be a hyperplane of $\R \mathrm{P}^n$ that does not contain any of the lines of $l$. The points $P_{i,j} := l_{i,j} \cap \Pi$ build a Q-net $P$ in the quadric $\mathcal{Q} \cap \Pi$. By Theorem~\ref{thm: mxmiteratedlaplaceconjugate}, the point $P_{m,m}$ is contained in $\mathcal{Q} \cap \Pi$ if and only if the points $\mathcal{L}^{m-1}_A P$ and $\mathcal{L}_B^{m-1} P$ are conjugate relative to $\mathcal{Q}$. As $\Pi$ is arbitrary, the line $l_{m,m}$ is contained in $\mathcal{Q}$ if and only if $\mathcal{L}^{m-1}_A l$ and $\mathcal{L}_B^{m-1} l$ are conjugate relative to $\mathcal{Q}$.
\end{proof}

\begin{thm}
\label{thm: twofamiliessphericalconjugateplanesLie}
Let $l: \Z^2 \to \mathcal{L}_0^4$ be a principal contact element net with two families of spherical parameter lines.  Then, $\mathcal{L}^2_Al$ and $\mathcal{L}^2_Bl$ are Laplace degenerate. Each set of points $(\vee_i l_{i,j})^\perp$ and $(\vee_j l_{i,j})^\perp$ is contained in a $2$-plane. The two $2$-planes are conjugate relative to $\mathcal{L}^4$.
Let $P: \Z^2 \to \R^3$ be the corresponding circular net. Let $\mathcal{H}_j$ and $\mathcal{V}_i$ be the spheres through the points $\{P_{i,j} \mid i \in \Z\}$ and $\{P_{i,j} \mid j \in \Z\}$, respectively. Then, 
\begin{itemize}
\item Each family of spheres has two common (possibly imaginary) points.
\item The centres of the spheres $\mathcal{H}_j$ are contained in a plane that is orthogonal to the plane containing the centres of the spheres $\mathcal{V}_i$.
\end{itemize}
\end{thm}

\begin{proof}
Proposition~\ref{prop: 1parasphericalgoursatlaplace} ensures that $\mathcal{L}^2_Al$ and $\mathcal{L}^2_Bl$ are Laplace degenerate. The lines $\{\mathcal{L}_A^2 l(i)\}_{i \in \Z}$ span the space $X:=\cap_{j \in \Z}H_j$, where $H_j := \mathrm{join}\{l_{i,j} \mid i \in \Z\}$. The space $X$ is $2$-dimensional. Indeed, it is the join of the concurrency points  $\cap_j l_{i,j} \vee l_{i+1,j}$, $\cap_j l_{i+1,j} \vee l_{i+2,j}$ and $\cap_j l_{i+2,j} \vee l_{i+3,j}$. The lines $\{\mathcal{L}_B^2 l(j)\}_{j \in \Z}$ span the $2$-dimensional space $Y:=\cap_{i \in \Z}V_i$ where $V_i := \mathrm{join}\{l_{i,j} \mid j \in \Z\}$.  By Proposition~\ref{prop: isotropiclineslaplacetransform}, $\mathcal{L}_A^2 l(i) \perp \mathcal{L}_B^2l(j)$ for any $i, j$. So, $X$ and $Y$ are conjugate planes. The points $\{H_j^\perp\}_{j \in \Z}$ are contained in $Y$. The points $\{V_i^\perp\}_{i \in \Z}$ are contained in $X$.

For each $j$, $H_j \cap \mathcal{L}^4 \cap [e_6]^\perp$ is the sphere in $\mathcal{M}^3= \mathcal{L}^4 \cap [e_6]^\perp$ that contains the points $\{\hat{P}_{i,j} = l_{i,j} \cap [e_6]^\perp \}_{i \in \Z}$. The $2$-dimensional space $X$ intersects $[e_6]^\perp$ in a line that intersects $\mathcal{M}^3$ in two (possibly imaginary) points that are common to all the spheres of one family. Similarly, the second family of spheres also has two (possibly imaginary) common points. Thus, the centres of  the spheres $\{\mathcal{H}_j\}_{j \in\Z}$ are contained in a plane and the centres of  the spheres $\{\mathcal{V}_i\}_{i \in\Z}$ are contained in a plane. The two planes are orthogonal. Indeed, this is a limiting case of the two confocal quadrics in 
Corollary~\ref{cor: sphericalsphericalR3confocaldeferents}.
\end{proof}

Theorem~\ref{thm: twofamiliessphericalconjugateplanesLie} has a direct analogue in the classical theory of smooth surfaces with two families of spherical curvature liness. If a smooth surface in $\R^3$ has two families of spherical curvature lines, then the spheres correspond to points in the exterior of $\mathcal{L}^4$ that are contained in two planes that are conjugate relative to $\mathcal{L}^4$ \cite{blaschke1929, cho2021constrained}.

\section*{Acknowledgments}
This research was supported by the DFG Collaborative Research Center TRR 109 ``Discretization in Geometry and Dynamics" and the Berlin Mathematical School. Thanks to the Leverhulme Trust for the grant SAS-2018-040 that enabled A.~Y.~F. to initiate the work in this article. We thank Niklas Affolter and Yuri Suris for stimulating discussions on iterated Laplace transformations. Thanks to Th\'eo Tyburn and Oliver Gross for making Figure~\ref{figure: 4x4circulargridincidence}.

\appendix

\section{M\"obius Geometry}\label{appendix: mobiusgeo}


Here we give a brief overview of M\"obius geometry and Darboux cyclides. For instance, see \cite{blaschke1929, hertrich2003introduction, DDG} for more information on M\"obius geometry and see \cite{coolidge1916} for more information on Darboux cyclides. 

M\"obius geometry is a subgeometry of Lie geometry, with points distinguishable among all hyperspheres as those of radius zero. As it was explained in Section~\ref{ss: Lie}, it is a restriction of Lie geometry to the polar hyperplane $[e_{n+3}]^\perp$ of a time-like vector $e_{n+3}$. Thus one can use the Lie-geometric description and just omit the $e_{n+3}$-component, passing to the Lorentz space $\R^{n+1,1}\subset \R^{n+1,2} $. The resulting objects are points of the $(n+1)$-dimensional projective space $\R \mathrm{P}^{n+1}:= \mathrm{P}(\R^{n+1,1})$. We continue to use notations of Section~\ref{ss: Lie} in the context of M\"obius geometry. 

The $n$-dimensional \emph{M\"obius quadric} is given by
$$\mathcal{M}^{n}:=\{[x]= [x_1 \ldots, x_{n+2}]\in \R \mathrm{P}^{n+1}\mid \langle x , x\rangle =0 \},$$
where 
$$
\langle x,y\rangle=\sum_{i=1}^{n+1} x_i y_i-x_{n+2} y_{n+2}.
$$
is the Lorentz product of two vectors $x,y\in\R^{n+1,1}$.
We will use the basis $e_1,\ldots, e_n, e_o, e_\infty$ (see (\ref{eq:e_0-infty})), and denote the corresponding coordinates of the vectors $x\in\R^{n+1,1}$ as follows:
$$
x=\sum_{i=1}^n x_i e_i + x_0 e_0 + x_\infty e_\infty, \quad x_0=x_{n+2}-x_{n+1}, \ x_\infty=x_{n+2}+x_{n+1}.
$$

Without loss of generality two arbitrary points on the M\"obius quadric can be normalized as $[e_\infty]$ and $[e_0]$.  Let us consider the tangent hyperplanes to $\mathcal{M}^{n}$ at these points as well as their intersection plane:
\begin{eqnarray*}
\mathcal{H}^\infty & := &\{ [x]\in \R \mathrm{P}^{n+1}\mid x_0=0 \}, \\
\mathcal{H}^0 & := &\{ [x] \in \R \mathrm{P}^{n+1}\mid x_\infty=0 \}, \\
\mathcal{H} & := &\mathcal{H}^\infty\cap \mathcal{H}^0=\{ [x] \in \R \mathrm{P}^{n+1}\mid x_0=x_\infty=0 \}.
\end{eqnarray*}

Elements of M\"obius geometry of $\R^n\cup \infty$ are modelled in the space $\R^{n+1,1}$ of homogeneous coordinates as follows:
\begin{itemize}
\item {\em Point $P\in\R^n$:}
\begin{equation}\label{eq: Moeb point}
P_{\rm Euc}=P+e_0+|P|^2 e_\infty.
\end{equation}
\item {\em Infinity $\infty$:}
\begin{equation}\label{eq: Moeb infty}
e_\infty.
\end{equation}
\item {\em Hypersphere $\mathcal{S}$ with center $c\in\R^n$ and radius $r>0$:}
\begin{equation}\label{eq: Moeb sphere}
S_{\rm Euc}=c+e_0+(|c|^2-r^2) e_\infty.
\end{equation}
\item {\em Hyperplane $\langle v,P \rangle_{\R^n}=d$ with
$|v|_{\R^n} =1$ and $d\in\R$:}
\begin{equation}\label{eq: Moeb plane}
E_{\rm Euc}=v+0\cdot e_0+2d e_\infty.
\end{equation}
\end{itemize}

In the projective space $ \R \mathrm{P}^{n+1}$ points of $\R^n \cup\infty$ are represented by points of the M\"obius quadric. Similarly to Section~\ref{ss: Lie} we denote the elements of the projective space corresponding to the representatives (\ref{eq: Moeb point}, \ref{eq: Moeb sphere}, \ref{eq: Moeb plane}) by:
$$
\hat{P}=[P_{\rm Euc}], \hat{S}=[S_{\rm Euc}], \hat{E}=[E_{\rm Euc}]\in  \R \mathrm{P}^{n+1}. 
$$

Hyperspheres and hyperplanes are intersections of $\mathcal{M}^n$ with hyperplanes $H$. These are represented by their space-like polar points $\hat{S}, \hat{E} \in \R \mathrm{P}^{n+1}$. Note, $\langle S_{\rm Euc} , S_{\rm Euc} \rangle=r^2 $. 
If the hyperplane $H$ does not intersect $\mathcal{M}^{n}$ then the corresponding hypersphere is called \emph{imaginary}, and is represented by its polar time-like vector, i.e. purely imaginary $r$. The incidence condition $P\in S$ reads as the polarity condition $ \langle S_{\rm Euc} , P_{\rm Euc}\rangle=0$. The hyperplanes are treated as spheres passing through $\infty$, i.e. satisfying the relation $ \langle E_{\rm Euc} , e_\infty \rangle=0$.

To sum up:  Points in the exterior of $\mathcal{M}^n$ (where $\langle x , x\rangle>0$) correspond either to Euclidean hyperspheres with real centres and real radii or else to hyperplanes in $\R^n$. Hyperplanes are interpreted as Euclidean hyperspheres passing through the infinity point. Points in $\mathcal{M}^n$ (where $\langle x , x\rangle=0$) correspond to Euclidean hyperspheres with real centres and null radii, i.e. to points of $\R\cup\infty$. Points in the interior of $\mathcal{M}^n$ (where $\langle x , x\rangle<0$) correspond to imaginary hyperspheres with real centres and purely imaginary radii. 

Two hyperspheres $\mathcal{S}'$ and $\mathcal{S}''$ are orthogonal in $\R^n$ if and only if their M\"obius representatives are orthogonal $\langle S'_{\rm Euc} , S''_{\rm Euc}\rangle=0$. Two (possibly imaginary) hyperspheres $\mathcal{M}^n \cap G_1$ and $\mathcal{M}^n \cap G_2$ are orthogonal if their poles $G_1^\perp$ and $G_2^\perp$ are conjugate relative to $\mathcal{M}^n$.
Similarly, the intersection of $\mathcal{M}^n$ with an $m$-dimensional projective subspace $H\subset\R \mathrm{P}^{n+1}$ is an $(m-1)$-dimensional sphere, which is imaginary if $H \cap \mathcal{M}^n = \emptyset$. 

Let $\sigma: \R \mathrm{P}^{n+1}\setminus N \to  \mathcal{H}^0$ be the central projection 
$$
[x=\sum_i^n x_ie_i +x_0 e_0+x_\infty e_\infty] \xrightarrow{\sigma}
 [\sum_i^n x_ie_i +x_0 e_0]
$$
with center $N:=[e_\infty]$. 
The hyperplane $\mathcal{H}^0$ is the projective closure of $\mathcal{H}^0\setminus \mathcal{H}$. The later can be identified with $\R^n := \{[x] \in \R \mathrm{P}^{n+1} \mid x_0=1, x_\infty =0 \}$, where each point $[P+e_0]$ is identified with the point $P=(x_1,\ldots, x_n) \in \R^n$.

So we also will use short notations
$
\overline{\R^{n}}:= \mathcal{H}^0, \ \R^{n}:= \mathcal{H}^0\setminus \mathcal{H},
$
and interpret $\mathcal{H}^0\setminus \mathcal{H}$ as the set of points in $\R^n$, and $\mathcal{H}^0$ as its projective closure. Note that we have already identified $\mathcal{H}^\infty$ with the space of hyperplanes in $\R^n$.
 
By comparing with (\ref{eq: Moeb point}), we see that the restriction to the M{\"o}bius quadric $\sigma:\mathcal{M}^n\setminus N \to \mathcal{H}^0\setminus \mathcal{H}\cong\R^n$ is given by
$$
\sigma(\hat{P})= [P+e_0] \cong P \in \R^n.
$$
 


The quadric $\mathcal{M}_\C^n :=\{[x] \in \C \mathrm{P}^{n+1}\mid \langle x,x\rangle =0\}$ is the \emph{complexification} of $\mathcal{M}^n$. The intersection of $\mathcal{M}_\C^n$ with an $m$-dimensional projective subspace in $\C \mathrm{P}^{n+1}$ is an $(m-1)$-dimensional \emph{complex sphere} in $\mathcal{M}_\C^n$. Similarly to the real case we denote by $\sigma_\C$  the central projection with center $N_\C := [e_\infty] \in \C \mathrm{P}^{n+1}$:
$$
\sigma_\C: \C\mathrm{P}^{n+1}\setminus N_\C \to \mathcal{H}^0_\C,
$$
and identify the points $[P+e_0]\cong P\in \C^n$ and the spaces  $\mathcal{H}^0_\C\cong \overline{\C^{n}},  \mathcal{H}^0_{\C}\setminus\mathcal{H}_{\C} \cong \C^{n}$.

From representations (\ref{eq: Moeb point}, \ref{eq: Moeb sphere}, \ref{eq: Moeb plane}) we easily obtain the following result.

\begin{lem}
\label{lem: stereoproj}
Let $\hat{S} \in \C \mathrm{P}^{n+1}$ be a point not contained in $\mathcal{H}^\infty_{\C}$, i.e. with non-vanishing $e_0$-component, and $\hat{S}^\perp$ its polar hyperplane. Then the image  $\sigma_\C(\hat{S}^\perp \cap \mathcal{M}^n_\C)$ is the round sphere $\sum_i^n (P_i -c_i)^2=r^2 $, with center $c$ and radius $r$ given by (\ref{eq: Moeb sphere}). In particular $\sigma(\hat{S})$ is the center of the sphere. Let $\hat{E} \in \C \mathrm{P}^{n+1}$ be a point in $\mathcal{H}^\infty_{\C}$, i.e. with vanishing $e_0$-component, and $\hat{E}^\perp$ its polar hyperplane. Then the image  $\sigma_\C(\hat{E}^\perp \cap \mathcal{M}^n_\C)$ is the hyperplane $\langle v,P \rangle=d$, with parameters given by (\ref{eq: Moeb plane}). In particular $\sigma(\hat{E})$ is the point at infinity of $\overline{\C^n}$ in the direction of the normals of the hyperplane.
\end{lem}

\begin{defn}\label{defn: cyclide}
The intersection of $\mathcal{M}^n \subset \mathbb{R}\mathrm{P}^{n+1}$, with a real quadric hypersurface $\mathcal{Q} \neq \mathcal{M}^n$ is an $(n-1)$-dimensional \emph{Darboux cyclide} in $\mathcal{M}^n$. The central projection $\sigma(\mathcal{Q}\cap \mathcal{M}^n)$ is a Darboux cyclide in $\R^n=\mathcal{H}^0\setminus \mathcal{H}$.
\end{defn}

Similarly, the intersection of $\mathcal{M}_\C^n$ with a complex quadric hypersurface $\mathcal{Q} \neq \mathcal{M}_\C^n$ is a \emph{complex Darboux cyclide}. The central projection $\sigma_\C(\mathcal{Q} \cap \mathcal{M}^{n}_\C)$ is a complex Darboux cyclide in $\overline{\C^n}$.

Definition~\ref{defn: cyclide} is a direct generalisation of the definition of Darboux cyclides in 
\cite{coolidge1916}. $1$-dimensional and $2$-dimensional Darboux cyclides are classical curves and surfaces that were notably studied by John Casey \cite{casey1871cyclics, casey1871cyclides}, Gaston Darboux \cite{darboux1873} and Th{\'e}odore Moutard \cite{Moutard1864a, Moutard1864b}. $1$-dimensional Darboux cyclides are also known as \emph{cyclics}. Recently, $2$-dimensional Darboux cyclides have garnered renewed interest \cite{pottmann2012, lubbes2018kinematic, skopenkov2019}.

Because $1$-dimensional Darboux cyclides in $\overline{\R^2}$ are generically quartic curves with two imaginary singular points that are the circular points at infinity, they are also known as \emph{bicircular quartics} \cite{casey1871cyclics}. A direct generalisation is that Darboux cyclides in $\overline{\R^n}$ are generically \emph{bihyperspherical quartics}. By definition, an algebraic hypersurface $\mathcal{S}$ in $\overline{\R^n}$ is \emph{bihyperspherical} if its complexification $\mathcal{S}_\C$ has a singular locus that contains the $(n-2)$-dimensional quadric 
$$
\mathscr{C}^{n-2}:=\mathcal{M}_{\C}\cap \mathcal{H}_{\C}= \{[x]\in \overline{\C^n} \mid x_1^2+x_2^2+\ldots + x_n^2 =0, x_{0}= x_\infty =0\}
$$ 
that is contained in the hyperplane at infinity of $\overline{\C^n}$. For example, any pair of hyperspheres is a bihyperspherical quartic. Indeed, hyperspheres in $\overline{\R^n}$ can be characterised as quadrics such that their complexifications are quadrics in $\overline{\C^n}$ that contain $\mathscr{C}^{n-2}$.

Complex $2$-dimensional Darboux cyclides in $\overline{\C^3}$ are generically quartic surfaces that contain $\mathscr{C}^1$ as a double conic. More generally, quartic surfaces with a double conic were notably studied by Ernst Kummer \cite{kummer1863}. Quartic surfaces with a double conic are surfaces that can be studied as the projections of the intersection surfaces of pairs of quadric hypersurfaces in $\C \mathrm{P}^4$ \cite{segre1884}.

\begin{defn}
\label{defn: confocalquadrics}
A $1$-parameter family of quadric hypersurfaces in $\overline{\C^n}$ is a family of \emph{confocal quadrics} if they are the projective duals of a pencil of quadrics such that $\mathscr{C}^{n-2}$ is the dual of a degenerate quadric that belongs to the pencil.
\end{defn}

Confocal quadrics in $\R^n$ are quadrics such that their complexifications are confocal quadrics in $\C^n$.

Definition~\ref{defn: confocalquadrics} is due to Darboux, who showed its equivalence to standard definitions of confocal quadrics \cite{darboux1873}.


Let $\mathcal{Q} \cap \mathcal{M}_\C^n$ be a complex Darboux cyclide. The two distinct quadrics $\mathcal{Q}$ and $\mathcal{M}_\C^n$ span a pencil of quadrics in $\mathbb{C}\mathrm{P}^{n+1}$. By definition, the complex Darboux cyclide is \emph{general} if the pencil of quadrics contains $n+2$ degenerate quadrics. This is the maximal number of degenerate quadrics that the pencil can contain.

Theorem~\ref{thm: cyclide} is a direct generalisation of properties of $1$-dimensional and $2$-dimensional Darboux cyclides that can be found in \cite{darboux1873, coolidge1916}.

\begin{thm}
\label{thm: cyclide}
Let $\mathcal{D}:= \mathcal{Q}\cap \mathcal{M}_\C^{n}$ be a general complex Darboux cyclide. There are $n+2$ different ways to generate the Darboux cyclide $\mathcal{D}$ as the envelope of a $(n-1)$-dimensional family of hyperspheres in $\mathcal{M}_\C^n$ that are twice tangent to $\mathcal{D}$. For each generation $i$, the hyperspheres have a common orthogonal hypersphere $\mathcal{S}_i$. The $n+2$ common orthogonal hyperspheres $\mathcal{S}_i$ are pairwise orthogonal. Consider the Darboux cyclide $\sigma_\C(\mathcal{D})$ in $\overline{\C^n}$. For each generation of $\sigma_\C(\mathcal{D})$, the centres of the generating spheres are contained in a quadric hypersurface of $\overline{\C^n}$. These $n+2$ quadric hypersurfaces are confocal.
\end{thm}

\begin{proof}
There are $n+2$ degenerate quadric hypersurfaces in the pencil spanned by $\mathcal{Q}$ and $\mathcal{M}^n_\C$. Let $\mathcal{K}_1,\ldots,\mathcal{K}_{n+2}$ be the degenerate quadrics. Let $A_i$ be the point that is the apex of $\mathcal{K}_i$. Any tangent hyperplane of $\mathcal{K}_i$ intersects $\mathcal{M}^n_\C$ in a hypersphere of $\mathcal{M}_{\C}^n$. The tangent hyperplane is tangent to $\mathcal{K}_i$ along the points of an isotropic line that passes through the apex $A_i$. The isotropic line intersects $\mathcal{M}_\C^n$ in two points. The hypersphere is tangent to $\mathcal{D}$ at the two foregoing points. The cone $\mathcal{K}_i$ has a $(n-1)$-dimensional family of tangent hyperplanes. They intersect $\mathcal{M}^n_\C$ in a $(n-1)$-dimensional family of hyperspheres that are orthogonal to the hypersphere $A_i^{\perp} \cap \mathcal{M}^n_\C$, where $A_i^{\perp}$ is the polar hyperplane of $A_i$ relative to $\mathcal{M}^n_\C$. 

Consider an arbitrary plane that contains two of the points $A_i, \ldots, A_{n+2}$. Generically, the plane intersects $\mathcal{D}:= \mathcal{Q}\cap \mathcal{M}_\C^{n}$ in four points. The four points determine a planar quad such that its two Laplace points are two of the points $A_i$. By Lemma~\ref{lem: planarquadpolarlaplace}, the two Laplace points are conjugate relative to $\mathcal{M}^n_\C$. Therefore, any two of the points in $A_i$ must be conjugate relative to $\mathcal{M}^n_\C$. Equivalently, the spheres $\mathcal{S}_i=A_i^{\perp}\cap \mathcal{M}^n_\C$ are pairwise orthogonal.

Let $\mathcal{K}^{\perp}_i$ be the polar quadric of $\mathcal{K}_i$ relative to $\mathcal{M}^n_\C$. Each $\mathcal{K}^{\perp}_i$ is a $(n-1)$-dimensional non-degenerate quadric that is contained in the hyperplane $A_i^{\perp}$. By Lemma~\ref{lem: stereoproj}, the centres of the hyperspheres in $\overline{\C^n}$ of the $n+2$ different generations of $\sigma_\C(\mathcal{D})$ are exactly the points of the quadrics $\sigma_\C(\mathcal{K}^{\perp}_i)$.

Let $N_\C^\perp$ be the tangent hyperplane of $\mathcal{M}^n_\C$ at the point $N_\C$. The intersection $N_\C^\perp \cap \mathcal{D}$ projects under $\sigma_\C$ to the quadric $\mathscr{C}^{n-2}$. The quadrics $\sigma_\C (\mathcal{K}^{\perp}_i)$ are confocal quadrics because they belong to the dual of a pencil of quadrics such that $\mathscr{C}^{n-2}$ also belongs to the dual pencil. To see this, consider an arbitrary point in $\mathscr{C}^{n-2}$. It can be described as the projection $\sigma_\C(P)$ of some point  $P \in  N_\C^\perp \cap \mathcal{D}$. Let $P^{\perp}$ be the polar hyperplane relative to $\mathcal{M}_\C^n$. The point $P$ is contained in all of the quadrics $\mathcal{K}_i$. Dually, $P^{\perp}$ is tangent to all of the quadrics $\mathcal{K}_i^{\perp}$. The hyperplane $P^{\perp}$ intersects $\overline{\C^n}$ in a $(n-1)$-dimensional projective subspace that is tangent to $\mathscr{C}^{n-2}$ at the point where the line $P \vee N_\C$ intersects $\mathscr{C}^{n-2}$. Each quadric $\mathcal{K}^{\perp}_i$ is tangent to the hyperplane  $P^{\perp}$. So, each projected quadric $\sigma_\C(\mathcal{K}^{\perp}_i)$ is tangent to the foregoing $(n-1)$-dimensional projective subspace that is tangent to $\mathscr{C}^{n-2}$. By Definition~\ref{defn: confocalquadrics}, this means that the $n+2$ quadrics are confocal.
\end{proof}

Figure~\ref{figure: twocyclics} illustrates how cyclics can be generated according to Theorem~\ref{thm: cyclide}.

\begin{figure}[htbp]
\[\includegraphics[width=0.5\textwidth]{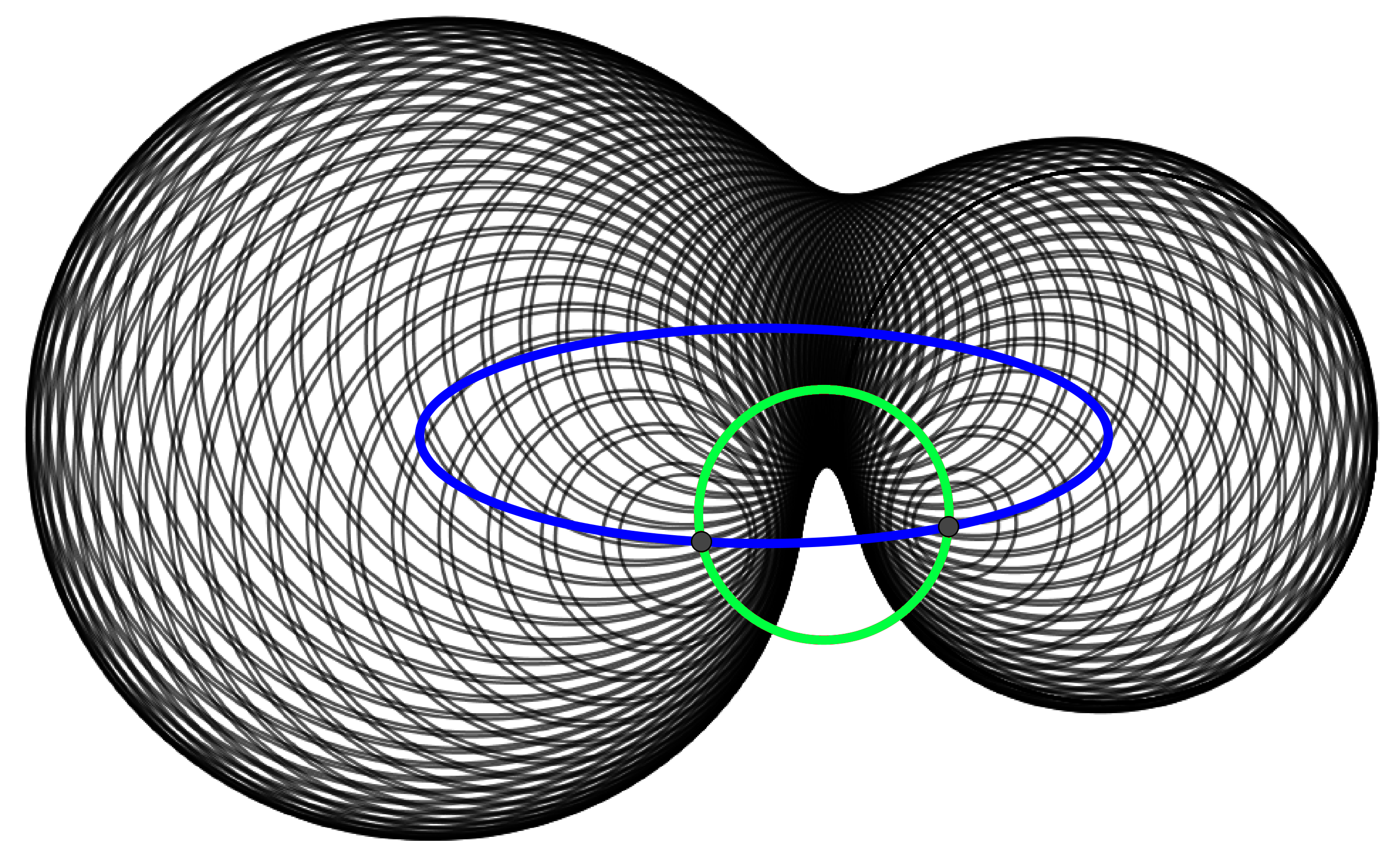}\]
\caption{The envelope of the black circles is a cyclic. The generating circles have a common orthogonal circle and their centres are contained in a conic. Each generating circle is twice tangent to the cyclic. Some of the pairs of tangency points are imaginary. Some of the generating circles are imaginary. Any generating circle with null radius is a focal point of the cyclic.}
  \label{figure: twocyclics}
\end{figure}

\bibliographystyle{acm}
\bibliography{bibfile}

\end{document}